\documentclass{article}
\usepackage{amsmath,bm,colortbl,amssymb,url,amsthm,array,natbib,graphicx}
\DeclareMathOperator*{\argmax}{arg\,max}
\theoremstyle{plain}

\newtheorem{proposition}{Proposition}
\theoremstyle{remark}
\newtheorem*{remark}{Remark} 
\newcommand{\T}{\textsf{T}} 
\allowdisplaybreaks

\begin{document}
\title{On the three-person game \textit{baccara banque}}
\author{S. N. Ethier\thanks{Department of Mathematics, University of Utah, 155 South 1400 East, Salt Lake City, UT 84112, USA. e-mail: ethier@math.utah.edu.  Partially supported by a grant from the Simons Foundation (209632).}\ \ and Jiyeon Lee\thanks{Department of Statistics, Yeungnam University, 214-1 Daedong, Kyeongsan, Kyeongbuk 712-749, South Korea. e-mail: leejy@yu.ac.kr.  Supported by the Basic Science Research Program through the National Research Foundation of Korea (NRF) funded by the Ministry of Science, ICT \& Future Planning (No.~2013R1A1A3A04007670).}}
\date{}
\maketitle

\begin{abstract}
\textit{Baccara banque} is a three-person zero-sum game parameterized by $\theta\in(0,1)$.  A study of the game by Downton and Lockwood claimed that the Nash equilibrium is of only academic interest.  Their preferred alternative is what we call the \textit{independent cooperative equilibrium}.  But this solution exists only for certain $\theta$.  A third solution, which we call the \textit{correlated cooperative equilibrium}, always exists.  Under a ``with replacement'' assumption as well as a simplifying assumption concerning the information available to one of the players, we derive each of the three solutions for all $\theta$..\medskip\par

\end{abstract}

\section{Introduction}\label{Intro}

The three-person game \textit{baccara banque} (or \textit{baccara \`a deux tableaux}) is closely related to the  two-person game \textit{baccara chemin de fer}.  In fact, \textit{baccara banque} has been described as ``a game in which a banker plays \textit{chemin-de-fer} simultaneously against two players'' (Downton and Lockwood 1976).  Game-theoretic analyses of \textit{baccara chemin de fer} have been provided by Kemeny and Snell (1957), Foster (1964a), Downton and Lockwood (1975), Deloche and Oguer (2007), and Ethier and G\'amez (2013).  The more complicated game \textit{baccara banque} has received less attention.  Foster (1964b) was the first to approach the game from the perspective of game theory, though the details of his research were not published.  Kendall and Murchland (1964) used simulation to study the game.  Downton and Holder (1972) discussed the special case of highly unbalanced stakes.  Judah and Ziemba (1983) analyzed a variant of the game in which two of the three players have mandated strategies.  Downton and Lockwood (1976) provided the most detailed study of \textit{baccara banque}, although it is partially incorrect.

To explain the purpose of this paper, we must first describe the rules of \textit{baccara banque}.  There are three players, Player~1, Player~2, and Banker.  Two bets are available to participants, a bet on the hand of Player~1 and a bet on the hand of Player~2.  Six 52-card decks are mixed together and dealt from a \textit{sabot} or shoe.  Denominations A, 2--9, 10, J, Q, K have values 1, 2--9, 0, 0, 0, 0, respectively, and suits are irrelevant.  The total of a hand, comprising two or three cards, is the sum of the values of the cards, modulo 10.  In other words, only the final digit of the sum is used to evaluate a hand.  Two cards are dealt face down to each of Player~1, Player~2, and Banker.  A two-card total of 8 or 9 is a \textit{natural}.  First, if either Banker or both Players have naturals, play ends.  If only one Player has a natural and Banker does not, that Player wins the amount bet from Banker, while play continues between the other Player and Banker.  Next, if neither Player~1 nor Banker has a natural, Player~1 has the option of drawing a third card.  Then, if neither Player~2 nor Banker has a natural, Player~2 has the option of drawing a third card.  In either case, the Player must draw on a two-card total of 4 or less and stand on a two-card total of 6 or 7.  When his two-card total is 5, he is free to draw or stand as he chooses.  Any third card is dealt face up.  Finally, if at least one Player and Banker fail to have naturals, Banker has the option of drawing a third card, and his strategy is unconstrained.  Bets are then settled, both Player~1 vs.\ Banker and Player~2 vs.\ Banker.  In both competitions, the higher total wins.  Winning bets are paid by Banker at even odds.  Losing bets are collected by Banker.  Equal totals result in a \textit{push} (no money changes hands).

There is a subtle point in the rules that is left ambiguous in most descriptions of the game, concerning the information available to Player~2 when he makes his decision.  In what Downton and Lockwood (1976) called the traditional form of the game, Player~2 sees Player~1's third card, if any, or that he has a natural.  The traditional rule is unambiguously stated in Morehead and Mott-Smith (1950, pp.~522--523), for example.  In a more recent variation, Player~2 would know only Player~1's intention to draw or stand, or that he has a natural.  This rule has been used in Great Britain (Downton and Lockwood 1976) and in Monte Carlo (Barnhart 1980, pp.~42--43).

Thus, we have a three-person zero-sum game.  Let us assume, as did Kemeny and Snell (1957), that cards are dealt with replacement and that only two-card totals (not compositions) are seen.  Player~1 has two pure strategies, draw or stand on two-card totals of 5.  Assuming the traditional form of the game, Player~2 also has a draw-or-stand decision on two-card totals of 5 in each of 12 possible situations (Player~1 third-card value 0--9; or stand or natural).  Banker then has a draw-or-stand decision in each of $(12\times 12-1)\times 8=1144$ possible situations (12 possibilities for Player~1, 12 for Player~2, and 8 for Banker, except when both Players have naturals).  Therefore, the game is a $2\times 2^{12}\times 2^{1144}$ trimatrix game, which is zero-sum.  Under the rules of the more recent variation, the game would be $2\times 2^{3}\times 2^{1144}$.

A simplified model assumes that Player~2 ignores his information about Player~1's third card (or stand or natural).  We then have a $2\times 2\times 2^{1144}$ trimatrix game, which is again zero-sum.  

In any case, the payoffs depend on $\theta\in(0,1)$, where the amounts bet on Player~1's and Player~2's hands are in the proportions $\theta:1-\theta$.  To see intuitively why $\theta$ plays an important role, suppose Player~1's third card is 7 and Player~2's third card is 8.  If Banker were playing \textit{baccara chemin de fer} against Player~1, he would draw on 0--6 and stand on 7.  If he were playing \textit{baccara chemin de fer} against Player~2, he would draw on 0--2 and stand on 3--7.  Notice that Banker would act differently against the two Players if his two-card total were 3--6.  In \textit{baccarat banque}, however, he must make the same move (draw or stand) against both Players.  Le Myre (1935, p.~114) called this a ``cruel embarrassment'' for Banker.  The parameter $\theta$ determines his correct choice in these conflicting situations.

Downton and Lockwood (1976) claimed that the Nash equilibrium is of only academic interest because ``it implies an attitude to the game by all three participants, which is unlikely to be realized in practice.''  Presumably, they meant that the two Players regard themselves as competing against Banker but not against each other.  Downton and Lockwood's preferred alternative is what we call the \textit{independent cooperative equilibrium}, that is, the solution of the two-person zero-sum game in which Players~1 and 2, acting independently, form a coalition against Banker.  Actually, the idea goes back to Foster (1964b).  The independence is a requirement of rules not previously stated, which do not permit collaboration between Players~1 and 2.  But, as we will show, this solution does not necessarily exist, in the sense that the lower and upper values of the game may differ.  Nevertheless, the lower value and the Players' maximin strategy are of relevance to the Players.  A third solution, which we call the \textit{correlated cooperative equilibrium}, always exists.  Its value and Banker's minimax strategy are of relevance to Banker.  Here there is no independence constraint.

Downton and Lockwood (1976) assumed the full model (in which Player~2 has $2^{12}$ pure strategies).  They evaluated the Players' behavioral strategies in the independent cooperative equilibrium at $\theta=0.1,0.2,\ldots,0.9$, reporting Banker's behavioral strategy at $\theta=0.3,0.5,0.7$.  They rounded the Players' mixing probabilities to two decimal places, and rounded Banker's mixing probabilities to 0 or 1.  These results are, with minor exceptions, correct.  They also derived the Nash equilibrium at the same level of detail, but these results are incorrect because their algorithm for finding the Nash equilibrium is flawed, as we explain later.

In this paper we focus our attention on the simplified model (in which Player~2 has two pure strategies).  In particular, by symmetry we may assume, without loss of generality, that $\theta\in(0,1/2]$.  In effect, we interpret Player~1 as the Player on whose hand the smaller amount is bet.  We have found, for every $\theta\in(0,1/2]$, the correlated cooperative equilibrium, which is typically unique in behavioral strategies, and the Nash equilibrium, which is often nonunique in behavioral strategies.  As for the independent cooperative equilibrium, we have found the maximin strategy of the Players, which is again typically unique in behavioral strategies; the minimax strategy of Banker coincides with the corresponding strategy in the correlated cooperative equilibrium.  As Downton and Lockwood (1976) put it, the latter Banker strategy ``provides a safety-first strategy which guarantees a return to the bank, whatever strategy the players actually adopt.''  It is for this reason that we regard the correlated cooperative equilibrium as more useful, from Banker's perspective, than the Nash equilibrium.  On the other hand, from the Players' perspective, it is arguable whether the independent cooperative equilibrium is more useful than the Nash equilibrium.  (The Players' strategy in the correlated cooperative equilibrium typically requires collaboration and is therefore illegal.)

To mention a few of our findings, the two cooperative equilibria coincide when $\theta\in(9588/37663, 55716/128711)\approx(0.254573, 0.432877)$, and the Nash equilibrium is nearly the same.  Elsewhere, with two exceptions, they differ.  The correlated cooperative equilibrium is piecewise continuous in $\theta$ but with 109 discontinuities in $(0,1/2)$.  The Nash equilibrium is piecewise continuous in $\theta$ but with 102 discontinuities in $(0,1/2)$.  In both cases, the discontinuities come from  Banker's strategy.  The Players' strategies are continuous in $\theta$ in the Nash equilibrium except for one point of discontinuity.  In the independent cooperative equilibrium, there are as many as 13 points of discontinuity in the Players' strategies.  The game's lower value (to the Players) is continuous on $(0,1/2]$ and increasing on $[0,0.496000]$, approximately, and it is maximized at about $0.496000$.  The maximum value is about $-0.008679984$.  The game's upper value (to the Players) is continuous on $(0,1/2]$ and increasing on $(0,0.496088]$, approximately, and it is maximized at about $0.496088$.  The maximum value is about $-0.008677388$.

The correlated cooperative equilibrium, the independent cooperative equilibrium, and the Nash equilibrium are not easy to describe precisely.  Complete, albeit necessarily lengthy, descriptions are provided in Appendixes A, B, and C. 

There have been a number of attempts to quantify Banker's advantage at \textit{baccara banque} in the case of equal amounts bet on Players~1 and 2 (i.e., $\theta=1/2$).  Le Myre (1935, p.~166) and Boll (1944, pp.~43, 70) made the first estimates (1.11\% and 0.87\%, resp.), assuming that the Players independently draw on 5 with probability 1/2, and Banker makes a best response.  The same assumption was made by Barnhart (1980, p.~81), who obtained 0.84\%.  None of these authors was familiar with game theory (or with computers).  Foster (1964b) and Downton and Lockwood (1976) gave the first game-theoretic estimates (0.87\% and 0.85\%, resp.), which are in fact accurate to two significant digits under their respective assumptions.  Kendall and Murchland (1964) gave a simulated estimate (0.819\%), which is inaccurate owing to small sample size.  Judah and Ziemba (1983) determined Banker's best response when the Players always draw on 5, and obtained 0.81685\% (the correct figure is about 0.922104\%).  Under the simplified model that we are assuming and with $\theta=1/2$, Banker's advantage is about 0.8677394\%, as we will see below.

We conclude this introduction with a historical note.  \textit{Baccara banque} was made famous by the Prince of Wales (later Edward VII) in the Royal Baccarat Scandal of 1891 (Shore 1932).  It became the game of choice for wealthy gamblers in 1922 when Nicolas Zographos, a founding member of the Greek Syndicate, announced ``Tout va'' or unlimited stakes.  As he put it (Graves 1963, pp.~27--28), 
\begin{quote}
My idea is so sensational that practically nobody will play chemin-de-fer.  If I guarantee to take any stake of any size, all the millionaires will want to take part in this fantastic party.  The biggest gamblers in the world will come to ruin me.  I suggest we start at Deauville.
\end{quote} 
In recent years, both \textit{baccara banque} and \textit{baccara chemin de fer} have been largely superseded by a nonstrategic form of the game.  Nevertheless, \textit{baccara banque} is still offered at the Salons Priv\'es of the Casino de Monte-Carlo, Thurs.--Sun.\ from 5 p.m.

\section{Evaluation of the payoffs}\label{payoff}

We consider the simplified game, a $2\times2\times2^{1144}$ trimatrix game parameterized by $\theta\in(0,1/2]$.  Here $\theta$ can be interpreted as the proportion of the total amount bet that is bet on Player~1.  Both Players and Banker are assumed to know $\theta$.  The distribution of the total of a two-card hand is 
$$
q(i):=\frac{16+9\delta_{i,0}}{(13)^2},\qquad i=0,1,\ldots,9,
$$
where $\delta_{i,j}$ is the Kronecker delta, and the distribution of the value of a card is 
$$
q'(k):=\frac{1+3\delta_{k,0}}{13},\qquad k=0,1,\ldots,9.
$$
Let $M:\{0,1,\ldots\}\mapsto\{0,1,\ldots,9\}$ be the function $M(i):=\text{Mod}(i,10)$, the remainder when $i$ is divided by 10.  We denote Player~1's pure strategies by $0$ (stand on 5) and $1$ (draw on 5), and similarly for Player~2's pure strategies.  Banker's pure strategies are identified with subsets $T\subset[\{0,1,\ldots,11\}\times\{0,1,\ldots,11\}-\{(11,11)\}]\times\{0,1,\ldots,7\}$, with $T$ indicating the set of triples $(k_1,k_2,j)$ on which Banker draws.  Here $k_1$ is Player 1's third-card value, $k_2$ is Player 2's third-card value, and $j$ is Banker's two-card total.  We let $T^c$ denote the complement of $T$ with respect to this product set containing $(12\times12-1)\times 8=1144$ triples.  Here Player third-card values 10 and 11 are code for ``stand'' and ``natural'', respectively.

We define the $2\times2\times2^{1144}$ three-dimensional array $\bm a(\theta)$ to have $(u_1,u_2,T)$ entry, for $u_1\in\{0,1\}$, $u_2\in\{0,1\}$, and $T\subset[\{0,1,\ldots,11\}\times\{0,1,\ldots,11\}-\{(11,11)\}]\times\{0,1,\ldots,7\}$, equal to
\begin{eqnarray}
&&\!\!\!\!\!\!\!\!a_{u_1,u_2,T}(\theta)\nonumber\\
&&\!\!\!\!\!{}=\bigg(\sum_{i_1=0}^9\sum_{i_2=0}^9\sum_{j=8}^9+\sum_{i_1=8}^9\sum_{i_2=8}^9\sum_{j=0}^7\bigg) q(i_1)q(i_2)q(j)\nonumber\\
&&\qquad{}[\theta\,\text{sgn}(i_1-j)+(1-\theta)\,\text{sgn}(i_2-j)]\\
&&{}+\sum_{i_1=8}^9\sum_{i_2=0}^{4+u_2}\sum_{j=0}^7\sum_{k_2=0}^9\sum_{l=0}^9 1_T((11,k_2,j))q(i_1)q(i_2)q(j)q'(k_2)q'(l)\nonumber\\
&&\qquad{}[\theta+(1-\theta)\,\text{sgn}(M(i_2+k_2)-M(j+l))]\\
&&{}+\sum_{i_1=8}^9\sum_{i_2=0}^{4+u_2}\sum_{j=0}^7\sum_{k_2=0}^9 1_{T^c}((11,k_2,j))q(i_1)q(i_2)q(j)q'(k_2)\nonumber\\
&&\qquad{}[\theta+(1-\theta)\,\text{sgn}(M(i_2+k_2)-j)]\\
&&{}+\sum_{i_1=8}^9\sum_{i_2=5+u_2}^7\sum_{j=0}^7\sum_{l=0}^9 1_T((11,10,j))q(i_1)q(i_2)q(j)q'(l)\nonumber\\
&&\qquad{}[\theta+(1-\theta)\,\text{sgn}(i_2-M(j+l))]\\
&&{}+\sum_{i_1=8}^9\sum_{i_2=5+u_2}^7\sum_{j=0}^7 1_{T^c}((11,10,j))q(i_1)q(i_2)q(j)\nonumber\\
&&\qquad{}[\theta+(1-\theta)\,\text{sgn}(i_2-j)]\\
&&{}+\sum_{i_1=0}^{4+u_1}\sum_{i_2=8}^9\sum_{j=0}^7\sum_{k_1=0}^9\sum_{l=0}^9 1_T((k_1,11,j))q(i_1)q(i_2)q(j)q'(k_1)q'(l)\nonumber\\
&&\qquad{}[\theta\,\text{sgn}(M(i_1+k_1)-M(j+l))+(1-\theta)]\\
&&{}+\sum_{i_1=0}^{4+u_1}\sum_{i_2=8}^9\sum_{j=0}^7\sum_{k_1=0}^9 1_{T^c}((k_1,11,j))q(i_1)q(i_2)q(j)q'(k_1)\nonumber\\
&&\qquad{}[\theta\,\text{sgn}(M(i_1+k_1)-j)+(1-\theta)]\\
&&{}+\sum_{i_1=5+u_1}^7\sum_{i_2=8}^9\sum_{j=0}^7\sum_{l=0}^9 1_T((10,11,j))q(i_1)q(i_2)q(j)q'(l)\nonumber\\
&&\qquad{}[\theta\,\text{sgn}(i_1-M(j+l))+(1-\theta)]\\
&&{}+\sum_{i_1=5+u_1}^7\sum_{i_2=8}^9\sum_{j=0}^7 1_{T^c}((10,11,j))q(i_1)q(i_2)q(j)\nonumber\\
&&\qquad{}[\theta\,\text{sgn}(i_1-j)+(1-\theta)]\\
&&{}+\sum_{i_1=0}^{4+u_1}\sum_{i_2=0}^{4+u_2}\sum_{j=0}^7\sum_{k_1=0}^9\sum_{k_2=0}^9\sum_{l=0}^9 1_T((k_1,k_2,j))q(i_1)q(i_2)q(j)q'(k_1)q'(k_2)q'(l)\nonumber\\
&&\qquad{}[\theta\,\text{sgn}(M(i_1+k_1)-M(j+l))+(1-\theta)\,\text{sgn}(M(i_2+k_2)-M(j+l))]\nonumber\\
&&\\
&&{}+\sum_{i_1=0}^{4+u_1}\sum_{i_2=0}^{4+u_2}\sum_{j=0}^7\sum_{k_1=0}^9\sum_{k_2=0}^9 1_{T^c}((k_1,k_2,j))q(i_1)q(i_2)q(j)q'(k_1)q'(k_2)\nonumber\\
&&\qquad{}[\theta\,\text{sgn}(M(i_1+k_1)-j)+(1-\theta)\,\text{sgn}(M(i_2+k_2)-j)]\\
&&{}+\sum_{i_1=0}^{4+u_1}\sum_{i_2=5+u_2}^7\sum_{j=0}^7\sum_{k_1=0}^9\sum_{l=0}^9 1_T((k_1,10,j))q(i_1)q(i_2)q(j)q'(k_1)q'(l)\nonumber\\
&&\qquad{}[\theta\,\text{sgn}(M(i_1+k_1)-M(j+l))+(1-\theta)\,\text{sgn}(i_2-M(j+l))]\\
&&{}+\sum_{i_1=0}^{4+u_1}\sum_{i_2=5+u_2}^7\sum_{j=0}^7\sum_{k_1=0}^9 1_{T^c}((k_1,10,j))q(i_1)q(i_2)q(j)q'(k_1)\nonumber\\
&&\qquad{}[\theta\,\text{sgn}(M(i_1+k_1)-j)+(1-\theta)\,\text{sgn}(i_2-j)]\\
&&{}+\sum_{i_1=5+u_1}^7\sum_{i_2=0}^{4+u_2}\sum_{j=0}^7\sum_{k_2=0}^9\sum_{l=0}^9 1_T((10,k_2,j))q(i_1)q(i_2)q(j)q'(k_2)q'(l)\nonumber\\
&&\qquad{}[\theta\,\text{sgn}(i_1-M(j+l))+(1-\theta)\,\text{sgn}(M(i_2+k_2)-M(j+l))]\\
&&{}+\sum_{i_1=5+u_1}^7\sum_{i_2=0}^{4+u_2}\sum_{j=0}^7\sum_{k_2=0}^9 1_{T^c}((10,k_2,j))q(i_1)q(i_2)q(j)q'(k_2)\nonumber\\
&&\qquad{}[\theta\,\text{sgn}(i_1-j)+(1-\theta)\,\text{sgn}(M(i_2+k_2)-j)]\\
&&{}+\sum_{i_1=5+u_1}^7\sum_{i_2=5+u_2}^7\sum_{j=0}^7\sum_{l=0}^91_T((10,10,j))q(i_1)q(i_2)q(j)q'(l)\nonumber\\
&&\qquad{}[\theta\,\text{sgn}(i_1-M(j+l))+(1-\theta)\,\text{sgn}(i_2-M(j+l))]\\
&&{}+\sum_{i_1=5+u_1}^7\sum_{i_2=5+u_2}^7\sum_{j=0}^7 1_{T^c}((10,10,j))q(i_1)q(i_2)q(j)\nonumber\\
&&\qquad{}[\theta\,\text{sgn}(i_1-j)+(1-\theta)\,\text{sgn}(i_2-j)].
\end{eqnarray}
Term 1 corresponds to the case in which Banker and/or both Players have naturals.  Terms 2--5 (resp., 6--9) correspond to the case in which only Player~1 (resp., only Player~2) has a natural.  Terms 10--17 correspond to the case in which there are no naturals.

Notice that, in the three-person zero-sum game, $\theta a_{u_1,u_2,T}(1)$ is the payoff to Player~1, $(1-\theta)a_{u_1,u_2,T}(0)$ is the payoff to Player~2, and $$-[\theta a_{u_1,u_2,T}(1)+(1-\theta)a_{u_1,u_2,T}(0)]=-a_{u_1,u_2,T}(\theta)$$ is the payoff to Banker, all measured in units of total amount bet.

\section{Correlated cooperative equilibrium}\label{corr}

We first find the correlated cooperative equilibrium, that is, the solution of the two-person zero-sum game in which the two Players form a coalition against Banker and are not constrained to act independently.  This is a $2^2\times2^{1144}$ matrix game with payoff matrix having entries $a_{u_1,u_2,T}(\theta)$ as defined in Section \ref{payoff}.  For fixed $\theta\in(0,1/2]$, we can obtain a solution as follows:  Given an arbitrary mixture $\bm p=(p_{00},p_{01},p_{10},p_{11})$ of the four pure strategies of the Players, minimize
$$
p_{00}\,a_{0,0,T}(\theta)+p_{01}\,a_{0,1,T}(\theta)+p_{10}\,a_{1,0,T}(\theta)+p_{11}\,a_{1,1,T}(\theta)
$$
as a function of $T$ (this is Banker's best response $T=T_\theta(\bm p)$), and then maximize
$$
E_\theta(\bm p):=p_{00}\,a_{0,0,T_\theta(\bm p)}(\theta)+p_{01}\,a_{0,1,T_\theta(\bm p)}(\theta)+p_{10}\,a_{1,0,T_\theta(\bm p)}(\theta)+p_{11}\,a_{1,1,T_\theta(\bm p)}(\theta)
$$
as a function of $\bm p$.  The maximizing $\bm p$ is the Players' maximin strategy, and the maximal value of $E_\theta(\bm p)$ is the value of the game, assuming $\theta$ is fixed. 

For a given Banker information set $(k_1,k_2,j)$ ($k_1=$ Player~1's third-card value, $k_2=$ Player~2's third-card value, $j=$ Banker's two-card total), Banker's optimal move (draw or stand) may or may not depend on the Players' mixed strategy $\bm p$.  In fact, for only $m$ of the 1144 information sets, where $52\le m\le 69$, is there dependence on $\bm p$.  Writing $p_{11}=1-p_{00}-p_{01}-p_{10}$, it follows that
\begin{eqnarray*}
E_\theta(\bm p)&=&a_0+b_0 p_{00}+c_0 p_{01}+d_0 p_{10}\\
&&\;{}+\sum_{i=1}^m\min(a_i+b_i p_{00}+c_i p_{01}+d_i p_{10},\,a_i'+b_i' p_{00}+c_i' p_{01}+d_i' p_{10}),
\end{eqnarray*}
where the constants $a_i,b_i,c_i,d_i,a_i',b_i',c_i',d_i'$ are computable rational numbers (if $\theta$ is rational) depending on $\theta$.  

For fixed $\theta$, the function $E_\theta(\bm p)$ is concave in $\bm p$ and its maximum occurs at an intersection of three of the $m+4$ planes
\begin{equation}\label{plane}
a_i+b_i p_{00}+c_i p_{01}+d_i p_{10}=a_i'+b_i' p_{00}+c_i' p_{01}+d_i' p_{10}, \quad 1\le i\le m,
\end{equation}
and $p_{00}=0$, $p_{01}=0$, $p_{10}=0$, and $1-p_{00}-p_{01}-p_{10}=0$.  The $m$ planes in \eqref{plane} might be called ``indifference planes''.  This leads to a simple algorithm to find the optimal $\bm p$.  For each of the $\binom{m+4}{3}$ potential points $\bm p$ just mentioned, check whether $p_{00}\ge0$, $p_{01}\ge0$, $p_{10}\ge0$, and $1-p_{00}-p_{01}-p_{10}\ge0$, and if so, evaluate $E_\theta(\bm p)$.  Then determine at which such $\bm p$ the value $E_\theta(\bm p)$ is largest and if it is uniquely so.  

In the case $\theta=1/2$, the number of summands is $m=68$ and there is a unique maximum at
$$
\bm p=(p_{00},p_{01},p_{10},p_{11})=\bigg(0,\frac{110}{543},\frac{110}{543},\frac{323}{543}\bigg);
$$
in addition, $E_{1/2}(\bm p)=-16655514960/[181(13)^9]\approx-0.008677394$ there.  The point $\bm p$ is the intersection of three planes, the two indifference planes for $(6,10,6)$ (Player~1's third card is 6, Player~2 stands, Banker's two-card total is 6) and $(10,6,6)$ and the plane $p_{00}=0$.  Banker's best response is displayed in Table \ref{best, theta=1/2}.  

\begin{table}[htb]
\caption{\label{best, theta=1/2}Banker's strategy in the correlated cooperative equilibrium when $\theta=1/2$. Specifically the table displays Banker's maximum drawing total as a function of Player~1's and Player~2's third-card values.  For example, if Player~1's third card is 7 and Player~2's third card is 8, the entry 3 signifies that Banker draws on 0--3 and stands on 4--7.  5+ signifies that Banker draws on 0--5, mixes on 6, and stands on 7.  The table is symmetric in Player~1 and Player~2.  Similar entries are shaded similarly for readability.}
\tabcolsep=1.8mm
\catcode`@=\active\def@{\phantom{+}}
\begin{center}
\begin{tabular}{ccccccccccccc}
\hline
\noalign{\smallskip}
Player~1's &\multicolumn{12}{c}{Player~2's third-card value (10 = stand, 11 = natural)}\\
third-card &\\
value & 0@ & 1@ & 2@ & 3@ & 4@ & 5@ & 6@ & 7@ & 8@ & 9@ & \!\!10 & \!\!11 \\
\noalign{\smallskip} \hline
\noalign{\smallskip}
0 & \cellcolor[gray]{0.9}3@ & \cellcolor[gray]{0.9}3@ & \cellcolor[gray]{0.9}3@ & \cellcolor[gray]{0.8}4@ & \cellcolor[gray]{0.8}4@ & \cellcolor[gray]{0.8}4@ & \cellcolor[gray]{0.8}4@ & \cellcolor[gray]{0.9}3@ & \cellcolor[gray]{0.9}3@ & \cellcolor[gray]{0.9}3@ & \cellcolor[gray]{0.7}5@ & \cellcolor[gray]{0.9}3@ \\
1 & \cellcolor[gray]{0.9}3@ & \cellcolor[gray]{0.9}3@ & \cellcolor[gray]{0.8}4@ & \cellcolor[gray]{0.8}4@ & \cellcolor[gray]{0.8}4@ & \cellcolor[gray]{0.8}4@ & \cellcolor[gray]{0.8}4@ & \cellcolor[gray]{0.8}4@ & \cellcolor[gray]{0.9}3@ & \cellcolor[gray]{0.9}3@ & \cellcolor[gray]{0.7}5@ & \cellcolor[gray]{0.9}3@ \\
2 & \cellcolor[gray]{0.9}3@ & \cellcolor[gray]{0.8}4@ & \cellcolor[gray]{0.8}4@ & \cellcolor[gray]{0.8}4@ & \cellcolor[gray]{0.8}4@ & \cellcolor[gray]{0.8}4@ & \cellcolor[gray]{0.7}5@ & \cellcolor[gray]{0.8}4@ & \cellcolor[gray]{0.9}3@ & \cellcolor[gray]{0.9}3@ & \cellcolor[gray]{0.7}5@ & \cellcolor[gray]{0.8}4@ \\
3 & \cellcolor[gray]{0.8}4@ & \cellcolor[gray]{0.8}4@ & \cellcolor[gray]{0.8}4@ & \cellcolor[gray]{0.8}4@ & \cellcolor[gray]{0.8}4@ & \cellcolor[gray]{0.7}5@ & \cellcolor[gray]{0.7}5@ & \cellcolor[gray]{0.8}4@ & \cellcolor[gray]{0.8}4@ & \cellcolor[gray]{0.9}3@ & \cellcolor[gray]{0.7}5@ & \cellcolor[gray]{0.8}4@ \\
4 & \cellcolor[gray]{0.8}4@ & \cellcolor[gray]{0.8}4@ & \cellcolor[gray]{0.8}4@ & \cellcolor[gray]{0.8}4@ & \cellcolor[gray]{0.7}5@ & \cellcolor[gray]{0.7}5@ & \cellcolor[gray]{0.7}5@ & \cellcolor[gray]{0.7}5@ & \cellcolor[gray]{0.8}4@ & \cellcolor[gray]{0.8}4@ & \cellcolor[gray]{0.7}5@ & \cellcolor[gray]{0.7}5@ \\
5 & \cellcolor[gray]{0.8}4@ & \cellcolor[gray]{0.8}4@ & \cellcolor[gray]{0.8}4@ & \cellcolor[gray]{0.7}5@ & \cellcolor[gray]{0.7}5@ & \cellcolor[gray]{0.7}5@ & \cellcolor[gray]{0.7}5@ & \cellcolor[gray]{0.7}5@ & \cellcolor[gray]{0.7}5@ & \cellcolor[gray]{0.8}4@ & \cellcolor[gray]{0.7}5@ & \cellcolor[gray]{0.7}5@ \\
6 & \cellcolor[gray]{0.8}4@ & \cellcolor[gray]{0.8}4@ & \cellcolor[gray]{0.7}5@ & \cellcolor[gray]{0.7}5@ & \cellcolor[gray]{0.7}5@ & \cellcolor[gray]{0.7}5@ & \cellcolor[gray]{0.6}6@ & \cellcolor[gray]{0.6}6@ & \cellcolor[gray]{0.7}5@ & \cellcolor[gray]{0.8}4@ & \cellcolor[gray]{0.7}5+ & \cellcolor[gray]{0.6}6@ \\
7 & \cellcolor[gray]{0.9}3@ & \cellcolor[gray]{0.8}4@ & \cellcolor[gray]{0.8}4@ & \cellcolor[gray]{0.8}4@ & \cellcolor[gray]{0.7}5@ & \cellcolor[gray]{0.7}5@ & \cellcolor[gray]{0.6}6@ & \cellcolor[gray]{0.6}6@ & \cellcolor[gray]{0.9}3@ & \cellcolor[gray]{0.9}3@ & \cellcolor[gray]{0.6}6@ & \cellcolor[gray]{0.6}6@ \\
8 & \cellcolor[gray]{0.9}3@ & \cellcolor[gray]{0.9}3@ & \cellcolor[gray]{0.9}3@ & \cellcolor[gray]{0.8}4@ & \cellcolor[gray]{0.8}4@ & \cellcolor[gray]{0.7}5@ & \cellcolor[gray]{0.7}5@ & \cellcolor[gray]{0.9}3@ & 2@ & \cellcolor[gray]{0.9}3@ & \cellcolor[gray]{0.7}5@ & 2@ \\
9 & \cellcolor[gray]{0.9}3@ & \cellcolor[gray]{0.9}3@ & \cellcolor[gray]{0.9}3@ & \cellcolor[gray]{0.9}3@ & \cellcolor[gray]{0.8}4@ & \cellcolor[gray]{0.8}4@ & \cellcolor[gray]{0.8}4@ & \cellcolor[gray]{0.9}3@ & \cellcolor[gray]{0.9}3@ & \cellcolor[gray]{0.9}3@ & \cellcolor[gray]{0.7}5@ & \cellcolor[gray]{0.9}3@ \\
10 &\cellcolor[gray]{0.7}5@ & \cellcolor[gray]{0.7}5@ & \cellcolor[gray]{0.7}5@ & \cellcolor[gray]{0.7}5@ & \cellcolor[gray]{0.7}5@ & \cellcolor[gray]{0.7}5@ & \cellcolor[gray]{0.7}5+ & \cellcolor[gray]{0.6}6@ & \cellcolor[gray]{0.7}5@ & \cellcolor[gray]{0.7}5@ & \cellcolor[gray]{0.7}5@ & \cellcolor[gray]{0.7}5@ \\
11 &\cellcolor[gray]{0.9}3@ & \cellcolor[gray]{0.9}3@ & \cellcolor[gray]{0.8}4@ & \cellcolor[gray]{0.8}4@ & \cellcolor[gray]{0.7}5@ & \cellcolor[gray]{0.7}5@ & \cellcolor[gray]{0.6}6@ & \cellcolor[gray]{0.6}6@ & 2@ & \cellcolor[gray]{0.9}3@ & \cellcolor[gray]{0.7}5@ &   \\
\noalign{\smallskip}
\hline
\end{tabular}
\end{center}
\end{table}

There are two information sets, $(6,10,6)$  and $(10,6,6)$, at which Banker is indifferent.  This leads to a $4\times4$ matrix game with payoff matrix 
\begin{equation}\label{4x4}
\arraycolsep=2mm
\bm A:=-\frac{8}{(13)^9}\begin{pmatrix}
11815316 & 11681780 & 11681780 & 11548244\\ 
11621229 & 11427789 & 11680301 & 11486861\\
11621229 & 11680301 & 11427789 & 11486861\\
11421510 & 11467270 & 11467270 & 11513030\end{pmatrix}.
\end{equation}
Rows correspond to pure strategies of the Players, SS, SD, DS, and DD on 5 by Player~1 and Player~2.  Columns correspond to pure strategies of Banker, which follow Table \ref{best, theta=1/2} except for SS, SD, DS, DD on $(6,10,6)$ and $(10,6,6)$.  This game has two extreme equilibria $(\bm p, \bm q)$, where $\bm p$ is as above, and
$$
\bm q=(q_{00},q_{01},q_{10},q_{11})=\bigg(0,\frac{671}{5792},\frac{671}{5792},\frac{2225}{2896}\bigg)
$$
or
$$
\bm q=(q_{00},q_{01},q_{10},q_{11})=\bigg(\frac{671}{5792},0,0,\frac{5121}{5792}\bigg).
$$
The corresponding Banker behavioral strategies are the same for both equilibria,
\begin{eqnarray*}
P(\text{Banker draws on }(6,10,6))&=&q_{10}+q_{11}=5121/5792,\\
P(\text{Banker draws on }(10,6,6))&=&q_{01}+q_{11}=5121/5792.
\end{eqnarray*}
This completes the derivation in the case $\theta=1/2$.

Next, we extend this solution to the largest $\theta$-interval in $(0,1/2]$ to which Table \ref{best, theta=1/2} applies.  The maximum of $E_{1/2}(\bm p)$ found above occurred at the intersection of three planes, the two indifference planes for $(6,10,6)$ and $(10,6,6)$ and the plane $p_{00}=0$.  The intersection of the three corresponding $\theta$-dependent planes occurs at
\begin{eqnarray*}
\bm p(\theta):&=&\bigg(0,\frac{10\,(-3171 - 17332\,\theta + 20640\,\theta^2)}{3\,[901 - 443072\,\theta(1-\theta)]}, \frac{10\,(137 - 23948\,\theta + 20640\,\theta^2)}{3\,[901 - 443072\,\theta(1-\theta)]},\nonumber\\
 &&\qquad\qquad\qquad\qquad\qquad\quad\frac{33043 - 916416\,\theta(1-\theta)}{3\,[901 - 443072\,\theta(1-\theta)]}\bigg).
\end{eqnarray*}
With this choice of $\bm p(\theta)$ we can ask, what is the smallest $\theta$ for which Banker's best response is given by Table \ref{best, theta=1/2}?  Checking each of the 1144 Banker information sets, we find that the first change occurs at $(2,5,5)$.  The contribution to the difference between the Players' expectation when Banker draws and the Players' expectation when Banker stands (due to $(2,5,5)$) vanishes at $\theta_*\approx0.496088$; more precisely, $\theta_*$ is a root of the cubic polynomial $6896169 - 1190915420\,\theta + 3549548480\,\theta^2 - 2372477184\,\theta^3$.  

On the interval $(\theta_*,1/2]$ Banker mixes at $(6,10,6)$ and $(10,6,6)$, and it remains to determine the mixing probabilities.  With $\bm A(\theta)$ denoting the $\theta$-dependent version of \eqref{4x4}, the value $v(\theta)$ of the game satisfies
$$
\bm p(\theta)\bm A(\theta)=(v(\theta),v(\theta),v(\theta),v(\theta)).
$$
We find that
\begin{equation}\label{v(theta)}
v(\theta) = -\frac{80\,[2421541645 - 515181045616\,\theta(1-\theta)]}{(13)^9\,[901 - 443072\,\theta(1-\theta)]}.
\end{equation}
Since the Players have three strategies active, we seek a $3\times3$ kernel, and two of the four possibilities give nonnegative Banker mixing probabilities, the ones corresponding to $q_{00}(\theta)=0$ and to $q_{01}(\theta)=0$.  The resulting two solutions of
$$
\bm A(\theta)\bm q(\theta)^{\T}=(x(\theta),v(\theta),v(\theta),v(\theta))^{\T},
$$
where $x(\theta)\le v(\theta)$, give the same Banker behavioral strategies, namely,
\begin{eqnarray*}
P(\text{Banker draws on }(6,10,6))&=&q_{10}(\theta)+q_{11}(\theta)\\
&=&\frac{21311777 - 393439433\,\theta + 620812136\,\theta^2}{208\,[901 - 443072\,\theta(1-\theta)]},\\
P(\text{Banker draws on }(10,6,6))&=&q_{01}(\theta)+q_{11}(\theta)\\
&=&\frac{248684480 - 848184839\,\theta + 620812136\,\theta^2}{208\,[901 - 443072\,\theta(1-\theta)]}.
\end{eqnarray*}
This completes the derivation for the interval $(\theta_*,1/2]$.  

Repeating this process (from right to left, or from left to right), we find that there are 110 such intervals in $(0,1/2]$.  That is, there exist $0=\theta_0<\theta_1<\theta_2<\cdots<\theta_{109}=\theta_*<\theta_{110}=1/2$ such that the correlated cooperative equilibrium is a rational function of $\theta$ on interval $i$, namely $(\theta_{i-1},\theta_i)$, for $i=1,2,\ldots,110$.  Each $\theta_i$ is a root of a polynomial of degree 4 or less.  At the boundary points, discontinuities occur in Banker's strategy.  

There are two types of intervals, those in which the number of Banker information sets at which Banker mixes is two and those in which it is three.  When it is two, the resulting $4\times4$ game has a $3\times3$ kernel.  When it is three, the resulting $4\times8$ game has a $4\times4$ kernel.  For intervals 1--41, $p_{11}(\theta)=0$;  for intervals 42--46, 102--103, and 107--110, $p_{00}(\theta)=0$; for intervals 61--66, $p_{10}(\theta)=0$; and for all remaining intervals the Players have all strategies active.  In all cases, despite the correlated cooperative equilibrium being nonunique in mixed strategies, it is unique in behavioral strategies.  (This can be proved algebraically.)  However there are exceptions.  At each boundary point, the solutions from both adjacent intervals apply, so there is nonuniqueness of Banker behavioral strategies at the 109 such $\theta$.

The value function is continuous on $(0,1/2]$, increasing on $(0,\theta_*]$ and decreasing on $[\theta_*,1/2]$ (see \eqref{v(theta)}).  Its maximum value is $v(\theta_*)\approx-0.008677388$.  See Figure \ref{graph-values} for a sketch of the graph.

\begin{figure}[htb]
\centering
\includegraphics[width=4.5in]{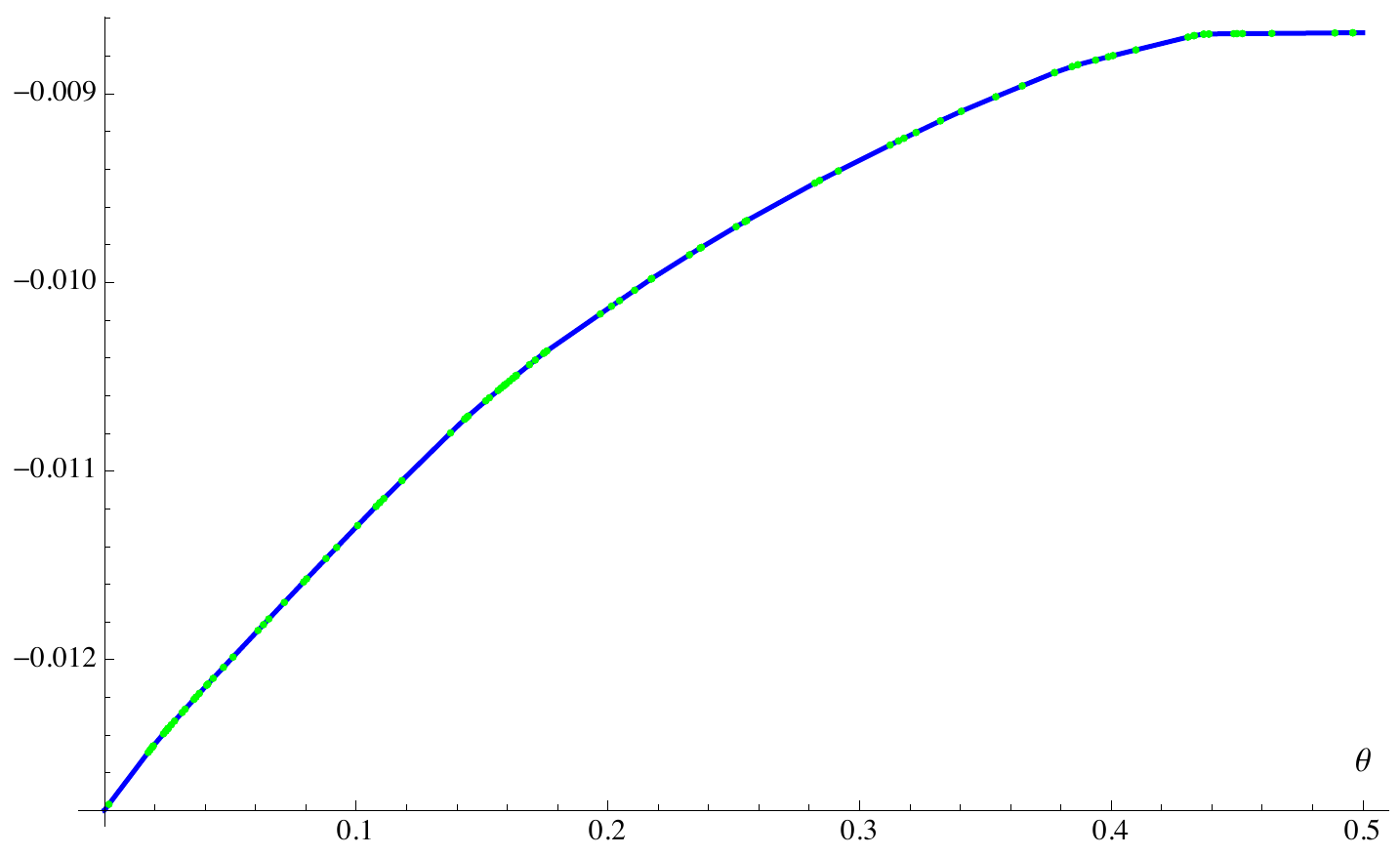}
\caption{\label{graph-values}The graph of the value of the game to the Players (or minus the value to Banker), assuming the correlated cooperative equilibrium.}
\end{figure}

Of particular interest are intervals 75--97, in which the Players' strategy at equilibrium is $p_{00}(\theta)=4/121$, $p_{01}(\theta)=p_{10}(\theta)=18/121$, and $p_{11}(\theta)=81/121$, that is, Players~1 and 2 play independently, drawing on 5 with probabilities $p_1(\theta)=p_2(\theta)=9/11$.  In this case, the correlated and independent cooperative equilibria coincide, as we demonstrate in Proposition \ref{prop} below.

The complete details of the correlated cooperative equilibrium are summarized in Appendix A.

\section{Independent cooperative equilibrium}\label{ice}

Let $X_1$ and $X_2$ be Player~1's and Player~2's two-card totals, and let $X_1'$ and $X_2'$ be their third-card values.  Let $Y$ be Banker's two-card total, and let $Y'$ be Banker's third-card value.  Let $p_1$ be the probability that Player~1 draws on 5, and let $p_2$ be the probability that Player~2 draws on 5.  Let $U_1$ and $U_2$ be the mixed strategies of Players~1 and 2, that is, random variables with distributions $P(U_1=1)=p_1=1-P(U_1=0)$ and $P(U_2=1)=p_2=1-P(U_2=0)$.  Assume they are independent of $X_1,X_2,Y,X_1',X_2',Y'$ but not necessarily of each other.

Then Equation~(2.1b) of Downton and Lockwood (1976), which represents the conditional expected gain to Banker when he draws (measured in units of total amount bet), given that Player~1's third-card value is $k_1\in\{0,1,\ldots,9\}$, Player~2's third-card value is $k_2\in\{0,1,\ldots,9\}$, and Banker's two-card total is $j\in\{0,1,\ldots,7\}$, can be written 
\begin{eqnarray*}
\!\!\!&&a(k_1,k_2,j;p_1,p_2)\\
\!\!\!&&{}=\theta E[{\rm sgn}(M(Y+Y')-M(X_1+X_1'))\mid X_1\le4+U_1,\, X_1'=k_1,\,Y=j]\nonumber\\
\!\!\!&&\;\;{}+(1-\theta)E[{\rm sgn}(M(Y+Y')-M(X_2+X_2'))\mid X_2\le4+U_2,\, X_2'=k_2,\,Y=j].\nonumber
\end{eqnarray*}
Generally, one does not add two conditional expectations when they are conditioned on different events.  However, if $U_1$ and $U_2$ are independent, then 
\begin{eqnarray*}
&&\!\!\!\!\!a(k_1,k_2,j;p_1,p_2)\\
&&{}=\theta E[{\rm sgn}(M(Y+Y')-M(X_1+X_1'))\mid X_1\le4+U_1,\, X_1'=k_1,\,Y=j,\\
&&\qquad\qquad\qquad\qquad\qquad\qquad\qquad\qquad\qquad{} X_2\le4+U_2,\, X_2'=k_2]\\
&&\quad{}+(1-\theta)E[{\rm sgn}(M(Y+Y')-M(X_2+X_2'))\mid X_2\le4+U_2,\, X_2'=k_2,\\
&&\qquad\qquad\qquad\qquad\qquad\qquad\qquad\qquad\qquad{} Y=j,\,X_1\le4+U_1,\, X_1'=k_1]\\
&&{}=E[\theta\,{\rm sgn}(M(Y+Y')-M(X_1+X_1'))\\
&&\qquad{}+(1-\theta){\rm sgn}(M(Y+Y')-M(X_2+X_2'))\mid\\
&&\qquad\qquad{}X_1\le4+U_1,\, X_2\le4+U_2,\, X_1'=k_1,\, X_2'=k_2,\, Y=j],
\end{eqnarray*}
which is evidently what was intended.  Here we have used the simple fact that
$$
E[X\mid A]=E[X\mid A\cap B]\quad \text{if }1_B\text{ is independent of }1_A\text{ and }X.
$$
The point is that Downton and Lockwood (1976) effectively assumed that Player~1 and Player~2 act independently, even though they make no such assumption explicitly.  Thus, their ``co-operative optimum strategy'' is what we call the independent cooperative equilibrium.

Let us find the independent cooperative equilibrium, that is, the solution of the two-person zero-sum game in which the two Players form a coalition against Banker but are constrained to act independently.  This is a $2^2\times2^{1144}$ matrix game with payoff matrix having entries $a_{u_1,u_2,T}(\theta)$ as defined in Section \ref{payoff}, but mixtures of the four pure strategies of the Players must have the form $((1-p_1)(1-p_2),(1-p_1)p_2,p_1(1-p_2),p_1p_2)$ for some $p_1,p_2\in[0,1]$.  The following proposition shows that the solution need not exist, in the sense that the lower and upper values of the game may differ.  (For a closely related result, see Maschler, Solan, and Zamir 2013, p.~179.)  Let $\Delta_n:=\{\bm p=(p_1,\ldots,p_n)\in[0,1]^n:p_1+\cdots+p_n=1\}$.

\begin{proposition}\label{prop}
Let $\bm A$ be the payoff matrix for a $4\times n$ matrix game, with the additional constraint that the row player is required to use a mixed strategy of the form 
$$
\bm p=((1-p_1)(1-p_2),(1-p_1)p_2,p_1(1-p_2),p_1p_2)
$$
for some $p_1,p_2\in[0,1]$.  Let us describe such elements of $\Delta_4$ as belonging to $\Delta_2\times\Delta_2$.  Then the lower value of the game is 
\begin{equation}\label{lower}
\underline{v}=\max_{\bm p\in\Delta_2\times\Delta_2}\min_{\bm q\in\Delta_n}\bm p\bm A\bm q^{\emph{\T}}=\max_{\bm p\in\Delta_2\times\Delta_2}\min_{1\le j\le n}(\bm p\bm A)_j,
\end{equation}
while the upper value of the game is
\begin{equation}\label{upper}
\overline{v}=\max_{\bm p\in\Delta_4}\min_{\bm q\in\Delta_n}\bm p\bm A\bm q^{\emph{\T}}=\max_{\bm p\in\Delta_4}\min_{1\le j\le n}(\bm p\bm A)_j,
\end{equation}
which is equal to the value of the unconstrained game.  In particular, $\underline{v}=\overline{v}$ if and only if the maximum in \eqref{upper} occurs at a point in $\Delta_2\times\Delta_2$.
\end{proposition}

\begin{proof}
Equation \eqref{lower} is by definition.
The value of the unconstrained game is, by the minimax theorem,
\begin{equation}\label{upper2}
\max_{\bm p\in\Delta_4}\min_{\bm q\in\Delta_n}\bm p\bm A\bm q^{\T}=\min_{\bm q\in\Delta_n}\max_{\bm p\in\Delta_4}\bm p\bm A\bm q^{\T}=\min_{\bm q\in\Delta_n}\max_{\bm p\in\Delta_2\times\Delta_2}\bm p\bm A\bm q^{\T},
\end{equation}
the right side of which is, by definition, the upper value of the constrained game.  The last equality uses the fact that a linear function has the same maximum over $\Delta_4$ as over $\Delta_2\times\Delta_2$ because the latter contains the extreme points of the former (namely $(0,0,0,1),(0,0,1,0),(0,1,0,0),(1,0,0,0)$).
\end{proof}

\begin{remark}
A mixed strategy $\bm p$ for the row player that achieves the maximum in \eqref{lower} is called a \textit{maximin} strategy, and it assures the row player of an expected gain of at least $\underline{v}$.  A mixed strategy $\bm q$ for the column player that achieves the minimum in the center or on the right side of \eqref{upper2} is called a \textit{minimax} strategy, and it assures the column player of an expected loss of at most $\overline{v}$.
\end{remark}

For fixed $\theta\in(0,1/2]$, we can obtain the Players' maximin strategy and the lower value of the game as follows:  Given an arbitrary probabilities $p_1$ and $p_2$ (of drawing on 5 for Player~1 and Player~2), minimize
$$
(1-p_1)(1-p_2)a_{0,0,T}(\theta)+(1-p_1)p_2\,a_{0,1,T}(\theta)+p_1(1-p_2)a_{1,0,T}(\theta)+p_1p_2\,a_{1,1,T}(\theta)
$$
as a function of $T$ (this is Banker's best response $T=T_\theta(p_1,p_2)$), and then maximize
\begin{eqnarray}\label{E0}
E_\theta^0(p_1,p_2)&:=&(1-p_1)(1-p_2)a_{0,0,T_\theta(p_1,p_2)}(\theta)+(1-p_1)p_2\,a_{0,1,T_\theta(p_1,p_2)}(\theta)\nonumber\\
&&\qquad{}+p_1(1-p_2)a_{1,0,T_\theta(p_1,p_2)}(\theta)+p_1 p_2\,a_{1,1,T_\theta(p_1,p_2)}(\theta)
\end{eqnarray}
as a function of $(p_1,p_2)$.  The maximizing $(p_1,p_2)$ is the Players' maximin strategy, and the maximal value of $E_\theta^0(p_1,p_2)$ is the lower value of the game, assuming $\theta$ is fixed.  (Cf.~\eqref{lower}.)

For a given Banker information set $(k_1,k_2,j)$, Banker's optimal move (draw or stand) may or may not depend on the Players' strategy $(p_1,p_2)$.  In fact, for only $m$ of the 1144 information sets, where $52\le m\le69$, is there dependence on $(p_1,p_2)$.  It follows that
$E_\theta^0(p_1,p_2)$ has the form
\begin{eqnarray*}
E_\theta^0(p_1,p_2)&=&E_\theta((1-p_1)(1-p_2),(1-p_1)p_2,p_1(1-p_2),p_1p_2)\\
&=&a_0+b_0 p_1+c_0 p_2+d_0 p_1p_2\\
&&\;{}+\sum_{i=1}^m\min(a_i+b_i p_1+c_i p_2+d_i p_1p_2,\,a_i'+b_i' p_1+c_i' p_2+d_i' p_1p_2),
\end{eqnarray*}
where the constants $a_i,b_i,c_i,d_i,a_i',b_i',c_i',d_i'$ are computable rational numbers (if $\theta$ is rational) depending on $\theta$---but they are not the same as the ones in Section~\ref{corr}.  

For fixed $\theta$, the function $E_\theta^0(p_1,p_2)$, although it depends on only two variables instead of three, is more complicated than $E_\theta(\bm p)$.  It is not concave, and its maximum does not necessarily occur at an intersection of two of the $m+4$ curves
\begin{equation}\label{curve}
a_i+b_i p_1+c_i p_2+d_i p_1p_2=a_i'+b_i' p_1+c_i' p_2+d_i' p_1p_2, \quad 1\le i\le m,
\end{equation}
and $p_1=0$, $p_1=1$, $p_2=0$, and $p_2=1$.  Its maximum could occur at a point on a single curve but typically occurs at a point of intersection.  The $m$ curves in \eqref{curve} might be called ``indifference curves''.
This leads to an algorithm to find the optimal $(p_1,p_2)$.  For each of the $2\binom{m+4}{2}$ potential points $(p_1,p_2)$ just mentioned, check whether $0\le p_1\le1$ and $0\le p_2\le1$, and if so, evaluate $E_\theta^0(p_1,p_2)$.  Then determine at which such $(p_1,p_2)$ the value $E_\theta^0(p_1,p_2)$ is largest and if it is uniquely so.   Finally, confirm that this gives a global maximum.  (If it does not, look for a global maximum along one of the $m+4$ curves.  The global maximum cannot occur at a point that avoids all of these curves because $1$, $p_1$, $p_2$, and $p_1 p_2$ are harmonic in $(p_1,p_2)$.)

In the case $\theta=1/2$, the number of summands is $m=68$ and among the $2\binom{72}{2}$ potential points of intersection there is a unique maximum at
\begin{equation*}
p_1=p_2=\frac{-319+\sqrt{245569}}{224}\approx0.788166;
\end{equation*}
and $E_{1/2}^0(p_1,p_2) = 5\,(-1933207795 + 260493 \sqrt{245569})/[98(13)^9]\approx-0.00867999$ there.  It can then be confirmed that this determines a global maximum.  Furthermore, Banker's best response to this choice of $(p_1,p_2)$ is exactly as in Table \ref{best, theta=1/2}.  This completes the derivation in the case $\theta=1/2$.

Next, we extend this solution to the largest $\theta$-interval in $(0,1/2]$ for which Table \ref{best, theta=1/2} describes Banker's best response.  The maximum of $E_{1/2}^0(p_1,p_2)$ found above occurred at the intersection of the two indifference curves for $(6,10,6)$ and $(10,6,6)$.  The $\theta$-dependent versions of these two indifference curves intersect at the point $(p_1(\theta),p_2(\theta))$, where
\begin{eqnarray*}\label{p1,p2}
p_1(\theta)&=&\frac{-6191 + 932160\,\theta - 1065024\,\theta^2-s(\theta)}{32\, (151 - 12928\,\theta + 8256\,\theta^2)},\\
p_2(\theta)&=&\frac{139055 - 1197888\,\theta + 1065024\,\theta^2+s(\theta)}{32\, (4521 + 3584\,\theta - 8256\,\theta^2)},
\end{eqnarray*}
and $s(\theta)=\sqrt{[20687 - 1065024\,\theta(1-\theta)] [20687 - 1556544\,\theta(1-\theta)]}$.  Then, with $\bm A(\theta)$ as before, 
\begin{eqnarray*}
&&((1-p_1(\theta))(1-p_2(\theta)),(1-p_1(\theta))p_2(\theta),p_1(\theta)(1-p_2(\theta)),p_1(\theta)p_2(\theta))\bm A(\theta)\\
&&\qquad\qquad\qquad{}=(v(\theta),v(\theta),v(\theta),v(\theta))
\end{eqnarray*}
with
\begin{eqnarray*}
v(\theta)&:=&-[94430296089921 - 6646323952883456\,\theta - 
 25262343281817856\,\theta^2 \\
 &&\quad{}+ 63817334469402624\,\theta^3 - 
 31908667234701312\,\theta^4\\
&&\quad{}-3 (980324411 - 4975425984\,\theta(1-\theta))\,s(\theta)]\\
&&\;{}/[21208998746 (151 - 12928\,\theta + 8256\,\theta^2) (4521 + 3584\,\theta - 8256\,\theta^2)].
\end{eqnarray*}

The first change in the matrix of Table \ref{best, theta=1/2} occurs at the $(9,3)$ entry, which changes from 3 to 4 as $\theta$ goes from $\theta>\theta_*$ to $\theta<\theta_*$, where $\theta_*\approx0.4958752$ ($\theta_*$ is a root of a quartic polynomial).  However, for $\theta$ near $\theta_*$, we find that the intersection of the two indifference curves for $(6,10,6)$ and for $(10,6,6)$ does not determine a global maximum (or even a local maximum), and in fact the global maximum occurs along the indifference curve for $(10,6,6)$.  This leads to a different expression for the value, call it  $v_*(\theta)$, and the $\theta$ at which $v(\theta)=v_*(\theta)$, call it $\theta_{**}$, is the actual left endpoint of the first interval, $(\theta_{**},1/2]$.  We find that $\theta_{**}\approx0.496212$ ($\theta_{**}$ is a root of a polynomial of degree 8; see Appendix B).  This completes the derivation for the interval $(\theta_{**},1/2]$.

Repeating this process, we find that there are 131 such intervals in $(0,1/2]$.  That is, there exist $0=\theta_0<\theta_1<\theta_2<\cdots<\theta_{130}=\theta_{**}<\theta_{131}=1/2$ such that the independent cooperative equilibrium (described by the Players' maximin strategy and Banker's best response; the latter, as we saw in Proposition~\ref{prop}, is not Banker's minimax strategy so is useful primarily for determining the lower value function) is a continuous function of $\theta$ on interval $i$, namely $(\theta_{i-1},\theta_i)$, for $i=1,2,\ldots,131$.  Each $\theta_i$ is a root of a polynomial of degree 13 or less.  In exactly nine of these intervals, the maximum occurs along a single indifference curve rather than at a point of intersection.  At the boundary points there are discontinuities in Banker's best response, whereas the Players' strategies are typically continuous except for a number of discontinuities.  Actually, at the boundary points, solutions from both adjacent intervals apply, so there is nonuniqueness of the Players' strategies at points of discontinuity.  In Figure \ref{indep-p1,p2} we graph $p_1$ and $p_2$ as functions of $\theta$.

\begin{figure}[htb]
\centering
\includegraphics[width=4.5in]{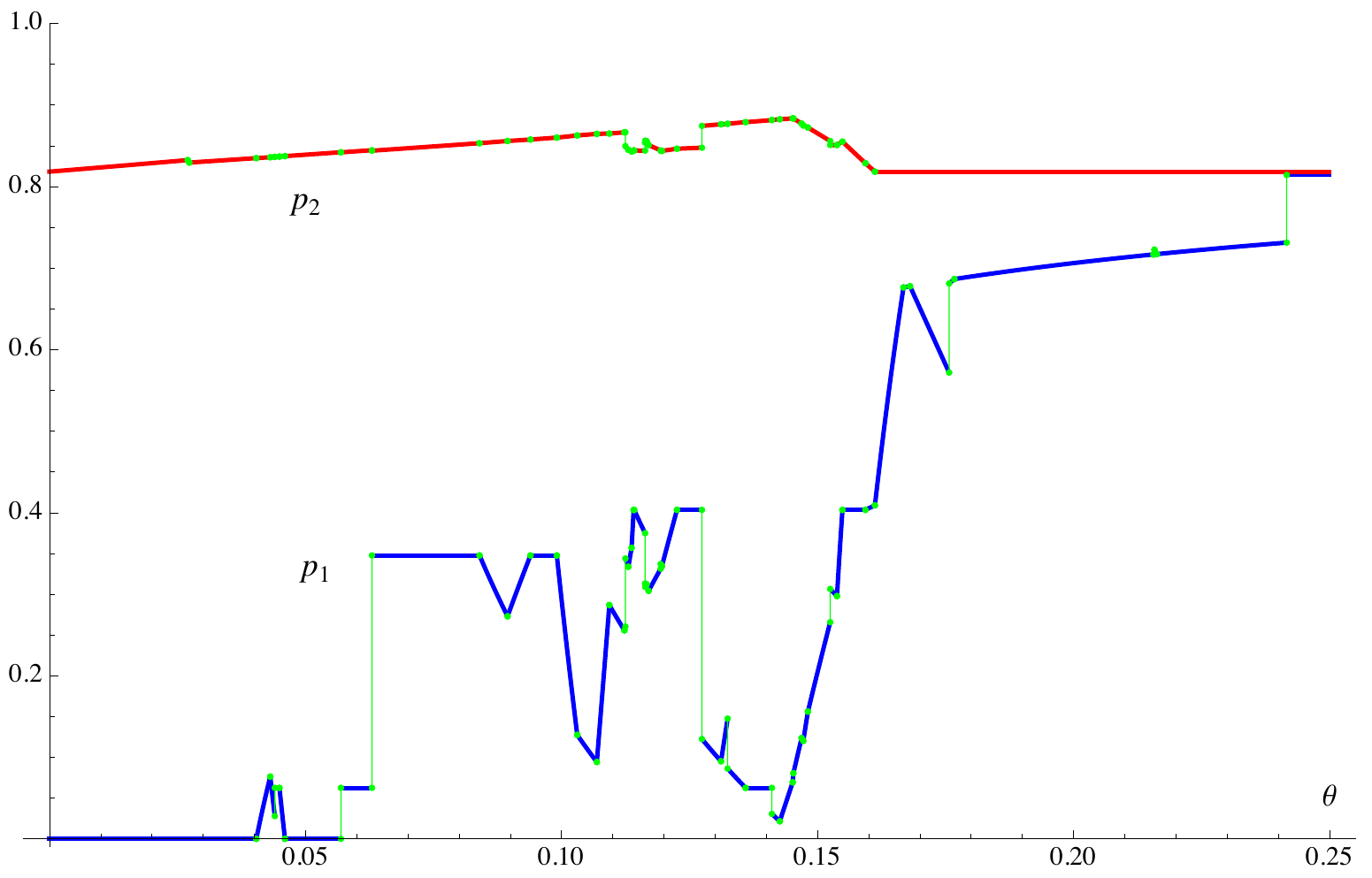}
\caption{\label{indep-p1,p2}The graphs of the Players' strategies $p_1$ and $p_2$ in the independent cooperative equilibrium, restricted to $\theta\in(0,1/4]$.  ($p_1=9/11$ for $0.241681<\theta<0.432877$ and $p_2=9/11$ for $0.161238<\theta<0.495084$, approximately.)  $p_1$ has 13 discontinuities on $(0,1/2]$, whereas $p_2$ has 11 discontinuities.}
\end{figure}

We have seen that the correlated and independent cooperative equilibria coincide when $\theta\in(9588/37663, 55716/128711)\approx(0.254573, 0.432877)$.  With two exceptions, these are the only $\theta$ values at which the two equilibria coincide.  The exceptions are $\theta_{79}\approx0.166815$ and $\theta_{88}\approx0.215651$, as can be seen from the plot of the difference between the upper and lower value functions in Figure \ref{upper-lower}.

\begin{figure}[htb]
\centering
\includegraphics[width=4.5in]{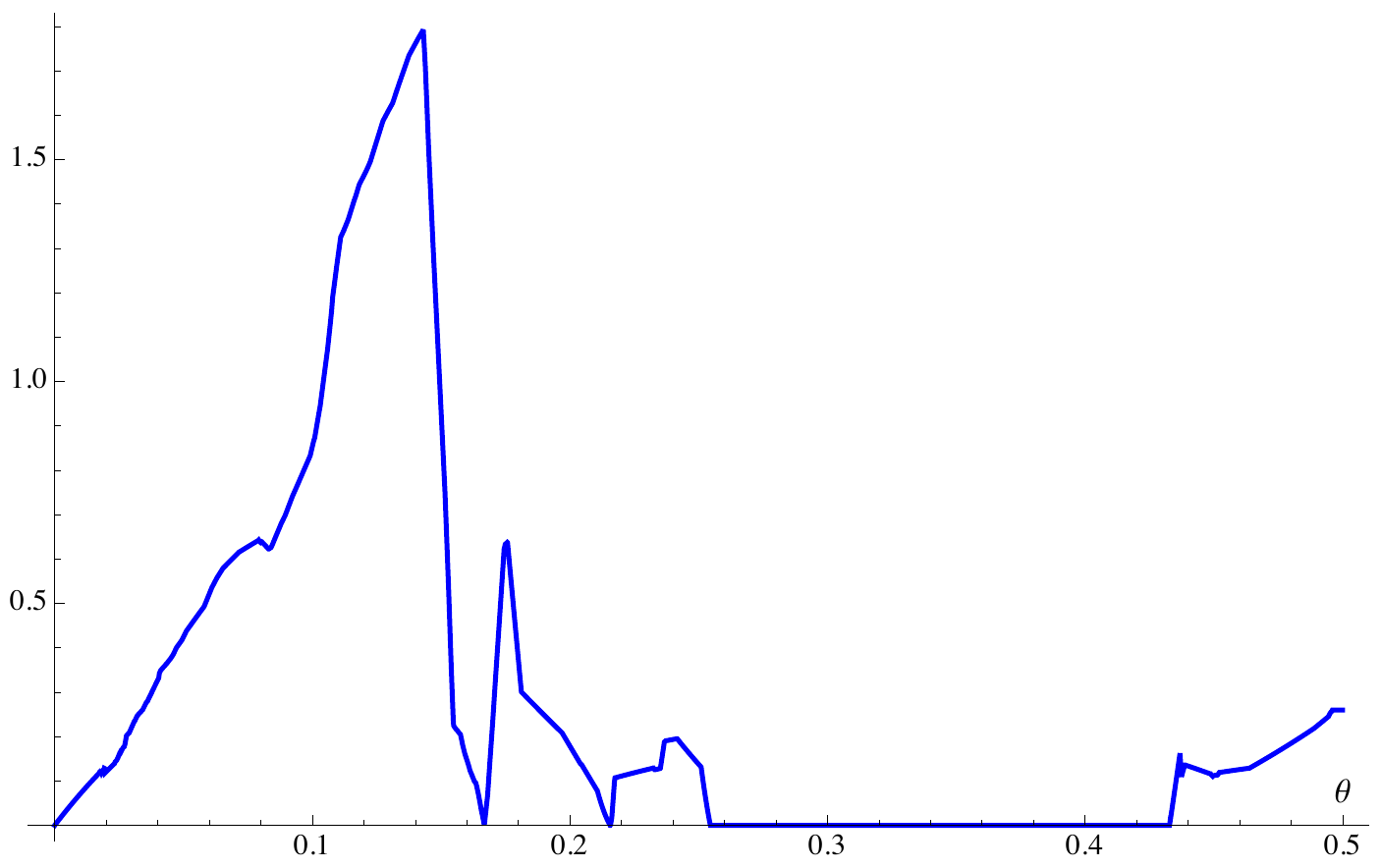}
\caption{\label{upper-lower}The graph of the difference between the upper and lower value functions, multiplied by $10^5$, in the independent cooperative equilibrium.}
\end{figure}

The complete details of the independent cooperative equilibrium are summarized in Appendix B.

\section{Nash equilibrium}\label{Nash}

Here we find the Nash equilibrium of our $2\times2\times2^{1144}$ trimatrix game.  Our method involves finding one or more Nash equilibria explicitly, but it does not permit any uniqueness assertions.  Let us begin with the case $\theta=1/2$.  For now we simply claim that there exists a Nash equilibrium in this case with $p_1=p_2=9/11$.  Banker's best response is as in Table \ref{best, theta=1/2} but with five changes:  Entry 5 at $(10,10)$, $(10,11)$, and $(11,10)$ becomes 5+.  Entry 5+ at $(6,10)$ and $(10,6)$ becomes 6.  Thus, there are now three information sets at which Banker is indifferent, $(10,10,6)$, $(10,11,6)$, and $(11,10,6)$.

Let $\bm A$ and $\bm B$ be the $4\times8$ payoff matrices for Player~1 vs.\ Banker and for Player~2 vs.\ Banker, but with the rows labeled by the Players' pure strategies: SS, SD, DS, DD on 5 by Player~1 and Player~2.  The columns are labeled by Banker's eight pure strategies: SSS, SSD, SDS, SDD, DSS, DSD, DDS, DDD on $(10,10,6)$, $(10,11,6)$, and $(11,10,6)$.  Of course, Banker makes a best response except in the three cases in which he is indifferent.  We find that
\begin{eqnarray*}
&&\!\!\!2^{-4}(13)^9\bm A\\
&&{}\!\!\!=-\begin{pmatrix}
5774122 & 5774122 & 4995370 & 4995370 & 4605994 & 4605994 & 3827242 & 3827242 \\
6359098 & 6359098 & 5580346 & 5580346 & 5580346 & 5580346 & 4801594 & 4801594 \\
5127763 & 5127763 & 5300819 & 5300819 & 5387347 & 5387347 & 5560403 & 5560403 \\
5756515 & 5756515 & 5929571 & 5929571 & 5929571 & 5929571 & 6102627 & 6102627\end{pmatrix}
\end{eqnarray*}
and
\begin{eqnarray*}
&&\!\!\!2^{-4}(13)^9 \bm B\\
&&\!\!\!{}=-\begin{pmatrix}5774122 & 4995370 & 5774122 & 4995370 & 4605994 & 3827242 & 4605994 & 3827242 \\
5127763 & 5300819 & 5127763 & 5300819 & 5387347 & 5560403 & 5387347 & 5560403 \\
6359098 & 5580346 & 6359098 & 5580346 & 5580346 & 4801594 & 5580346 & 4801594 \\
5756515 & 5929571 & 5756515 & 5929571 & 5929571 & 6102627 & 5929571 & 6102627\end{pmatrix}\!.
\end{eqnarray*}

Since the Players must act independently, we let
\begin{eqnarray*}
\bm a(p_1,p_2)&:=&((1-p_1)(1-p_2),(1-p_1)p_2,p_1(1-p_2),p_1p_2)\bm A,\\
\bm b(p_1,p_2)&:=&((1-p_1)(1-p_2),(1-p_1)p_2,p_1(1-p_2),p_1p_2)\bm B.
\end{eqnarray*}
We now determine whether there is a mixture $\bm q=(q_1,q_2,q_3,q_4,q_5,q_6,q_7,q_8)$ of Banker's eight pure strategies such that $\bm a(p_1,9/11)\bm q^{\T}$ is constant in $p_1$ and $\bm b(9/11,p_2)\bm q^{\T}$ is constant in $p_2$.  This would ensure that $p_1=9/11$ is a best response to $p_2=9/11$ and $\bm q$; similarly, $p_2=9/11$ is a best response to $p_1=9/11$ and $\bm q$.  And of course $\bm q$ is automatically a best response to $p_1=p_2=9/11$.  A necessary and sufficient condition on $\bm q$ is
\begin{eqnarray*}
q_6 &=& \frac{15175619}{10469888} - \frac{23 q_1 + 23 q_2 + 12 q_3 + 12 q_4 + 11 q_5}{11},\\
q_7 &=& \frac{15175619}{10469888} - \frac{23 q_1 + 12 q_2 + 23 q_3 + 12 q_4 + 11 q_5}{11},\\
q_8 &=& -\frac{9940675}{5234944}  + \frac{35 q_1 + 24 q_2 + 24 q_3 + 13 q_4 + 11 q_5}{11},
\end{eqnarray*}
and $q_j\ge0$ for $j=1,2,\ldots,8$.  Summing the three equations gives $q_6+q_7+q_8=1-q_1-q_2-q_3-q_4-q_5$, so any such $\bm q$ is automatically a probability vector.

By testing all possible supports of size two or three, we find that the eight Banker pure strategies are mixed in $11$ extreme Nash equilibria as follows:
\begin{eqnarray*}
1.&&(0, 15175619/33313280, 15175619/33313280, 0, 0, 0, 0, 1481021/16656640).\\
2.&&(0, 1/2, 4229827/11421696, 0, 0, 0, 1481021/11421696, 0).\\
3.&&(0, 4229827/11421696, 1/2, 0, 0, 1481021/11421696, 0, 0).\\
4.&&(0, 4705731/12373504, 4705731/12373504, 0, 1481021/6186752, 0, 0, 0).\\
5.&&(0, 3753923/10469888, 3753923/10469888, 1481021/5234944, 0, 0, 0, 0).\\
6.&&(15175619/21891584, 0, 0, 0, 0, 0, 0, 6715965/21891584).\\
7.&&(1988135/3331328, 0, 0, 0, 0, 1343193/6662656, 1343193/6662656, 0).\\
8.&&(1568577/3807232, 0, 0, 0, 2238655/3807232, 0, 0, 0).\\
9.&&(3753923/10469888, 0, 0, 6715965/10469888, 0, 0, 0, 0),\\
10.&&(4229827/10945792, 0, 6715965/21891584, 0, 0, 6715965/21891584, 0, 0).\\
11.&&(4229827/10945792, 6715965/21891584, 0, 0, 0, 0, 6715965/21891584, 0).
\end{eqnarray*}
If Player~1, Player~2, and Banker play according to their equilibrium strategies, Banker's expected gain per unit stake is
\begin{equation*}
\frac{11138203216}{(11)^2(13)^9}\approx0.00868040.
\end{equation*}

A list of extreme Nash equilibria is the usual way to express the solutions of a noncooperative game, but it is unnecessarily complicated in this case.  A better approach is to express these equilibria in terms of behavioral strategies.  Only one of the three information sets, $(10,10,6)$, $(10,11,6)$, and $(11,10,6)$, is encountered during the play of a single game.  Thus, knowing the draw probabilities in each of the three cases is sufficient.  Expressed in terms of these behavioral strategies, the Banker strategies in the 11 extreme equilibria all have the form $(r_1,r_2,r_2)$, and there are only two extreme points, namely
$$
\bigg(0,\frac{6715965}{10469888},\frac{6715965}{10469888}\bigg)\quad\text{and}\quad\bigg(\frac{2238655}{3807232},0,0\bigg).
$$
This completes the derivation in the case $\theta = 1/2$.

In fact the same Nash equilibria apply on the interval $(\theta_*,1/2]$ for $\theta_*=799/1604\approx0.498130$.  (Here they are not $\theta$-dependent.)  This is the first $\theta$ (moving right to left) at which a change occurs in Table \ref{best, theta=1/2} (beyond the five changes already noted).  As $\theta$ moves from $\theta>\theta_*$ to $\theta<\theta_*$, the $(2,5)$ entry in Table \ref{best, theta=1/2} changes from 4 to 5.  

We can repeat this process 40 times.  In each new interval the Players' strategy is the same (independent with $p_1=p_2=9/11$), while Banker's strategy changes from the previous interval.  This determines the Nash equilibria for all $\theta>5772/33847\approx0.170532$.  At the next interval there is no mixture $\bm q$ satisfying the required properties.

Now let us consider what happens when $\theta<5772/33847\approx0.170532$.  Each of the remaining 62 intervals is one of two types:  Banker mixes on two information sets (40 cases), or Banker mixes on one information set and $p_1=0$ (22 cases).  First we consider the interval whose right endpoint is $\theta'':=5772/33847\approx0.170532$.  Suppose we know that Banker mixes on $(10,0,4)$ and on $(11,10,6)$ in this interval.  The intersection of the two indifference curves occurs at $(p_1(\theta),p_2(\theta))$, where 
$$
p_1(\theta)=\frac{1443 - 9304\,\theta}{481 - 3850\,\theta},\quad  p_2(\theta) = \frac{9}{11}.
$$
Evaluating the payoff matrices $\bm A$ and $\bm B$ for Players~1 and 2 (now $4\times 4$), we can argue as above, and this leads to two extreme equilibria, which have the same Banker behavioral strategies.  The left endpoint of the interval is $$\theta':=(82755888 + 1123 \sqrt{2262279009})/803081778\approx0.169559$$ because of a change in Banker's strategy at $(9,8,3)$.  As for the information sets on which Banker mixes in the next interval, there are three candidates, namely any two of the three $(10,0,4)$, $(11,10,6)$, and $(9,8,3)$, with the first two being most likely.  This approach allows us to move from one interval to the next in a systematic way.  In those cases where Banker mixes on only one information set and $p_1=0$, we do not need to have $\bm a(p_1,p_2(\theta))\bm q^\T$ constant in $p_1$; it suffices that it be maximized at $p_1=0$.

In Figure \ref{Nash-p1,p2} we graph $p_1$ and $p_2$ as functions of $\theta$.  We restrict to $\theta\in(0,1/4]$ since $p_1=p_2=9/11$ whenever $\theta>5772/33847\approx0.170532$.  The complete details of the Nash equilibrium are summarized in Appendix C.

\begin{figure}[htb]
\centering
\includegraphics[width=4.5in]{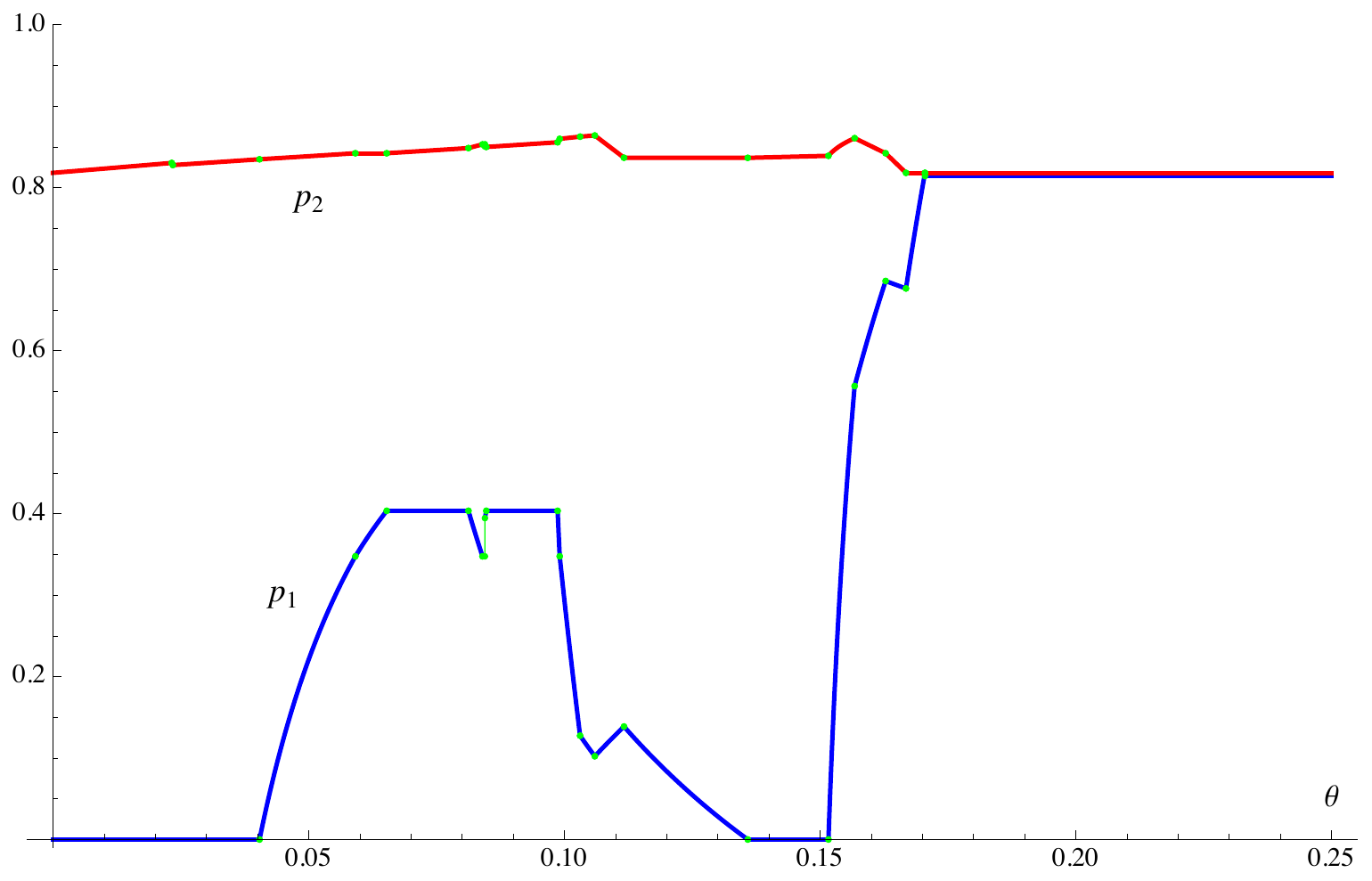}
\caption{\label{Nash-p1,p2}The graphs of the Players' strategies $p_1$ and $p_2$ in the Nash equilibrium, restricted to $\theta\in(0,1/4]$.  ($p_1=p_2=9/11$ for $\theta>5772/33847\approx0.170532$.)  Both $p_1$ and $p_2$ have a unique discontinuity, at $\theta_{31}\approx0.0844782$.}
\end{figure}

Downton and Lockwood (1976) observed that Player~1 has positive expectation for $\theta=1/10$.  The reason is clear:  Banker focusses his attention on Player~2 and therefore plays suboptimally against Player~1.  How large is this expectation for small $\theta$?  Let us consider interval 1 ($0<\theta<0.0172597$), in which Player~2 bets about 56.9383 or more times as much as Player~1.  Player~1's expectation, per unit bet by Player~1, when both Players and Banker use their equilibrium strategies, is
$$
\frac{928 (53214419 - 33787088\,\theta)}{116649493103 (89 - 68\,\theta)},
$$
which is about 0.00475669 at the left endpoint of the interval and about 0.004767\break43 at the right endpoint.  A 0.475\% advantage is substantial but it presumably occurs rarely.

Finally, let us explain why Downton and Lockwood's (1976) Nash equilibrium algorithm is incorrect.  Let $E_\theta^1(p_1,p_2)$ be the expectation of Player~1 per unit bet by Player~1 when Players~1 and 2 independently draw on 5 with probabilities $p_1$ and $p_2$ and Banker makes a best response to $(p_1,p_2)$ and $\theta$.  Let $E_\theta^2(p_1,p_2)$ be defined analogously.  Then, recalling \eqref{E0},
$$
E_\theta^0(p_1,p_2)=\theta E_\theta^1(p_1,p_2)+(1-\theta)E_\theta^2(p_1,p_2).
$$
($E_\theta^0(p_1,p_2)$ is continuous in $(p_1,p_2)$ for fixed $\theta$, but $E_\theta^1(p_1,p_2)$ and $E_\theta^2(p_1,p_2)$ are not.)  Downton and Lockwood (1976) proposed an algorithm for evaluating the Nash equilibrium based on the functions $E_\theta^1(p_1,p_2)$ and $E_\theta^2(p_1,p_2)$.  (We consider it in the context of our simplified model rather than in terms of the more elaborate model they analyzed.)  Specifically, they defined
\begin{equation}\label{p_2(p_1)}
\hat p_2(p_1):=\argmax_{p_2\in[0,1]}E_\theta^2(p_1,p_2),\quad p_1\in[0,1],
\end{equation}
and 
\begin{equation}\label{hats}
\hat p_1:=\argmax_{p_1\in[0,1]}E_\theta^1(p_1,\hat p_2(p_1)),\qquad \hat p_2:=\hat p_2(\hat p_1).
\end{equation}
It appears that the aim was to find $(\hat p_1,\hat p_2)$ such that 
\begin{equation}\label{NE1}
E_\theta^1(\hat p_1,\hat p_2)\ge E_\theta^1(p_1,\hat p_2)\text{ for all $p_1\in[0,1]$}
\end{equation}
and
\begin{equation}\label{NE2}
E_\theta^2(\hat p_1,\hat p_2)\ge E_\theta^2(\hat p_1,p_2)\text{ for all $p_2\in[0,1]$},
\end{equation}
though only the second of these two inequalities actually follows from \eqref{p_2(p_1)} and \eqref{hats}.  Then Banker would make a best response to $(\hat p_1,\hat p_2)$.  While inequality \eqref{NE1} appears to say that $\hat p_1$ is a best response by Player~1, it  does \textit{not} do so because its right-hand side $E_\theta^1(p_1, \hat p_2)$ is defined in terms of a Banker best response to $(p_1,\hat p_2)$, not $(\hat p_1,\hat p_2)$.  Inequality \eqref{NE2} has the same problem.  We observe that, when $\theta=1/2$ and $\hat p_1=\hat p_2=9/11$, \eqref{NE1} and \eqref{NE2} fail, which helps to confirm that the method is flawed.

\section{Summary}
\textit{Baccara banque} is a three-person zero-sum game parameterized by $\theta\in(0,1)$.  The players are called Player~1, Player~2, and Banker, and the amounts bet on the hands of Players~1 and 2 are in the proportions $\theta:1-\theta$.  Assuming cards are dealt with replacement, the game is a $2\times2^{12}\times2^{1144}$ trimatrix game.  Downton and Lockwood (1976) argued that the independent cooperative equilibrium, in which Players~1 and 2 form a coalition against Banker but act independently, is more useful than the Nash equilibrium.  They did not realize that the independent cooperative equilibrium need not exist, in the sense that the lower and upper values of the game may differ.  They also computed the Nash equilibrium incorrectly.

We consider a simplified model, in which Player~2 ignores Player~1's hand, and the game becomes a $2\times2\times2^{1144}$ trimatrix game.  This allows us to assume that $\theta\in(0,1/2]$.  We find what we call the correlated cooperative equilibrium, in which Players~1 and 2 are not constrained to act independently in their coalition against Banker, and the Nash equilibrium.  Moreover, in the independent cooperative equilibrium, we evaluate the game's lower value (to the Players) and its upper value, as well as the corresponding maximin strategy of the Players and minimax strategy of Banker.  

Results are necessarily complicated by the fact that Banker's strategy has more than 100 discontinuities over the interval $(0,1/2]$.  The Players' strategies are simpler, having only a single discontinuity in the Nash equilibrium and at most 13 discontinuities in the independent cooperative equilibrium.  Necessary and sufficient conditions on $\theta$ are given for the independent and correlated cooperative equilibria to coincide.

All three solutions are described in detail in a lengthy appendix.

\newpage

\section*{Appendix A.  Correlated cooperative equilibrium}

\begin{small}
We list the 109 points of discontinuity in $(0,1/2)$ of the correlated cooperative equilibrium.  All are roots of polynomials of degree 4 or less.\medskip

\noindent$\theta_1\approx0.00178731$ (root of $14859 - 8310422\,\theta - 1783296\,\theta^2$)\\
$\theta_2\approx0.0174038$ (root of $5523429 - 312988816\,\theta - 258236096\,\theta^2 + 374933504\,\theta^3 $)\\
$\theta_3\approx0.0181285$ (root of $62061 - 3387468\,\theta - 2042144\,\theta^2 + 3341184\,\theta^3$)\\
$\theta_4\approx0.0189082$ (root of $5523429 - 288064954\,\theta - 220590896\,\theta^2 + 327996416\,\theta^3 $)\\
$\theta_5\approx0.0192925$ (root of $69459 - 3562914\,\theta - 2005040\,\theta^2 + 3472128\,\theta^3 $)\\
$\theta_6\approx0.0194313$ (root of $5523429 - 279757000\,\theta - 238117376\,\theta^2 + 342425600\,\theta^3 $)\\
$\theta_7\approx0.0234950$ (root of $5523429 - 230832382\,\theta - 187612816\,\theta^2 + 273681408\,\theta^3 $)\\
$\theta_8\approx0.0243712$ (root of $5523429 - 222524428\,\theta - 175064416\,\theta^2 + 258035712\,\theta^3 $)\\
$\theta_9\approx0.0244572$ (root of $60965 - 2641495\,\theta + 6110982\,\theta^2 - 1140624\,\theta^3 $)\\
$\theta_{10}\approx0.0250657$ (root of $5523429 - 213028222\,\theta - 302679016\,\theta^2 + 409265088\,\theta^3 $)\\
$\theta_{11}\approx0.0252993$ (root of $4865007 - 185414566\,\theta - 281528608\,\theta^2 + 372929664\,\theta^3 $)\\
$\theta_{12}\approx0.0254210$ (root of $365790 - 13954949\,\theta - 17801764\,\theta^2 + 28135392\,\theta^3$)\\
$\theta_{13}\approx0.0265477$ (root of $6181851 - 224025970\,\theta - 344865592\,\theta^2 + 457867584\,\theta^3$)\\
$\theta_{14}\approx0.0278844$ (root of $5523429 - 188104360\,\theta - 370801588\,\theta^2 + 464234016\,\theta^3$)\\
$\theta_{15}\approx0.0308941$ (root of $6181851 - 190794154\,\theta - 313801288\,\theta^2 + 409265088\,\theta^3$)\\
$\theta_{16}\approx0.0320785$ (root of $5523429 - 163180498\,\theta - 292650880\,\theta^2 + 372929664\,\theta^3$)\\
$\theta_{17}\approx0.0355903$ (root of $5523429 - 147487696\,\theta - 227192292\,\theta^2 + 299189152\,\theta^3$)\\
$\theta_{18}\approx0.0363942$ (root of $4865007 - 128181994\,\theta - 158955480\,\theta^2 + 220151872\,\theta^3$)\\
$\theta_{19}\approx0.0377318$ (root of $5523429 - 139179742\,\theta - 201142056\,\theta^2 + 268754368\,\theta^3$)\\
$\theta_{20}\approx0.0406437$ (root of $62061 - 1516794\,\theta - 294464\,\theta^2 + 1095296\,\theta^3$)\\
$\theta_{21}\approx0.0411247$ (root of $6181851 - 141869536\,\theta - 217278396\,\theta^2 + 286922080\,\theta^3$)\\
$\theta_{22}\approx0.0432219$ (root of $4865007 - 103258132\,\theta - 227078052\,\theta^2 + 275120800\,\theta^3$)\\
$\theta_{23}\approx0.0472920$ (root of $5523429 - 105947926\,\theta - 243214392\,\theta^2 + 293288512\,\theta^3$)\\
$\theta_{24}\approx0.0511173$ (root of $5523429 - 97639972\,\theta - 217164156\,\theta^2 + 262853728\,\theta^3$)\\
$\theta_{25}\approx0.0610313$ (root of $5523429 - 81947170\,\theta - 151705568\,\theta^2 + 189113216\,\theta^3$)\\
$\theta_{26}\approx0.0632010$ (root of $4521 - 76162\,\theta + 73232\,\theta^2$)\\
$\theta_{27}\approx0.0653381$ (root of $6181851 - 83713858\,\theta - 181200024\,\theta^2 + 220151872\,\theta^3$)\\
$\theta_{28}\approx0.0714954$ (root of $4865007 - 62641468\,\theta - 83468756\,\theta^2 + 110075936\,\theta^3$)\\
$\theta_{29}\approx0.0791813$ (root of $4791849 - 54047252\,\theta - 90707308\,\theta^2 + 113588704\,\theta^3$)\\
$\theta_{30}\approx0.0803439$ (root of $23959245 - 261928306\,\theta - 500622944\,\theta^2 + 610645376\,\theta^3$)\\
$\theta_{31}\approx0.0880358$ (root of $5523429 - 57023308\,\theta - 73554860\,\theta^2 + 97808864\,\theta^3$)\\
$\theta_{32}\approx0.0922648$ (root of $6181851 - 59713102\,\theta - 89691200\,\theta^2 + 115976576\,\theta^3$)\\
$\theta_{33}\approx0.100626$ (root of $23959245 - 204695734\,\theta - 378049816\,\theta^2 + 457867584\,\theta^3$)\\
$\theta_{34}\approx0.107945$ (root of $23959245 - 179771872\,\theta - 446172388\,\theta^2 + 512836512\,\theta^3$)\\
$\theta_{35}\approx0.109408$ (root of $23959245 - 188079826\,\theta - 325949344\,\theta^2 + 396998016\,\theta^3$)\\
$\theta_{36}\approx0.111040$ (root of $23959245 - 179771872\,\theta - 373035748\,\theta^2 + 439699872\,\theta^3$)\\
$\theta_{37}\approx0.118237$ (root of $5523429 - 32099446\,\theta - 141677432\,\theta^2 + 152777792\,\theta^3$)\\
$\theta_{38}\approx0.137545$ (root of $23959245 - 139155208\,\theta - 302563092\,\theta^2 + 347791648\,\theta^3$)\\
$\theta_{39}\approx0.143112$ (root of $593073 - 3646986\,\theta - 4309072\,\theta^2 + 5837216\,\theta^3$)\\
$\theta_{40}\approx0.143412$ (root of $65897 + 19507362\,\theta - 153485048\,\theta^2 + 99417792\,\theta^3$)\\
$\theta_{41}\approx0.143932$ (root of $481 - 4126\,\theta + 5448\,\theta^2$)\\
$\theta_{42}\approx0.144546$ (root of $65897 - 4033520\,\theta + 28720932\,\theta^2 - 27466320\,\theta^3$)\\
$\theta_{43}\approx0.151528$ (root of $219817 - 1680414\,\theta + 1516200\,\theta^2$)\\
$\theta_{44}\approx0.152990$ (root of $4339965 - 32754362\,\theta + 28672944\,\theta^2$)\\
$\theta_{45}\approx0.156435$ (root of $3403491 - 14489390\,\theta-57224992\,\theta^2+68843136\,\theta^3$)\\
$\theta_{46}\approx0.157548$ (root of $4397015 - 89027802\,\theta + 444963880\,\theta^2 - 361964352\,\theta^3$)\\
$\theta_{47}\approx0.158564$ (root of $394378 - 3050443\,\theta + 3552264\,\theta^2$)\\
$\theta_{48}\approx0.159147$ (root of $11136593 - 56542746\,\theta - 92823008\,\theta^2 + 52841856\,\theta^3$)\\
$\theta_{49}\approx0.159187$ (root of $213242455 - 1046025684\,\theta - 2315452986\,\theta^2 + 2961237104\,\theta^3$)\\
$\theta_{50}\approx0.159740$ (root of $11136593 - 286621036\,\theta + 1589373892\,\theta^2 - 1472460384\,\theta^3$\\
\hglue1cm${} + 144814080\,\theta^4$)\\
$\theta_{51}\approx0.160944$ (root of $11136593 - 267206316\,\theta + 1358308292\,\theta^2 - 795305824\,\theta^3$)\\
$\theta_{52}\approx0.162296$ (root of $212131111 - 1000450356\,\theta - 2372588922\,\theta^2 + 2978282096\,\theta^3$)\\
$\theta_{53}\approx0.163363$ (root of $9950447 - 56278324\,\theta - 17556612\,\theta^2 - 92216736\,\theta^3$ \\
\hglue1cm${} + 159928320\,\theta^4$)\\
$\theta_{54}\approx0.163431$ (root of $9950447 - 38931444\,\theta - 149550212\,\theta^2 + 93143904\,\theta^3$)\\
$\theta_{55}\approx0.163513$ (root of $140933 - 990706\,\theta + 787712\,\theta^2$)\\
$\theta_{56}\approx0.163559$ (root of $6727103 - 39583584\,\theta - 9452544\,\theta^2$)\\
$\theta_{57}\approx0.163753$ (root of $186183 + 10649570\,\theta - 71977488\,\theta^2$)\\
$\theta_{58}\approx0.168900$ (root of $208377 + 18469120\,\theta - 116653632\,\theta^2$)\\
$\theta_{59}\approx0.171196$ ($=63/368$)$^{\phantom{0}}$\\
$\theta_{60}\approx0.174505$ (root of $432 + 22019\,\theta - 140366\,\theta^2$)\\
$\theta_{61}\approx0.174961$ (root of $864 + 119225\,\theta - 709664\,\theta^2$)\\
$\theta_{62}\approx0.175832$ (root of $1696 + 114441\,\theta - 705712\,\theta^2$)\\
$\theta_{63}\approx0.197056$ (root of $123527 - 14727622\,\theta + 71557040\,\theta^2$)\\
$\theta_{64}\approx0.201498$ ($=20687/102666$)$^{\phantom{0}}$\\
$\theta_{65}\approx0.204690$ (root of $123527 - 15275812\,\theta + 71680736\,\theta^2$)\\
$\theta_{66}\approx0.210684$ (root of $280635 - 34011814\,\theta + 155112880\,\theta^2$)\\
$\theta_{67}\approx0.217304$ (root of $29851341 - 2140863850\,\theta + 8775058896\,\theta^2 + 2046535680\,\theta^3$)\\
$\theta_{68}\approx0.217337$ (root of $314526381 - 1339632286\,\theta - 608162792\,\theta^2 + 521387520\,\theta^3$)\\
$\theta_{69}\approx0.217529$ (root of $25069567 - 109735176\,\theta - 25337856\,\theta^2$)\\
$\theta_{70}\approx0.232468$ (root of $32913017 - 2318653686\,\theta + 8870984112\,\theta^2 + 2125338624\,\theta^3$)\\
$\theta_{71}\approx0.236724$ (root of $29851341 - 2144933602\,\theta + 8080196496\,\theta^2 + 1892517888\,\theta^3$)\\
$\theta_{72}\approx0.237225$ (root of $13549985 - 978787096\,\theta + 3681183360\,\theta^2 + 860012544\,\theta^3$)\\
$\theta_{73}\approx0.251003$ (root of $11580791 - 43540248\,\theta - 10349568\,\theta^2$)\\
$\theta_{74}\approx0.254573$ ($=9588/37663$)\\
$\theta_{75}\approx0.255213$ ($=1591/6234$)\\
$\theta_{76}\approx0.282286$ ($=1141/4042$)\\
$\theta_{77}\approx0.284187$ ($=5796/20395$)\\
$\theta_{78}\approx0.291630$ ($=655/2246$)\\
$\theta_{79}\approx0.312202$ ($=655/2098$)\\
$\theta_{80}\approx0.315534$ ($=65/206$)\\
$\theta_{81}\approx0.317736$ ($=2397/7544$)\\
$\theta_{82}\approx0.322499$ ($=1141/3538$)\\
$\theta_{83}\approx0.332182$ ($=481/1448$)\\
$\theta_{84}\approx0.340483$ ($=2397/7040$)\\
$\theta_{85}\approx0.354185$ ($=1591/4492$)\\
$\theta_{86}\approx0.364699$ ($=655/1796$)\\
$\theta_{87}\approx0.377510$ ($=44268/117263$)\\
$\theta_{88}\approx0.377551$ ($=37/98$)\\
$\theta_{89}\approx0.384504$ ($=799/2078$)\\
$\theta_{90}\approx0.384544$ ($=2901/7544$)\\
$\theta_{91}\approx0.386798$ ($=46044/119039$)\\
$\theta_{92}\approx0.393855$ ($=141/358$)\\
$\theta_{93}\approx0.398947$ ($=1591/3988$)\\
$\theta_{94}\approx0.400756$ ($=1591/3970$)\\
$\theta_{95}\approx0.409866$ ($=1105/2696$)\\
$\theta_{96}\approx0.430543$ ($=967/2246$)\\
$\theta_{97}\approx0.432877$ ($=55716/128711$)\\
$\theta_{98}\approx0.433093$ (root of $702880555 - 205024028\,\theta - 3024826688\,\theta^2 - 575127552\,\theta^3$)\\
$\theta_{99}\approx0.436737$ (root of $1845299991 - 4521504554\,\theta + 1306396544\,\theta^2 - 1437818880\,\theta^3$)\\
$\theta_{100}\approx0.436929$ (root of $69654319 - 143712312\,\theta - 35945472\,\theta^2$)\\
$\theta_{101}\approx0.438827$ (root of $4643 - 72197\,\theta + 140412\,\theta^2$)\\
$\theta_{102}\approx0.438921$ ($=643552/1466215$)$^{\phantom{0}}$\\
$\theta_{103}\approx0.448792$ (root of $16779 - 231714\,\theta + 433000\,\theta^2$)\\
$\theta_{104}\approx0.450027$ (root of $16528913 - 1144331362\,\theta + 2213155728\,\theta^2 + 551159808\,\theta^3$)\\
$\theta_{105}\approx0.452204$ (root of $16653035 - 1155550924\,\theta + 2225068320\,\theta^2 + 550342656\,\theta^3$)\\
$\theta_{106}\approx0.463882$ (root of $77107697 - 9071939465\,\theta + 19222966632\,\theta^2 - 53280768\,\theta^3$)\\
$\theta_{107}\approx0.488948$ (root of $64032147 - 900502651\,\theta + 1573877560\,\theta^2$)\\
$\theta_{108}\approx0.496086$ (root of $6844383 - 1189837640\,\theta + 3546625856\,\theta^2 - 2370528768\,\theta^3$)\\
$\theta_{109}\approx0.496088$ (root of $6896169 - 1190915420\,\theta + 3549548480\,\theta^2 - 2372477184\,\theta^3$)\\
\end{small}
\medskip

\begin{small}
Banker's correlated cooperative equilibrium strategy (excluding mixing probabilities).  Pairs followed by a colon are Player~1's and Player~2's third-card values (10 = stand, 11 = natural).  Entries are maximum Banker drawing totals (e.g., 5 means that Banker draws on 5 or less and stands on 6 or 7).  Plus signs indicate that Banker mixes on next higher total (e.g., 5+ means Banker draws on 5 or less, mixes on 6, and stands on 7).  Number ranges in parentheses refer to the 110 intervals in $(0,1/2]$, with interval $i$ being $(\theta_{i-1},\theta_i)$; here $\theta_0=0$ and $\theta_{110}=1/2$.\medskip

$(0,0)$: 3.  
$(0,1)$: 3.  
$(0,2)$: 4 (1--98), 3 (99--110).  
$(0,3)$: 4.  
$(0,4)$: 5 (1--11), 4+ (12), 4 (13--110).  
$(0,5)$: 5 (1--84), 4 (85--110).  
$(0,6)$: 6 (1--14), 5 (15--90), 4 (91--110).  
$(0,7)$: 6 (1--34), 5 (35--36), 4 (37--79), 3 (80--110).  
$(0,8)$: 2 (1--37), 3 (38--110).  
$(0,9)$: 3.  
$(0,10)$: 5.  
$(0,11)$: 3.

$(1,0)$: 3.  
$(1,1)$: 3.  
$(1,2)$: 4.  
$(1,3)$: 4.  
$(1,4)$: 5 (1--9), 4 (10--110).  
$(1,5)$: 5 (1--92), 4 (93--110).  
$(1,6)$: 6 (1--6), 5 (7--99), 4 (100--110).  
$(1,7)$: 6 (1--30), 5 (31--35), 4 (36--110).  
$(1,8)$: 2 (1--24), 3 (25--110).  
$(1,9)$: 3.  
$(1,10)$: 5.  
$(1,11)$: 3.

$(2,0)$: 3.  
$(2,1)$: 3 (1--50), 3+ (51), 4 (52--110).  
$(2,2)$: 4.  
$(2,3)$: 4.  
$(2,4)$: 5 (1--18), 4 (19--110).  
$(2,5)$: 5 (1--109), 4 (110).  
$(2,6)$: 6 (1--4), 5 (5--110).  
$(2,7)$: 6 (1--29), 5 (30--63), 4 (64--110).  
$(2,8)$: 2 (1--16), 3 (17--110).  
$(2,9)$: 3.  
$(2,10)$: 5.  
$(2,11)$: 4.

$(3,0)$: 3 (1--88), 4 (89--110).  
$(3,1)$: 3 (1--27), 4 (28--110).  
$(3,2)$: 4.  
$(3,3)$: 4.  
$(3,4)$: 5 (1--28), 4 (29--110).  
$(3,5)$: 5.  
$(3,6)$: 6 (1--8), 5+ (9--10), 5 (11--110).  
$(3,7)$: 6 (1--33), 5 (34--86), 4 (87--110).  
$(3,8)$: 2 (1--6), 3 (7--94), 4 (95--110).  
$(3,9)$: 3.  
$(3,10)$: 5.  
$(3,11)$: 4.

$(4,0)$: 3 (1--71), 4 (72--110).  
$(4,1)$: 3 (1--15), 4 (16--110).  
$(4,2)$: 4.  
$(4,3)$: 4.  
$(4,4)$: 5.  
$(4,5)$: 5.  
$(4,6)$: 6 (1--17), 5 (18--110).  
$(4,7)$: 6 (1--38), 5 (39--110).  
$(4,8)$: 2 (1--2), 3 (3--75), 4 (76--110).  
$(4,9)$: 3 (1--84), 4 (85--110).  
$(4,10)$: 5.  
$(4,11)$: 4 (1--26), 5 (27--110).

$(5,0)$: 3 (1--66), 3+ (67--68), 4 (69--110).  
$(5,1)$: 3 (1--12), 3+ (13), 4 (14--110).  
$(5,2)$: 4.  
$(5,3)$: 4 (1--82), 5 (83--110).  
$(5,4)$: 5.  
$(5,5)$: 5.  
$(5,6)$: 6 (1--25), 5 (26--110).  
$(5,7)$: 6 (1--72), 5 (73--110).  
$(5,8)$: 2 (1--7), 3 (8--70), 4 (71--93), 5 (94--110).  
$(5,9)$: 3 (1--81), 4 (82--110).  
$(5,10)$: 5 (1), 5+ (2--8), 5 (9--110).  
$(5,11)$: 5.

$(6,0)$: 3 (1--83), 4 (84--110).  
$(6,1)$: 3 (1--21), 4 (22--110).  
$(6,2)$: 4 (1--105), 5 (106--110).  
$(6,3)$: 4 (1--76), 5 (77--110).  
$(6,4)$: 5.  
$(6,5)$: 5.  
$(6,6)$: 6.  
$(6,7)$: 6.  
$(6,8)$: 2 (1--19), 3 (20--85), 5 (86--110).  
$(6,9)$: 3 (1--103), 3+ (104), 4 (105--110).  
$(6,10)$: 6 (1--44), 5+ (45--55), 6 (56--107), 5+ (108--110).  
$(6,11)$: 6.

$(7,0)$: 3.  
$(7,1)$: 3 (1--32), 4 (33--45), 3 (46--47), 3+ (48), 3 (49--58),  4 (59--110).  
$(7,2)$: 4.  
$(7,3)$: 4.  
$(7,4)$: 5.  
$(7,5)$: 5.  
$(7,6)$: 6.  
$(7,7)$: 6.  
$(7,8)$: 2 (1--31), 3 (32--53), 2+ (54), 2 (55--57), 3 (58--110).  
$(7,9)$: 3.  
$(7,10)$: 6.  
$(7,11)$: 6.

$(8,0)$: 3.  
$(8,1)$: 3.  
$(8,2)$: 4 (1--95), 3 (96--110).  
$(8,3)$: 4.  
$(8,4)$: 5 (1--40), 4 (41--110).  
$(8,5)$: 5.  
$(8,6)$: 6 (1--42), 5 (43--110).  
$(8,7)$: 6 (1--78), 3 (79--110).  
$(8,8)$: 2.  
$(8,9)$: 3.  
$(8,10)$: 5+ (1), 6 (2--10), 5+ (11), 6 (12--13), 5+ (14--39), 5 (40--110). 
$(8,11)$: 2.

$(9,0)$: 3.  
$(9,1)$: 3.  
$(9,2)$: 4 (1--80), 3 (81--110).  
$(9,3)$: 4 (1--108), 3 (109--110).  
$(9,4)$: 5 (1--22), 4 (23--110).  
$(9,5)$: 5 (1--89), 4 (90--110).  
$(9,6)$: 6 (1--23), 5 (24--96), 4 (97--110).  
$(9,7)$: 6 (1--46), 4+ (47), 4 (48--65), 3 (66--110).  
$(9,8)$: 2 (1--64), 3 (65--110).  
$(9,9)$: 3.  
$(9,10)$: 5.  
$(9,11)$: 2 (1--43), 2+ (44), 3 (45--110).

$(10,0)$: 3 (1--39), 3+ (40--61), 4 (62--97), 4+ (98--102), 5 (103--110).  
$(10,1)$: 3 (1--5), 4 (6--87), 5 (88--110).  
$(10,2)$: 4 (1--77), 5 (78--110).  
$(10,3)$: 4 (1--59), 4+ (60), 5 (61--110).  
$(10,4)$: 5.  
$(10,5)$: 5.  
$(10,6)$: 6 (1--20), 5 (21--56), 5+ (57--59), 5 (60--61), 5+ (62--67), 6 (68--69), 5+ (70--73), 6 (74--100),  5 (101--104), 5+ (105--110).  
$(10,7)$: 6.  
$(10,8)$: 2 (1--3), 3 (4--62), 4 (63--68), 4+ (69), 5 (70--110).  
$(10,9)$: 3 (1--73), 3+ (74), 4 (75--91), 5 (92--110).  
$(10,10)$: 5+ (1--43), 5 (44--49), 5+ (50), 5 (51--52), 5+ (53), 5 (54--55), 5+ (56), 5 (57--67), 5+ (68--101), 5 (102), 5+ (103--107), 5 (108--110).  
$(10,11)$: 5 (1--74), 5+ (75--97), 5 (98--110).

$(11,0)$: 3.  
$(11,1)$: 3.  
$(11,2)$: 4.  
$(11,3)$: 4.  
$(11,4)$: 5.  
$(11,5)$: 5.  
$(11,6)$: 6.  
$(11,7)$: 6.  
$(11,8)$: 2.  
$(11,9)$: 3.  
$(11,10)$: 6 (1--48), 5+ (49), 6 (50--51), 5+ (52), 6 (53--54), 5+ (55--106), 5 (107--110).
\end{small}
\bigskip

The information sets on which Banker mixes in the correlated cooperative equilibrium.\par\medskip

\tabcolsep=1.5mm
\begin{center}
\begin{small}
\begin{tabular}{cc}
\hline
\noalign{\smallskip}
intervals & Banker mixes on \\
\noalign{\smallskip} \hline
\noalign{\smallskip}
1, 11, 14--39 & $(8,10,6)$, $(10,10,6)$ \\
2--8 & $(5,10,6)$, $(10,10,6)$ \\
9--10 & $(3,6,6)$, $(10,10,6)$ \\
12 & $(0,4,5)$, $(10,10,6)$ \\
13 & $(5,1,4)$, $(10,10,6)$ \\
40--43 & $(10,0,4)$, $(10,10,6)$ \\
44 & $(9,11,3)$, $(10,0,4)$ \\
45--46 & $(6,10,6)$, $(10,0,4)$ \\
47 & $(6,10,6)$, $(9,7,5)$, $(10,0,4)$ \\
48 & $(6,10,6)$, $(7,1,4)$, $(10,0,4)$ \\
\noalign{\smallskip}
\hline
\end{tabular}
\end{small}
\end{center}
\tabcolsep=1.5mm
\begin{center}
\begin{small}
\begin{tabular}{cc}
\hline
\noalign{\smallskip}
intervals & Banker mixes on \\
\noalign{\smallskip} \hline
\noalign{\smallskip}
49, 52, 55 & $(6,10,6)$, $(10,0,4)$, $(11,10,6)$ \\
50, 53 & $(6,10,6)$, $(10,0,4)$, $(10,10,6)$ \\
51 & $(2,1,4)$, $(6,10,6)$, $(10,0,4)$ \\
54 & $(6,10,6)$, $(7,8,3)$, $(10,0,4)$ \\
56 & $(10,0,4)$, $(10,10,6)$, $(11,10,6)$ \\
57--59 & $(10,0,4)$, $(10,6,6)$, $(11,10,6)$ \\
60 & $(10,0,4)$, $(10,3,5)$, $(11,10,6)$ \\
61 & $(10,0,4)$, $(11,10,6)$ \\
62--66 & $(10,6,6)$, $(11,10,6)$ \\
67 & $(5,0,4)$, $(10,6,6)$, $(11,10,6)$ \\
68 & $(5,0,4)$, $(10,10,6)$, $(11,10,6)$ \\
69 & $(10,8,5)$, $(10,10,6)$, $(11,10,6)$ \\
70--73, 105, 106 & $(10,6,6)$, $(10,10,6)$, $(11,10,6)$ \\
74 & $(10,9,4)$, $(10,10,6)$, $(11,10,6)$ \\
75--97 & $(10,10,6)$, $(10,11,6)$, $(11,10,6)$ \\
98--101 & $(10,0,5)$, $(10,10,6)$, $(11,10,6)$ \\
102 & $(10,0,5)$, $(11,10,6)$ \\
103 & $(10,10,6)$, $(11,10,6)$ \\
104 & $(6,9,4)$, $(10,10,6)$, $(11,10,6)$ \\
107 & $(10,6,6)$, $(10,10,6)$ \\
108--110 & $(6,10,6)$, $(10,6,6)$ \\
\noalign{\smallskip}
\hline
\end{tabular}
\end{small}
\end{center}
\bigskip

\begin{small}
Banker's mixing probabilities in his unique behavioral strategy in the correlated cooperative equilibrium, listed by interval (1--110).  For each interval, we list above the two or three information sets on which Banker mixes when $\theta$ belongs to that interval.  Below we list the probabilities that Banker draws in the case of such information sets.\medskip

\noindent1. $\frac{925509+9561334 \,\theta-77056 \,\theta^2}{208 (4521+6014 \,\theta-8944 \,\theta^2)},\frac{134699760-246855595 \,\theta+388315886 \,\theta^2}{129792 (4521+6014 \,\theta-8944 \,\theta^2)}$\\
2. $\frac{-14859+8310422 \,\theta+1783296 \,\theta^2}{208 (4521+3224 \,\theta-5632 \,\theta^2)},\frac{134699760-328995211 \,\theta+196170944 \,\theta^2}{129792 (4521+3224 \,\theta-5632 \,\theta^2)}$\\
3. $\frac{-37003+8289206 \,\theta+1714944 \,\theta^2}{208 (4521+3224 \,\theta-5632 \,\theta^2)},\frac{134914784-332390955 \,\theta+198590656 \,\theta^2}{129792 (4521+3224 \,\theta-5632 \,\theta^2)}$\\
4. $\frac{-5595+5887638 \,\theta+2816512 \,\theta^2}{208 (4521+3224 \,\theta-5632 \,\theta^2)},\frac{39936400-50848473 \,\theta+20106304 \,\theta^2}{43264 (4521+3224 \,\theta-5632 \,\theta^2)}$\\
5. $\frac{-27739+5775030 \,\theta+2823424 \,\theta^2}{208 (4521+3224 \,\theta-5632 \,\theta^2)},\frac{120024224-147452171 \,\theta+55820480 \,\theta^2}{129792 (4521+3224 \,\theta-5632 \,\theta^2)}$\\
6. $\frac{7413+2960374 \,\theta+4334336 \,\theta^2}{208 (4521+3224 \,\theta-5632 \,\theta^2)},\frac{88171728+61959733 \,\theta-95270720 \,\theta^2}{129792 (4521+3224 \,\theta-5632 \,\theta^2)}$\\
7. $\frac{-36875+2987510 \,\theta+4170496 \,\theta^2}{208 (4521+3224 \,\theta-5632 \,\theta^2)},\frac{88601776+49923637 \,\theta-86095680 \,\theta^2}{129792 (4521+3224 \,\theta-5632 \,\theta^2)}$\\
8. $\frac{-59019+2943510 \,\theta+4147712 \,\theta^2}{208 (4521+3224 \,\theta-5632 \,\theta^2)},\frac{29605600+16575399 \,\theta-28742592 \,\theta^2}{43264 (4521+3224 \,\theta-5632 \,\theta^2)}$\\
9. $\frac{81163-2869046 \,\theta-4150016 \,\theta^2}{32 (692+2327 \,\theta-72 \,\theta^2)},\frac{27013779+93990356 \,\theta+110019202 \,\theta^2}{259584 (692+2327 \,\theta-72 \,\theta^2)}$\\
10. $\frac{76715-2745718 \,\theta-4163840 \,\theta^2}{32 (692+2327 \,\theta-72 \,\theta^2)},\frac{27084979+96520116 \,\theta+110209282 \,\theta^2}{259584 (692+2327 \,\theta-72 \,\theta^2)}$\\
11. $\frac{863653+3996630 \,\theta+2303488 \,\theta^2}{208 (4521+6014 \,\theta-8944 \,\theta^2)},\frac{89221216+110787989 \,\theta+27048622 \,\theta^2}{129792 (4521+6014 \,\theta-8944 \,\theta^2)}$\\
12. $\frac{-58923+2757110 \,\theta+3863808 \,\theta^2}{128 (139+89 \,\theta-2344 \,\theta^2)},\frac{5687143-6801968 \,\theta+7600314 \,\theta^2}{259584 (139+89 \,\theta-2344 \,\theta^2)}$\\
13. $\frac{-58923+2757110 \,\theta+3863808 \,\theta^2}{32 (1523+1752 \,\theta-496 \,\theta^2)},\frac{60426701+69954942 \,\theta-16078884 \,\theta^2}{259584 (1523+1752 \,\theta-496 \,\theta^2)}$\\
14. $\frac{832709+3951958 \,\theta+2019328 \,\theta^2}{208 (4521+6014 \,\theta-8944 \,\theta^2)},\frac{90219440+98383637 \,\theta+26847342 \,\theta^2}{129792 (4521+6014 \,\theta-8944 \,\theta^2)}$\\
15. $\frac{744133+4541654 \,\theta+1418240 \,\theta^2}{208 (4521+6014 \,\theta-8944 \,\theta^2)},\frac{91079536+25992085 \,\theta+86262894 \,\theta^2}{129792 (4521+6014 \,\theta-8944 \,\theta^2)}$\\
16. $\frac{695397+4520374 \,\theta+1420544 \,\theta^2}{208 (4521+6014 \,\theta-8944 \,\theta^2)},\frac{91320192+19915829 \,\theta+89824238 \,\theta^2}{129792 (4521+6014 \,\theta-8944 \,\theta^2)}$\\
17. $\frac{673253+4576406 \,\theta+1345536 \,\theta^2}{208 (4521+6014 \,\theta-8944 \,\theta^2)},\frac{91535216+10306933 \,\theta+96932270 \,\theta^2}{129792 (4521+6014 \,\theta-8944 \,\theta^2)}$\\
18. $\frac{651109+4540086 \,\theta+1343232 \,\theta^2}{208 (4521+6014 \,\theta-8944 \,\theta^2)},\frac{91750240+9140949 \,\theta+96233006 \,\theta^2}{129792 (4521+6014 \,\theta-8944 \,\theta^2)}$\\
19. $\frac{655557+4454902 \,\theta+1352448 \,\theta^2}{208 (4521+6014 \,\theta-8944 \,\theta^2)},\frac{91939632+15553621 \,\theta+88476142 \,\theta^2}{129792 (4521+6014 \,\theta-8944 \,\theta^2)}$\\
20. $\frac{633413+4388118 \,\theta+1375232 \,\theta^2}{208 (4521+6014 \,\theta-8944 \,\theta^2)},\frac{92154656+17217301 \,\theta+85194926 \,\theta^2}{129792 (4521+6014 \,\theta-8944 \,\theta^2)}$\\
21. $\frac{664821+1699094 \,\theta+3092480 \,\theta^2}{208 (4521+6014 \,\theta-8944 \,\theta^2)},\frac{77049072+231822965 \,\theta-113576530 \,\theta^2}{129792 (4521+6014 \,\theta-8944 \,\theta^2)}$\\
22. $\frac{616085+1620246 \,\theta+3153920 \,\theta^2}{208 (4521+6014 \,\theta-8944 \,\theta^2)},\frac{77289728+231567317 \,\theta-114964050 \,\theta^2}{129792 (4521+6014 \,\theta-8944 \,\theta^2)}$\\
23. $\frac{620533+1756950 \,\theta+3010560 \,\theta^2}{208 (4521+6014 \,\theta-8944 \,\theta^2)},\frac{77479120+218587061 \,\theta-104094290 \,\theta^2}{129792 (4521+6014 \,\theta-8944 \,\theta^2)}$\\
24. $\frac{598389+1881590 \,\theta+2905856 \,\theta^2}{208 (4521+6014 \,\theta-8944 \,\theta^2)},\frac{77694144+203687477 \,\theta-91330834 \,\theta^2}{129792 (4521+6014 \,\theta-8944 \,\theta^2)}$\\
25. $\frac{576245+1975766 \,\theta+2826240 \,\theta^2}{208 (4521+6014 \,\theta-8944 \,\theta^2)},\frac{77909168+191617557 \,\theta-81149330 \,\theta^2}{129792 (4521+6014 \,\theta-8944 \,\theta^2)}$\\
26. $\frac{554101+1977590 \,\theta+2819328 \,\theta^2}{208 (4521+6014 \,\theta-8944 \,\theta^2)},\frac{78124192+187990549 \,\theta-78775122 \,\theta^2}{129792 (4521+6014 \,\theta-8944 \,\theta^2)}$\\
27. $\frac{554101+2050806 \,\theta+2766080 \,\theta^2}{208 (4521+6014 \,\theta-8944 \,\theta^2)},\frac{78124192+181474325 \,\theta-72611666 \,\theta^2}{129792 (4521+6014 \,\theta-8944 \,\theta^2)}$\\
28. $\frac{505365+2071990 \,\theta+2775296 \,\theta^2}{208 (4521+6014 \,\theta-8944 \,\theta^2)},\frac{78364848+173144405 \,\theta-65700370 \,\theta^2}{129792 (4521+6014 \,\theta-8944 \,\theta^2)}$\\
29. $\frac{509813+2024950 \,\theta+2779904 \,\theta^2}{208 (4521+6014 \,\theta-8944 \,\theta^2)},\frac{78554240+177096053 \,\theta-70383762 \,\theta^2}{129792 (4521+6014 \,\theta-8944 \,\theta^2)}$\\
30. $\frac{466453+1919990 \,\theta+2800384 \,\theta^2}{208 (4521+6014 \,\theta-8944 \,\theta^2)},\frac{79486960+178520853 \,\theta-75616402 \,\theta^2}{129792 (4521+6014 \,\theta-8944 \,\theta^2)}$\\
31. $\frac{423093+1941206 \,\theta+2732032 \,\theta^2}{208 (4521+6014 \,\theta-8944 \,\theta^2)},\frac{26806560+56278119 \,\theta-23414918 \,\theta^2}{43264 (4521+6014 \,\theta-8944 \,\theta^2)}$\\
32. $\frac{400949+1851638 \,\theta+2800384 \,\theta^2}{208 (4521+6014 \,\theta-8944 \,\theta^2)},\frac{80634704+173696341 \,\theta-75616402 \,\theta^2}{129792 (4521+6014 \,\theta-8944 \,\theta^2)}$\\
33. $\frac{352213+1750006 \,\theta+2907392 \,\theta^2}{208 (4521+6014 \,\theta-8944 \,\theta^2)},\frac{80875360+176638997 \,\theta-79094354 \,\theta^2}{129792 (4521+6014 \,\theta-8944 \,\theta^2)}$\\
34. $\frac{308853+1683190 \,\theta+2923264 \,\theta^2}{208 (4521+6014 \,\theta-8944 \,\theta^2)},\frac{81808080+175602773 \,\theta-81253522 \,\theta^2}{129792 (4521+6014 \,\theta-8944 \,\theta^2)}$\\
35. $\frac{135413+2303478 \,\theta+2376448 \,\theta^2}{208 (4521+6014 \,\theta-8944 \,\theta^2)},\frac{85538960+93886165 \,\theta-15383698 \,\theta^2}{129792 (4521+6014 \,\theta-8944 \,\theta^2)}$\\
36. $\frac{92053+2175734 \,\theta+2442496 \,\theta^2}{208 (4521+6014 \,\theta-8944 \,\theta^2)},\frac{86471680+98509269 \,\theta-22706770 \,\theta^2}{129792 (4521+6014 \,\theta-8944 \,\theta^2)}$\\
37. $\frac{-81387+2169462 \,\theta+2351360 \,\theta^2}{208 (4521+6014 \,\theta-8944 \,\theta^2)},\frac{90202560+72556501 \,\theta-9581906 \,\theta^2}{129792 (4521+6014 \,\theta-8944 \,\theta^2)}$\\
38. $\frac{-169963+2698742 \,\theta+2014464 \,\theta^2}{208 (4521+6014 \,\theta-8944 \,\theta^2)},\frac{91062656+14432725 \,\theta+43437998 \,\theta^2}{129792 (4521+6014 \,\theta-8944 \,\theta^2)}$\\
39. $\frac{-213323+2670070 \,\theta+2025728 \,\theta^2}{208 (4521+6014 \,\theta-8944 \,\theta^2)},\frac{91995376+10935477 \,\theta+44352302 \,\theta^2}{129792 (4521+6014 \,\theta-8944 \,\theta^2)}$\\
40. $\frac{213323-2670070 \,\theta-2025728 \,\theta^2}{2496 (481-5570 \,\theta+3056 \,\theta^2)},\frac{1658359-13283949 \,\theta-29056090 \,\theta^2}{43264 (481-5570 \,\theta+3056 \,\theta^2)}$\\
41. $\frac{208875-2783990 \,\theta-1927936 \,\theta^2}{2496 (481-5570 \,\theta+3056 \,\theta^2)},\frac{5095573-39106247 \,\theta-86691534 \,\theta^2}{129792 (481-5570 \,\theta+3056 \,\theta^2)}$\\
42. $\frac{92615 - 2230062 \,\theta - 667776 \,\theta^2}{2496 (481 - 5454 \,\theta + 4584 \,\theta^2)}, \frac{3988037 - 23548455 \,\theta - 57221838 \,\theta^2}{86528 (481 - 5454 \,\theta + 4584 \,\theta^2)}$\\
43. $\frac{88455-2298382 \,\theta-554496 \,\theta^2}{2496 (481-5454 \,\theta+4584 \,\theta^2)},\frac{3803957-20561255 \,\theta-58440718 \,\theta^2}{86528 (481-5454 \,\theta+4584 \,\theta^2)}$\\
44. $\frac{-3803957+20561255 \,\theta+58440718 \,\theta^2}{2342912\,\theta (3-\theta)},\frac{227 (461+691 \,\theta)}{79872 (3-\theta)}$\\
45. $\frac{11573381-103292423 \,\theta+22812402 \,\theta^2}{624 (12451-143846 \,\theta+133968 \,\theta^2)},\frac{-5048864-202281163 \,\theta+233361096 \,\theta^2}{7488 (12451-143846 \,\theta+133968 \,\theta^2)}$\\
46. $\frac{11977301-106978183 \,\theta+26405042 \,\theta^2}{624 (12451-143846 \,\theta+133968 \,\theta^2)},\frac{-5332784-210309803 \,\theta+241784776 \,\theta^2}{7488 (12451-143846 \,\theta+133968 \,\theta^2)}$\\
47. $\frac{-48475240-541023879 \,\theta+2712557322 \,\theta^2-754816328 \,\theta^3}{624 (89735-1705504 \,\theta+2143084 \,\theta^2-760352 \,\theta^3)},$\\
\hglue6mm$\frac{-213242455+1046025684 \,\theta+2315452986 \,\theta^2-2961237104 \,\theta^3}{48 (89735-1705504 \,\theta+2143084 \,\theta^2-760352 \,\theta^3)},$\\
\hglue6mm$\frac{370053385-6266617959 \,\theta-2641394628 \,\theta^2+7018562336 \,\theta^3}{7488 (89735-1705504 \,\theta+2143084 \,\theta^2-760352 \,\theta^3)}$\\
48. $\frac{-123198092+1091041883 \,\theta+359205050 \,\theta^2-1475858216 \,\theta^3}{208 (69459+2238452 \,\theta-5258404 \,\theta^2+3026528 \,\theta^3)},$\\
\hglue6mm$ \frac{213242455-1046025684 \,\theta-2315452986 \,\theta^2+2961237104 \,\theta^3}{16 (69459+2238452 \,\theta-5258404 \,\theta^2+3026528 \,\theta^3)},$\\
\hglue6mm$\frac{95479423+2816585283 \,\theta-719165568 \,\theta^2-1776998848 \,\theta^3}{2496 (69459+2238452 \,\theta-5258404 \,\theta^2+3026528 \,\theta^3)}$\\
49. $\frac{-54928-275947 \,\theta+2744000 \,\theta^2}{208 (481-4514 \,\theta+2800 \,\theta^2)},\frac{142352-2834323 \,\theta+1001672 \,\theta^2}{2496 (481-4514 \,\theta+2800 \,\theta^2)}, $\\
\hglue6mm$\frac{671062103-5800306644 \,\theta+4646070726 \,\theta^2+296174704 \,\theta^3}{951808 (1-\theta) (481-4514 \,\theta+2800 \,\theta^2)}$\\
50. $\frac{130465959-1495372045 \,\theta+1707815974 \,\theta^2-1150033920 \,\theta^3}{208 (724867-8403814 \,\theta+8914512 \,\theta^2-1173504 \,\theta^3)}, $\\
\hglue6mm$\frac{214524464-6107119193 \,\theta+3148582296 \,\theta^2+2829090816 \,\theta^3}{2496 (724867-8403814 \,\theta+8914512 \,\theta^2-1173504 \,\theta^3)},$\\
\hglue6mm$\frac{213242455-1046025684 \,\theta-2315452986 \,\theta^2+2961237104 \,\theta^3}{2704 (724867-8403814 \,\theta+8914512 \,\theta^2-1173504 \,\theta^3)}$\\               
51. $\frac{213242455-1046025684 \,\theta-2315452986 \,\theta^2+2961237104 \,\theta^3}{48 (23153-949486 \,\theta+1190332 \,\theta^2-355104 \,\theta^3)},$\\
\hglue6mm$ \frac{-123198092+120283463 \,\theta+2322139040 \,\theta^2-1784315036 \,\theta^3}{624 (23153-949486 \,\theta+1190332 \,\theta^2-355104 \,\theta^3)},$\\
\hglue6mm$\frac{95479423-3667499087 \,\theta-1403517698 \,\theta^2+4289189232 \,\theta^3}{7488 (23153-949486 \,\theta+1190332 \,\theta^2-355104 \,\theta^3)}$\\               
52. $\frac{-59088-260587 \,\theta+2732800 \,\theta^2}{208 (481-4514 \,\theta+2800 \,\theta^2)},\frac{145056-2940563 \,\theta+1072328 \,\theta^2}{2496 (481-4514 \,\theta+2800 \,\theta^2)}, $\\
\hglue6mm$\frac{669950759-5754731316 \,\theta+4588934790 \,\theta^2+313219696 \,\theta^3}{951808 (1-\theta) (481-4514 \,\theta+2800 \,\theta^2)}$\\               
53. $\frac{123085495-1411673325 \,\theta+1642838182 \,\theta^2-1145339904 \,\theta^3}{208 (724867-8403814 \,\theta+8914512 \,\theta^2-1173504 \,\theta^3)},$\\
\hglue6mm$ \frac{218599392-6269126489 \,\theta+3319081368 \,\theta^2+2818179072 \,\theta^3}{2496 (724867-8403814 \,\theta+8914512 \,\theta^2-1173504 \,\theta^3)},$\\
\hglue6mm$\frac{212131111-1000450356 \,\theta-2372588922 \,\theta^2+2978282096 \,\theta^3}{2704 (724867-8403814 \,\theta+8914512 \,\theta^2-1173504 \,\theta^3)}$\\               
54. $\frac{-47315718+692994547 \,\theta-346341660 \,\theta^2-448146544 \,\theta^3}{208 (62061+2369578 \,\theta-4947956 \,\theta^2+2592352 \,\theta^3)},$\\
\hglue6mm$ \frac{213124087-962537108 \,\theta-2451756218 \,\theta^2+3019759728 \,\theta^3}{16 (62061 + 2369578\,\theta - 4947956\,\theta^2 + 2592352\,\theta^3)},$\\
\hglue6mm$\frac{85310017+3044468687 \,\theta-376267422 \,\theta^2-2342871792 \,\theta^3}{2496 (62061+2369578 \,\theta-4947956 \,\theta^2+2592352 \,\theta^3)}$\\               
55. $\frac{-57648-306827 \,\theta+2777600 \,\theta^2}{208 (481-4514 \,\theta+2800 \,\theta^2)},\frac{142640-3021203 \,\theta+1122504 \,\theta^2}{2496 (481-4514 \,\theta+2800 \,\theta^2)}, $\\
\hglue6mm$\frac{670943735-5716818068 \,\theta+4509767494 \,\theta^2+354697328 \,\theta^3}{951808 (1-\theta) (481-4514 \,\theta+2800 \,\theta^2)}$\\               
56. $\frac{142640-4173075 \,\theta-2408448 \,\theta^2}{2496 (481-5002 \,\theta)},\frac{157696-632085 \,\theta-2195200 \,\theta^2}{2704 (481-5002 \,\theta)}, $\\
\hglue6mm$\frac{433295863-4920014797 \,\theta+4607208550 \,\theta^2-920027136 \,\theta^3}{951808 (1-\theta) (481-5002 \,\theta)}$\\               
57. $\frac{-22816+4461109 \,\theta+5350200 \,\theta^2}{24960 (1-\theta) (9+280 \,\theta)},\frac{-157696+632085 \,\theta+2195200 \,\theta^2}{2080 (1-\theta) (9+280 \,\theta)}, $\\
\hglue6mm$\frac{40537035+1205694466 \,\theta-2661114501 \,\theta^2+1058937800 \,\theta^3}{4759040 (1-\theta)^2 (9+280 \,\theta)}$\\               
58. $\frac{-22816+4543029 \,\theta+5268280 \,\theta^2}{24960 (1-\theta) (9+280 \,\theta)},\frac{-159136+678325 \,\theta+2150400 \,\theta^2}{2080 (1-\theta) (9+280 \,\theta)}, $\\
\hglue6mm$\frac{40241115+1197439186 \,\theta-2648071781 \,\theta^2+1054446280 \,\theta^3}{4759040 (1-\theta)^2 (9+280 \,\theta)}$\\               
59. $\frac{-23616+4698709 \,\theta+5113400 \,\theta^2}{24960 (1-\theta) (9+280 \,\theta)},\frac{3 (-54432+249895 \,\theta+694400 \,\theta^2)}{2080 (1-\theta) (9+280 \,\theta)}, $\\
\hglue6mm$\frac{39550635+1178207266 \,\theta-2614516101 \,\theta^2+1040813000 \,\theta^3}{4759040 (1-\theta)^2 (9+280 \,\theta)}$\\               
60. $\frac{442656+1283147 \,\theta-835940 \,\theta^2}{49920 (1-\theta) (9+70 \,\theta)},\frac{3 (-54432+249895 \,\theta+694400 \,\theta^2)}{4160 (1-\theta) (9+70 \,\theta)}, $\\
\hglue6mm$\frac{39550635+237069964 \,\theta-635395209 \,\theta^2+216396530 \,\theta^3}{4759040 (1-\theta)^2 (9+70 \,\theta)}$\\               
61. $\frac{-3 (2208+25105 \,\theta)}{29120 (1-10 \,\theta)},\frac{40396933-367661399 \,\theta+234849010 \,\theta^2}{33313280 (1-\theta) (1-10 \,\theta)}$\\               
62. $\frac{3 (35744-215885 \,\theta)}{1040 (1-232 \,\theta)},\frac{19006483-923418635 \,\theta+709668280 \,\theta^2}{4759040 (1-\theta) (1-232 \,\theta)}$\\               
63. $\frac{9 (15728-97615 \,\theta)}{1040 (1-232 \,\theta)},\frac{24298003-1038434315 \,\theta+871808440 \,\theta^2}{4759040 (1-\theta) (1-232 \,\theta)}$\\               
64. $\frac{3 (46944-295405 \,\theta)}{1040 (1-232 \,\theta)},\frac{24158723-1031268235 \,\theta+868838840 \,\theta^2}{4759040 (1-\theta) (1-232 \,\theta)}$\\               
65. $\frac{140752-867655 \,\theta}{1040 (1-232 \,\theta)},\frac{24125843-1023613035 \,\theta+862565560 \,\theta^2}{4759040 (1-\theta) (1-232 \,\theta)}$\\               
66. $\frac{140032-868615 \,\theta}{1040 (1-232 \,\theta)},\frac{23986563-1015474955 \,\theta+858593720 \,\theta^2}{4759040 (1-\theta) (1-232 \,\theta)}$\\               
67. $\frac{411136+609771 \,\theta-10809400 \,\theta^2}{640 (1-\theta) (27+1160 \,\theta)},\frac{-1072128+8868415 \,\theta-8686150 \,\theta^2}{2080 (1-\theta) (27+1160 \,\theta)}, $\\
\hglue6mm$\frac{-277320798+9380893649 \,\theta-9650377190 \,\theta^2+5126929200 \,\theta^3}{9518080 (1-\theta)^2 (27+1160 \,\theta)}$\\               
68. $\frac{1566288-7508941 \,\theta+1916928 \,\theta^2}{64 (1443-3776 \,\theta+3840 \,\theta^2)},\frac{1128288-6511775 \,\theta+6273350 \,\theta^2}{2704 (1443-3776 \,\theta+3840 \,\theta^2)}, $\\
\hglue6mm$\frac{3 (-234698601+2106617848 \,\theta-1302007712 \,\theta^2+60047360 \,\theta^3)}{951808 (1-\theta) (1443-3776 \,\theta+3840 \,\theta^2)}$\\                                                       
69. $\frac{-1473936 + 7267277 \,\theta - 1671168 \,\theta^2}{208 (1591 + 2220 \,\theta - 2304 \,\theta^2)}, \frac{-1023584 +5938577 \,\theta - 4025130 \,\theta^2}{2704 (1591 + 2220 \,\theta - 2304 \,\theta^2)}, $\\
\hglue6mm$\frac{1227448545 + 172448755 \,\theta - 3186435220 \,\theta^2 + 968146944 \,\theta^3}{951808 (1 -\theta) (1591 +2220 \,\theta - 2304 \,\theta^2)}$\\               
70. $\frac{1836272-9645965 \,\theta+552960 \,\theta^2}{208 (151-13656 \,\theta-3072 \,\theta^2)},\frac{504096-1573429 \,\theta-8719160 \,\theta^2}{2704 (151-13656 \,\theta-3072 \,\theta^2)}, $\\
\hglue6mm$\frac{116495745-11182282057 \,\theta+12956078952 \,\theta^2-810123264 \,\theta^3}{951808 (1-\theta) (151-13656 \,\theta-3072 \,\theta^2)}$\\               
71. $\frac{1810816-9681485 \,\theta+589824 \,\theta^2}{208 (151-13656 \,\theta-3072 \,\theta^2)},\frac{495616-1461269 \,\theta-8822840 \,\theta^2}{2704 (151-13656 \,\theta-3072 \,\theta^2)},$\\
\hglue6mm$\frac{118812689-11388664009 \,\theta+13269797736 \,\theta^2-795672576 \,\theta^3}{951808 (1-\theta) (151-13656 \,\theta-3072 \,\theta^2)}$\\               
72. $\frac{1718464-9296077 \,\theta+393216 \,\theta^2}{208 (151-13656 \,\theta-3072 \,\theta^2)},\frac{478336-2145429 \,\theta-8121400 \,\theta^2}{2704 (151-13656 \,\theta-3072 \,\theta^2)}, $\\
\hglue6mm$\frac{106896977-10274714825 \,\theta+12423289704 \,\theta^2-603389952 \,\theta^3}{951808 (1-\theta) (151-13656 \,\theta-3072 \,\theta^2)}$\\               
73. $\frac{1707984-9236909 \,\theta+368640 \,\theta^2}{208 (151-13656 \,\theta-3072 \,\theta^2)},\frac{476576-2237109 \,\theta-8027960 \,\theta^2}{2704 (151-13656 \,\theta-3072 \,\theta^2)}, $\\
\hglue6mm$\frac{105242017-10120590057 \,\theta+12298962024 \,\theta^2-580583424 \,\theta^3}{951808 (1-\theta) (151-13656 \,\theta-3072 \,\theta^2)}$\\               
74. $\frac{-1676576+6396461 \,\theta-1007616 \,\theta^2}{624 (799-4296 \,\theta-1024 \,\theta^2)},\frac{459232-440039 \,\theta-8027960 \,\theta^2}{8112 (799-4296 \,\theta-1024 \,\theta^2)}, $\\
\hglue6mm$\frac{556876633-3740462513 \,\theta+3474862232 \,\theta^2-96796672 \,\theta^3}{951808 (1-\theta) (799-4296 \,\theta-1024 \,\theta^2)}$\\               
75. $\frac{15 (5+4 \,\theta)}{1859},\frac{-2175152+9077165 \,\theta-368640 \,\theta^2}{10469888 \,\theta},\frac{7666637-9664557 \,\theta+368640 \,\theta^2}{10469888 (1-\theta)}$\\               
76. $\frac{68 (1+\theta)}{1859},\frac{-2149696+9076749 \,\theta-417792 \,\theta^2}{10469888 \,\theta},\frac{3 (2611807-3318927 \,\theta+139264 \,\theta^2)}{10469888 (1-\theta)}$\\               
77. $\frac{71+62 \,\theta}{1859},\frac{-2131440+9093965 \,\theta-380928 \,\theta^2}{10469888 \,\theta},\frac{3 (2587695-3297999 \,\theta+126976 \,\theta^2)}{10469888 (1-\theta)}$\\               
78. $\frac{2 (3+70 \,\theta)}{1859},\frac{-2633760+9762061 \,\theta-860160 \,\theta^2}{10469888 \,\theta},\frac{7763085-10805869 \,\theta+860160 \,\theta^2}{10469888 (1-\theta)}$\\               
79. $\frac{21+146 \,\theta}{1859},\frac{-2602320+9261133 \,\theta-897024 \,\theta^2}{10469888 \,\theta},\frac{7401405-10557421 \,\theta+897024 \,\theta^2}{10469888 (1-\theta)}$\\               
80. $\frac{41+90 \,\theta}{1859},\frac{-2560400+9743181 \,\theta-552960 \,\theta^2}{10469888 \,\theta},\frac{6919165-9823469 \,\theta+552960 \,\theta^2}{10469888 (1-\theta)}$\\               
81. $\frac{2 (18+47 \,\theta)}{1859},\frac{-2542720+9774189 \,\theta-577536 \,\theta^2}{10469888 \,\theta},\frac{7039725-10006957 \,\theta+577536 \,\theta^2}{10469888 (1-\theta)}$\\               
82. $\frac{37+92 \,\theta}{1859},\frac{-2504368+9747661 \,\theta-565248 \,\theta^2}{10469888 \,\theta},\frac{7015613-10052909 \,\theta+565248 \,\theta^2}{10469888 (1-\theta)}$\\               
83. $\frac{8 (5+11 \,\theta)}{1859},\frac{-2486112+9728941 \,\theta-540672 \,\theta^2}{10469888 \,\theta},\frac{6943277-9994349 \,\theta+540672 \,\theta^2}{10469888 (1-\theta)}$\\               
84. $\frac{4 (19+10 \,\theta)}{1859},\frac{-2393760+9631021 \,\theta-245760 \,\theta^2}{10469888 \,\theta},\frac{6075245-9017069 \,\theta+245760 \,\theta^2}{10469888 (1-\theta)}$\\               
85. $\frac{81+40 \,\theta}{1859},\frac{-2202000+9318701 \,\theta-245760 \,\theta^2}{10469888 \,\theta},\frac{5954685-9267949 \,\theta+245760 \,\theta^2}{10469888 (1-\theta)}$\\               
86. $\frac{67+48 \,\theta}{1859},\frac{-2151088+9461613 \,\theta-294912 \,\theta^2}{10469888 \,\theta},\frac{6292253-9747501 \,\theta+294912 \,\theta^2}{10469888 (1-\theta)}$\\               
87. $\frac{8 (9+5 \,\theta)}{1859},\frac{-2140608+9474317 \,\theta-245760 \,\theta^2}{10469888 \,\theta},\frac{6171693-9596045 \,\theta+245760 \,\theta^2}{10469888 (1-\theta)}$\\               
88. $\frac{-19+144 \,\theta}{1859},\frac{-2907920+10567149 \,\theta-884736 \,\theta^2}{10469888 \,\theta},\frac{3 (2057231-3555887 \,\theta+294912 \,\theta^2)}{10469888 (1-\theta)}$\\               
89. $\frac{17+120 \,\theta}{1859},\frac{-2815568+10037997 \,\theta-737280 \,\theta^2}{10469888 \,\theta},\frac{3 (1767887-3268143 \,\theta+245760 \,\theta^2)}{10469888 (1-\theta)}$\\               
90. $\frac{2 (9+64 \,\theta)}{1859},\frac{-2777216+9831789 \,\theta-786432 \,\theta^2}{10469888\,\theta},\frac{5279549-9890861 \,\theta+786432 \,\theta^2}{10469888 (1-\theta)}$\\               
91. $\frac{30 (1+4 \,\theta)}{1859},\frac{-2591552+9500525 \,\theta-737280 \,\theta^2}{10469888 \,\theta},\frac{4990205-9849517 \,\theta+737280 \,\theta^2}{10469888 (1-\theta)}$\\               
92. $\frac{147+16 \,\theta}{1859},\frac{-3389648+9346189 \,\theta-98304 \,\theta^2}{10469888 \,\theta},\frac{4990205-9643181 \,\theta+98304 \,\theta^2}{10469888 (1-\theta)}$\\               
93. $\frac{4 (37+2 \,\theta)}{1859},\frac{-3351296+9427469 \,\theta-49152 \,\theta^2}{10469888 \,\theta},\frac{4966093-9628941 \,\theta+49152 \,\theta^2}{10469888 (1-\theta)}$\\               
94. $\frac{141+14 \,\theta}{1859},\frac{-3325840+9462989 \,\theta-86016 \,\theta^2}{10469888 \,\theta},\frac{5134877-9872941 \,\theta+86016 \,\theta^2}{10469888 (1-\theta)}$\\               
95. $\frac{2 (67+12 \,\theta)}{1859},\frac{-3300384+9426637 \,\theta-147456 \,\theta^2}{10469888 \,\theta},\frac{5303661-10141229 \,\theta+147456 \,\theta^2}{10469888 (1-\theta)}$\\               
96. $\frac{3 (43+12 \,\theta)}{1859},\frac{-3282704+9313901 \,\theta-221184 \,\theta^2}{10469888 \,\theta},\frac{5424221-10360973 \,\theta+221184 \,\theta^2}{10469888 (1-\theta)}$\\               
97. $\frac{6 (22+7 \,\theta)}{1859},\frac{-3236288+9087341 \,\theta-258048 \,\theta^2}{10469888 \,\theta},\frac{5351885-10386893 \,\theta+258048 \,\theta^2}{10469888 (1-\theta)}$\\               
98. $\frac{-3236288+9087341 \,\theta-258048 \,\theta^2}{832 (4643-3136 \,\theta)},\frac{597248+507655 \,\theta-215040 \,\theta^2}{2704 (4643-3136 \,\theta)}, $\\
\hglue6mm$\frac{2258982005-5298034229 \,\theta+1351794912 \,\theta^2-24772608 \,\theta^3}{951808 (1-\theta) (4643-3136 \,\theta)}$\\               
99. $\frac{-3165568+9786349 \,\theta-159744 \,\theta^2}{832 (4643-3136 \,\theta)},\frac{468608+554375 \,\theta-133120 \,\theta^2}{2704 (4643-3136 \,\theta)}, $\\
\hglue6mm$\frac{2462531125-5649925813 \,\theta+1313911008 \,\theta^2-15335424 \,\theta^3}{951808 (1-\theta) (4643-3136 \,\theta)}$\\               
100. $\frac{-3119152+9847277 \,\theta-98304 \,\theta^2}{832 (4643-3136 \,\theta)},\frac{493088+478695 \,\theta-81920 \,\theta^2}{2704 (4643-3136 \,\theta)}, $\\
\hglue6mm$\frac{2431998757-5606528181 \,\theta+1255145696 \,\theta^2-9437184 \,\theta^3}{951808 (1-\theta) (4643-3136 \,\theta)}$\\               
101. $\frac{-3150560+12687725 \,\theta+540672 \,\theta^2}{832 (4643-3136 \,\theta)},\frac{226848+212455 \,\theta+450560 \,\theta^2}{2704 (4643-3136 \,\theta)}, $\\
\hglue6mm$\frac{2431998757-5330883253 \,\theta+266170592 \,\theta^2+51904512 \,\theta^3}{951808 (1-\theta) (4643-3136 \,\theta)}$\\               
102. $\frac{-289952+1199975 \,\theta}{4160 (85-64 \,\theta)},\frac{229873711-506046847 \,\theta+54464160 \,\theta^2}{4759040 (1-\theta) (85-64 \,\theta)}$\\               
103. $\frac{643552-1466215 \,\theta}{86528 (1-12 \,\theta)},\frac{-6489737+15582418 \,\theta-3235968 \,\theta^2}{951808 (1-\theta) (1-12 \,\theta)}$\\               
104. $\frac{7013536-15296877 \,\theta-540672 \,\theta^2}{48 (799+196 \,\theta+512 \,\theta^2)},\frac{1083872-4138343 \,\theta+5724060 \,\theta^2}{8112 (799+196 \,\theta+512 \,\theta^2)}, $\\
\hglue6mm$\frac{294691771-94558498 \,\theta-1051044150 \,\theta^2+441237504 \,\theta^3}{2855424 (1-\theta) (799+196 \,\theta+512 \,\theta^2)}$\\               
105. $\frac{6975184-15306285 \,\theta-565248 \,\theta^2}{208 (151-13656 \,\theta-3072 \,\theta^2)},\frac{3 (643552-561383 \,\theta-3345000 \,\theta^2)}{2704 (151-13656 \,\theta-3072 \,\theta^2)},$\\
\hglue6mm$ \frac{3 (26254219-3078030947 \,\theta+6452715384 \,\theta^2-32636928 \,\theta^3)}{951808 (1-\theta) (151-13656 \,\theta-3072 \,\theta^2)}$\\               
106. $\frac{6936544-15290829 \,\theta-614400 \,\theta^2}{208 (151-13656 \,\theta-3072 \,\theta^2)},\frac{1921216-1809109 \,\theta-9900600 \,\theta^2}{2704 (151-13656 \,\theta-3072 \,\theta^2)},$\\
\hglue6mm$ \frac{77107697-9071939465 \,\theta+19222966632 \,\theta^2-53280768 \,\theta^3}{951808 (1-\theta) (151-13656 \,\theta-3072 \,\theta^2)}$\\               
107. $\frac{6425897-15661618 \,\theta+3299328 \,\theta^2}{208 (151-14100 \,\theta+15456 \,\theta^2)},\frac{64032147-900502651 \,\theta+1573877560 \,\theta^2}{86528 (151-14100 \,\theta+15456 \,\theta^2)}$\\               
108. $\frac{64594371-1176979579 \,\theta+1850354488 \,\theta^2}{624 (901-443072 \,\theta+443072 \,\theta^2)},\frac{754077120-2548457077 \,\theta+1857619768 \,\theta^2}{624 (901-443072 \,\theta+443072 \,\theta^2)}$\\               
109. $\frac{64273971-1182223259 \,\theta+1859945528 \,\theta^2}{624 (901-443072 \,\theta+443072 \,\theta^2)},\frac{750080400-2542979957 \,\theta+1856478008 \,\theta^2}{624 (901-443072 \,\theta+443072 \,\theta^2)}$\\               
110. $\frac{21311777-393439433 \,\theta+620812136 \,\theta^2}{208 (901-443072 \,\theta+443072 \,\theta^2)},\frac{248684480-848184839 \,\theta+620812136 \,\theta^2}{208 (901-443072 \,\theta+443072 \,\theta^2)}$
\end{small}\bigskip

\begin{small}
Players' joint strategy in the correlated cooperative equilibrium. ($p_{00}=P(\text{both}$ $\text{Players stand on 5})$, $p_{01}=P(\text{Player~1 stands on 5, Player~2 draws on 5})$, $p_{10}=\break P(\text{Player~1 draws on 5, Player~2 stands on 5})$, and $p_{11}=P(\text{both Players draw on 5})$.)  This strategy is included for completeness, but it cannot be recommended because it is illegal for the Players to collaborate during the course of play.
\end{small}\bigskip

\begin{small}
\noindent(intervals 1, 11, 14--39)\quad $p_{00} = \frac{2 (411 - 2666 \, \theta + 664 \,\theta^2)}{4521 + 6014 \,\theta - 8944 \,\theta^2}, \quad p_{01} = \frac{3 (1233 + 2070 \,\theta - 1712 \,\theta^2)}{4521 + 6014 \, \theta - 8944 \, \theta^2},$\\
$p_{10} = \frac{5136 \,\theta(1 - \theta)}{4521 +6014 \, \theta - 8944 \, \theta^2}, \quad p_ {11} = 0$.\\

\noindent(intervals 2--8)\quad $p_{00} = \frac{2 (411-1820 \, \theta - 704 \, \theta^2)}{4521 + 3224 \, \theta - 5632 \, \theta^2}, \quad
p_{01} = \frac{3 (1233 + 1584 \, \theta - 704 \, \theta^2)}{4521 + 3224 \, \theta - 5632 \, \theta^2},$\\
$p_{10} = \frac{2112 \,\theta (1 -  \theta)}{4521 + 3224 \, \theta - 5632 \, \theta^2}, \quad p_ {11} = 0$.\\

\noindent(intervals 9--10)\quad $p_ {00} = \frac{15 (115 - 3562 \, \theta + 6104 \, \theta^2)}{4 (692+ 2327 \, \theta - 72 \, \theta^2)}, \quad p_ {01} = \frac{21783 - 69220 \,\theta - 367392 \, \theta^2}{32 (692 + 2327 \, \theta - 72 \, \theta^2)},$\\
$p_ {10} = \frac{-13439 + 571124 \,\theta - 367392 \, \theta^2}{32 (692 + 2327 \, \theta - 72 \, \theta^2)}, \quad p_ {11} = 0$.\\

\noindent(interval 12)\quad $p_ {00} = \frac{5 (-195 +8774 \, \theta - 16400 \, \theta^2)}{4 (139 + 89 \, \theta - 2344 \, \theta^2 )}, \quad
p_ {01} = \frac{411 + 122320 \, \theta + 290496 \, \theta^2}{32 (139 + 89 \, \theta - 2344 \, \theta^2)},$\\
$p_ {10} = \frac{11837 - 470432 \, \theta + 290496 \, \theta^2}{32 (139 + 89 \, \theta - 2344 \, \theta^2)}, \quad p_ {11} = 0$.\\

\noindent(interval 13)\quad $p_ {00} = \frac{5 (495 - 12298 \, \theta + 20536 \, \theta^2)}{4 (1523 + 1752 \, \theta - 496 \, \theta^2)}, \quad
p_ {01} = \frac{43977 - 83404 \, \theta - 418656 \, \theta^2}{32 (1523 + 1752 \, \theta - 496 \, \theta^2)},$\\
$p_ {10} = \frac{-15041 + 631388 \, \theta - 418656 \, \theta^2}{32 (1523 + 1752 \, \theta - 496 \, \theta^2)}, \quad
p_ {11} = 0$.\\

\noindent(intervals 40--41)\quad $p_ {00} = \frac{2 (-481 +4126 \, \theta - 5448 \, \theta^2)}{481 - 5570 \, \theta + 3056 \, \theta^2}, \quad
p_ {01} = \frac{2 \,\theta (-1263 + 3488 \, \theta)}{481 - 5570 \, \theta + 3056 \, \theta^2},$\\
$p_ {10} = \frac{1443 - 11296 \, \theta + 6976 \, \theta^2}{481 - 5570 \, \theta + 3056 \, \theta^2}, \quad p_ {11} = 0$.\\

\noindent(intervals 42--43)\quad $p_ {00} = 0, \quad p_{01} =\frac{2 (-481+ 3147 \,\theta - 2940 \,\theta^2)}{481 - 5454 \,\theta + 4584 \,\theta^2},$\\
$p_{10}=\frac{210 \,\theta (3 - 28 \,\theta)}{481 - 5454 \,\theta + 4584 \,\theta^2}, \quad p_{11}=\frac{3 (481 - 4126 \,\theta + 5448 \,\theta^2)}{481 - 5454 \,\theta + 4584 \,\theta^2}$.\\

\noindent(interval 44)\quad $p_ {00} = 0, \quad p_ {01} = \frac{105}{176}, \quad
p_ {10} = \frac{35 (-457 + 3514 \,\theta)}{5632 (3 -\theta)},
\quad 
p_ {11} = \frac{22811 - 125262 \,\theta}{5632 (3 -\theta)}$.\\

\noindent(intervals 45--46)\quad $p_ {00} = 0, \quad
p_ {01} = \frac{10 (-10101 + 67307 \, \theta - 57480 \, \theta^2)}{3 (12451 - 143846 \,\theta + 133968 \,\theta^2)},$\\
$p_ {10} = \frac{70 (137 - 1502 \, \theta + 1896 \, \theta^2)}{3 (12451 - 143846 \,\theta + 133968 \,\theta^2)}, \quad p_ {11} = \frac{128773 - 999468 \, \theta + 843984 \, \theta^2}{3 (12451 - 143846 \,\theta + 133968 \,\theta^2)}$.\\

\noindent(interval 47)\quad $p_ {00} = \frac{25 (-4397015 + 89027802 \, \theta - 444963880 \, \theta^2 + 361964352 \, \theta^3)}{256 (89735 - 1705504 \,\theta + 2143084 \, \theta^2 - 760352 \,\theta^3)},$\\ 
$p_ {01} = \frac{5 (156408105 - 2049708422 \, \theta + 7190327384 \, \theta^2 - 5317767360 \, \theta^3)}{768 (89735- 1705504 \,\theta+ 2143084 \, \theta^2 - 760352 \,\theta^3)},$\\
$p_ {10} = \frac{5 (55904905- 1154912102 \, \theta + 5702450392 \, \theta^2 - 4862813376 \, \theta^3)}{768 (89735 - 1705504 \,\theta + 2143084 \, \theta^2 - 760352 \,\theta^3)},$\\
$p_ {11} = \frac{-662872445 + 8036190398 \, \theta - 29445709368 \, \theta^2 + 23171626944 \, \theta^3}{768 (89735 - 1705504 \,\theta + 2143084 \, \theta^2 - 760352 \,\theta^3)}$.\\

\noindent(interval 48)\quad $p_ {00} = \frac{25 (-3403491 + 14489390 \,\theta + 57224992 \, \theta^2 - 68843136 \,\theta^3)}{256 (69459 + 2238452 \,\theta - 5258404 \, \theta^2 + 3026528 \,\theta^3)},$\\
$p_ {01} = \frac{5 (40355679 - 222492902 \, \theta + 58289248 \, \theta^2 + 126103680 \,\theta^3)}{256 (69459 + 2238452 \,\theta - 5258404 \, \theta^2 + 3026528 \,\theta^3)},$\\
$p_ {10} = \frac{5 (14424319 - 35160198 \,\theta - 308072864 \, \theta^2 + 357017728 \,\theta^3)}{256 (69459 + 2238452 \,\theta - 5258404 \, \theta^2 + 3026528 \,\theta^3)},$\\
$p_ {11} = \frac{-171031211 + 1499074462 \,\theta - 1527858144 \, \theta^2 + 80262528 \,\theta^3}{256 (69459 + 2238452 \,\theta - 5258404 \, \theta^2 + 3026528 \,\theta^3)}$.\\

\noindent(intervals 49, 52, 55)\quad $p_ {00} = \frac{-3848 + 36683 \, \theta - 82014 \, \theta^2}{22 (481 - 4514 \, \theta + 2800 \, \theta^2)}, \quad 
p_ {01} = \frac{3 (-5772 + 33139 \, \theta -1750 \, \theta^2)}{22 (481 - 4514 \, \theta + 2800 \, \theta^2)},$\\
$p_ {10} = \frac{5772 - 54739 \, \theta + 93214 \, \theta^2}{22 (481 - 4514 \, \theta + 2800 \, \theta^2)}, \quad
p_ {11} = \frac{3 (8658 - 60223 \, \theta + 18550 \, \theta^2)}{22 (481 - 4514 \, \theta + 2800 \, \theta^2)}$.\\

\noindent(intervals 50, 53)\quad $p_ {00} = \frac{4 (-65897 + 1146867 \, \theta - 5284716 \, \theta^2 + 4245120 \, \theta^3)}{724867 - 8403814 \, \theta + 8914512 \, \theta^2 - 1173504 \, \theta^3},$\\
$p_ {01} = \frac{6 (-197691 + 841277 \, \theta + 1894240 \, \theta^2 - 2579200 \, \theta^3)}{724867 - 8403814 \, \theta + 8914512 \, \theta^2 - 1173504 \, \theta^3},$\\
$p_ {10} = \frac{2 (197691 - 2903972 \,\theta + 10299072 \,\theta^2 - 7737600 \,\theta^3)}{724867 - 8403814 \, \theta + 8914512 \, \theta^2 - 1173504 \, \theta^3},$\\
$p_ {11} = \frac{3 (593073 - 4077000 \, \theta - 636736 \, \theta^2 + 4265472 \, \theta^3)}{724867 - 8403814 \, \theta + 8914512 \, \theta^2 - 1173504 \, \theta^3}$\\

\noindent(interval 51)\quad $p_ {00} = \frac{25 (-1134497 + 30973420 \, \theta - 172413316 \, \theta^2 + 143212128 \, \theta^3)}{256 (23153 - 949486 \, \theta + 1190332 \, \theta^2 - 355104 \, \theta^3)},$\\
$p_ {01} = \frac{5 (40355679 - 669765812 \, \theta + 2713368508 \, \theta^2 - 2092066720 \, \theta^3)}{768 (23153 - 949486 \, \theta + 1190332 \, \theta^2 - 355104 \, \theta^3)},$\\
$p_ {10} = \frac{5 (14424319 - 413385108 \, \theta + 2224298236 \, \theta^2 - 1926737312 \, \theta^3)}{768 (23153 - 949486 \, \theta + 1190332 \, \theta^2 - 355104 \, \theta^3)},$\\
$p_ {11} = \frac{-171031211 +2363542852 \, \theta - 10843160044 \, \theta^2 + 9080390688 \, \theta^3}{768 (23153 - 949486 \, \theta + 1190332 \, \theta^2 - 355104 \, \theta^3)}$.\\

\noindent(interval 54)\quad $p_ {00} =\frac{25 (-3040989 + 8297500 \, \theta + 78170828 \, \theta^2 - 83959584 \, \theta^3)}{256 (62061 + 2369578 \, \theta - 4947956 \, \theta^2 + 2592352 \, \theta^3)},$\\
$p_ {01} = \frac{5 (36057441 - 181705388 \, \theta - 86345468 \, \theta^2 + 234249120 \, \theta^3)}{256 (62061 + 2369578 \, \theta - 4947956 \, \theta^2 + 2592352 \, \theta^3)},$\\
$p_ {10} = \frac{5 (12888001 - 8103372 \, \theta - 393409596 \, \theta^2 + 416833952 \, \theta^3)}{256 (62061 + 2369578 \, \theta - 4947956 \, \theta^2 + 2592352 \, \theta^3)},$\\
$p_ {11} = \frac{-152814869 + 1348218268 \, \theta - 822172116 \, \theta^2 - 492783648 \, \theta^3}{256 (62061 + 2369578 \, \theta - 4947956 \, \theta^2 + 2592352 \, \theta^3)}$.\\

\noindent(interval 56)\quad $p_ {00} = \frac{4 (-481 + 8919 \, \theta - 36924 \, \theta^2)}{11 (481 - 5002 \,\theta)}, \quad
p_ {01} = \frac{6 (-1443 + 3669 \, \theta + 24616 \, \theta^2)}{11 (481 - 5002 \,\theta)},$\\
$p_ {10} = \frac{2 (1443 - 22840 \,\theta + 73848 \,\theta^2)}{11 (481 - 5002 \,\theta)}, \quad p_ {11} = \frac{3 (4329 - 22344 \, \theta - 49232 \, \theta^2)}{11 (481 - 5002 \,\theta)}$.\\

\noindent(intervals 57--59)\quad $p_ {00} = \frac{3 (-192 - 36851 \,\theta + 233568 \,\theta^2)}{176 (1 - \theta) (9 + 280 \, \theta)}, \quad
p_ {01} = \frac{-12960 +668749 \, \theta - 3154464 \, \theta^2}{880 (1 - \theta) (9 + 280 \, \theta)}, $\\
$p_ {10} = \frac{864 +119225 \, \theta - 709664 \, \theta^2}{176 (1 - \theta) (9 + 280 \, \theta)}, \quad
p_ {11} = \frac{19440 - 473629 \, \theta + 2952864 \, \theta^2}{880 (1 -\theta) (9 + 280 \, \theta)}$.\\

\noindent(interval 60)\quad $p_ {00} = \frac{-288 - 21043 \, \theta + 139246 \, \theta^2}{88 (1 - \theta) (9 + 70 \, \theta)}, \quad
p_ {01} = \frac{-6480 + 159343 \, \theta - 652598 \, \theta^2}{440 (1 - \theta) (9 + 70 \, \theta)},$\\
$p_ {10} = \frac{432 + 22019 \, \theta - 140366 \, \theta^2}{88 (1 - \theta) (9 + 70 \, \theta)}, \quad p_ {11} = \frac{9720 - 137383 \, \theta + 627398 \, \theta^2}{440 (1 - \theta) (9 + 70 \, \theta)}$.\\

\noindent(interval 61)\quad $p_ {00} = \frac{2}{11}, \quad
p_ {01} = \frac{2 (-564 + 3077 \, \theta)}{385 (1 - 10 \, \theta)}, \quad
p_ {10} = 0, \quad p_ {11} = \frac{1443 - 9304 \, \theta}{385 (1 -10 \, \theta)}$.\\

\noindent(intervals 62--66)\quad $p_ {00} = \frac{2}{11}, \quad
p_ {01} = \frac{24 (-17 + 5 \, \theta)}{55 (1 - 232 \, \theta)},\quad
p_ {10} = 0, \quad
p_ {11} = \frac{3 (151 - 3520 \,\theta)}{55 (1 - 232 \, \theta)}$.\\

\noindent(interval 67)\quad $p_ {00} = \frac{3 (-84329 + 11724002 \, \theta - 52100240 \, \theta^2)}{5632 (1 - \theta) (27 + 1160 \, \theta)},$\\
$p_ {01} = \frac{9233463 - 185098334 \, \theta + 690292080 \, \theta^2}{28160 (1 - \theta) (27 + 1160 \, \theta)}, \quad p_ {10} = \frac{280635 - 34011814 \, \theta + 155112880 \, \theta^2}{5632 (1 - \theta) (27 + 1160 \, \theta)},$\\
$p_ {11} = \frac{-8611383 + 211202654 \, \theta - 717018480 \, \theta^2}{28160 (1 - \theta) (27 + 1160 \, \theta)}$.\\

\noindent(interval 68)\quad $p_ {00} = \frac{15 (10101 - 163438 \, \theta + 538968 \, \theta^2)}{88 (1443 - 3776 \, \theta + 3840 \, \theta^2)}, \quad
p_ {01} = \frac{3 (454545 - 872754 \, \theta - 5348720 \, \theta^2)}{176 (1443 - 3776 \, \theta + 3840 \, \theta^2)},$\\
$p_ {10} = \frac{-128427 + 2391154 \, \theta - 8023080 \, \theta^2}{88 (1443 - 3776 \, \theta + 3840 \, \theta^2)}, \quad
p_ {11} = \frac{3 (-385281 + 691506 \, \theta + 5533040 \, \theta^2)}{176 (1443 - 3776 \, \theta + 3840 \, \theta^2)}$.\\

\noindent(interval 69)\quad $p_ {00}=\frac{4 (-1591 + 29401\,\theta - 100668\,\theta^2)}{11 (1591 + 2220\,\theta - 2304\,\theta^2)},\quad
p_ {01} = \frac{6 (-4773 + 11835\,\theta + 66344\,\theta^2)}{11 (1591 + 2220\,\theta - 2304\,\theta^2)},$\\
$p_ {10} = \frac{2 (4773 - 56582 \,\theta + 199032 \,\theta^2)}{11 (1591 + 2220\,\theta - 2304\,\theta^2)}, \quad
p_ {11} = \frac{3 (14319 - 17010 \,\theta - 139600 \,\theta^2)}{11 (1591 + 2220\,\theta - 2304\,\theta^2)}$.\\

\noindent(intervals 70--73, 105--106)\quad $p_ {00} = \frac{4 (-151 +900 \, \theta - 2256 \, \theta^2)}{11 (151 -13656 \, \theta - 3072 \, \theta^2)},  $\\
$p_ {01} = \frac{18 (-151 - 1516 \, \theta + 160 \, \theta^2)}{11 (151- 13656 \, \theta - 3072 \, \theta^2)}, \quad p_ {10} = \frac{6 (151 - 5152 \, \theta + 480 \, \theta^2)}{11 (151 - 13656 \, \theta - 3072 \, \theta^2)},$\\
$p_ {11} = \frac{9 (453 - 10624 \, \theta - 3392 \, \theta^2)}{11 (151 -13656 \, \theta - 3072 \, \theta^2)}$.\\

\noindent(interval 74)\quad $p_ {00} = \frac{4 (-799 +13493 \, \theta - 41180 \, \theta^2)}{11 (799 - 4296 \, \theta - 1024 \, \theta^2)}, \quad
p_ {01} = \frac{6 (-2397 + 2127 \, \theta + 27112 \, \theta^2)}{11 (799 - 4296 \, \theta - 1024 \, \theta^2)},$\\
$p_ {10} = \frac{2 (2397 - 31282 \, \theta + 81336 \, \theta^2)}{11 (799 - 4296 \, \theta - 1024 \, \theta^2)}, \quad
p_ {11} = \frac{3 (7191- 17142 \, \theta -57296 \, \theta^2)}{11 (799 - 4296 \, \theta- 1024 \, \theta^2)}$.\\

\noindent(intervals 75--97)\quad $p_ {00} = \frac{4}{121}, \quad
p_ {01} = \frac{18}{121}, \quad
p_ {10} = \frac{18}{121}, \quad
p_ {11} = \frac{81}{121}$.\\

\noindent(intervals 98--101)\quad $p_ {00} = \frac{4 (-4643 + 72197 \, \theta - 140412 \, \theta^2)}{11 (4643 - 3136 \, \theta)}, \quad
p_ {01} = \frac{6 (-13929 - 6273 \, \theta + 93608 \, \theta^2)}{11 (4643 - 3136 \, \theta)},$\\
$p_ {10} = \frac{2 (13929 - 147530 \, \theta + 280824 \, \theta^2)}{11 (4643 - 3136 \, \theta)}, \quad
p_ {11} = \frac{3 (41787 +3138 \, \theta - 187216 \, \theta^2)}{11 (4643 - 3136 \, \theta)}$.\\

\noindent(interval 102)\quad $p_ {00} = 0, \quad
p_ {01} = \frac{2 (-4643 + 11381 \, \theta)}{55 (85 -64 \, \theta)},\quad
p_ {10} = \frac{2}{11}, \quad
p_ {11} = \frac{13111 - 25642 \, \theta}{55 (85 - 64 \, \theta)}$.\\

\noindent(interval 103)\quad $p_ {00} = 0, \quad
p_ {01} = \frac{-24 \, \theta}{11 (1 - 12 \, \theta)},$ \quad
$p_ {10} = \frac{2}{11}, \quad
p_ {11} = \frac{3 (3 -28 \, \theta)}{11 (1 - 12 \, \theta)}$.\\

\noindent(interval 104)\quad $p_ {00} = \frac{5 (16779 - 231714 \, \theta + 433000 \, \theta^2)}{88 (799 + 196 \, \theta + 512 \, \theta^2)}, \quad
p_ {01} = \frac{3 (251685 + 113398 \, \theta - 1437872 \, \theta^2)}{176 (799 + 196 \, \theta + 512 \, \theta^2)},$\\
$p_ {10} = \frac{-71111 + 1161706 \, \theta - 2156808 \, \theta^2}{88 (799 + 196 \, \theta + 512 \, \theta^2)}, \quad
p_ {11} = \frac{3 (-213333 - 103990 \, \theta + 1462448 \, \theta^2)}{176 (799 + 196 \, \theta + 512 \, \theta^2)}$.\\

\noindent(interval 107)\quad $p_ {00} = 0, \quad
p_ {01} = \frac{2 (-151 - 912 \, \theta + 1200 \, \theta^2)}{151 - 14100 \, \theta + 15456 \, \theta^2}, \quad
p_ {10} = \frac{-2400 \,\theta(1 - \theta)}{151 - 14100 \, \theta + 15456 \, \theta^2},$\\
$p_ {11} = \frac{3 (151 - 3292 \, \theta + 3552 \, \theta^2)}{151 - 14100 \, \theta + 15456 \, \theta^2}$.\\

\noindent(intervals 108--110)\quad $p_ {00} = 0, \quad
p_ {01} = \frac{10 (-3171 - 17332 \, \theta + 20640 \, \theta^2)}{3 (901 - 443072 \, \theta + 443072 \, \theta^2)},$\\
$p_ {10} = \frac{10 (137 - 23948 \, \theta + 20640 \, \theta^2)}{3 (901 - 443072 \, \theta + 443072 \, \theta^2)}, \quad p_ {11} = \frac{33043 - 916416 \, \theta + 916416 \, \theta^2}{3 (901 - 443072 \, \theta + 443072 \, \theta^2)}$.\\
\end{small}\medskip

\begin{small}
The value of the game to the Players (or minus the value to Banker) as a function of $\theta$ (listed by interval number), assuming the correlated cooperative equilibrium.\medskip

\noindent1. $-16 (38351717391-1819286647 \,\theta-144331339654 \,\theta^2+100038236832 \,\theta^3)$\\
\hglue6mm${}/[10604499373 (4521+6014\,\theta-8944 \,\theta^2)]$\\
2. $-16 (38351717391-25482817681 \,\theta-85919346208 \,\theta^2+62743532544 \,\theta^3)$\\
\hglue6mm${}/[10604499373 (4521+3224\,\theta-5632 \,\theta^2)]$\\
3. $-16 (38346193962-25169828865 \,\theta-85661110112 \,\theta^2+62368599040 \,\theta^3)$\\
\hglue6mm${}/[10604499373 (4521+3224\,\theta-5632 \,\theta^2)]$\\
4. $-16 (38307467898-23056048833 \,\theta-84386812256 \,\theta^2+60283700224 \,\theta^3)$\\
\hglue6mm${}/[10604499373 (4521+3224\,\theta-5632 \,\theta^2)]$\\
5. $-16 (38301944469-22767983879 \,\theta-84166221360 \,\theta^2+59955703808 \,\theta^3)$\\
\hglue6mm${}/[10604499373 (4521+3224\,\theta-5632 \,\theta^2)]$\\
6. $-16 (38258602053-20544725543 \,\theta-82915076400 \,\theta^2+57789095936 \,\theta^3)$\\
\hglue6mm${}/[10604499373 (4521+3224\,\theta-5632 \,\theta^2)]$\\
7. $-16 (38247555195-19985211543 \,\theta-82438841648 \,\theta^2+57104244736 \,\theta^3)$\\
\hglue6mm${}/[10604499373 (4521+3224\,\theta-5632 \,\theta^2)]$\\
8. $-16 (38242031766-19754379161 \,\theta-82251228832 \,\theta^2+56830563328 \,\theta^3)$\\
\hglue6mm${}/[10604499373 (4521+3224\,\theta-5632 \,\theta^2)]$\\
9. $-(187382777855+392988314038 \,\theta-657014886600 \,\theta^2+213280439808 \,\theta^3)$\\
\hglue6mm${}/[21208998746 (692+2327 \,\theta-72\,\theta^2)]$\\
10. $-(187353514655+394256231638 \,\theta-659948157960 \,\theta^2+213827939328 \,\theta^3)$\\
\hglue6mm${}/[21208998746 (692+2327 \,\theta-72\,\theta^2)]$\\
11. $-16 (38231643330+4236073425 \,\theta-137929212146 \,\theta^2+90435219328 \,\theta^3)$\\
\hglue6mm${}/[10604499373 (4521+6014\,\theta-8944 \,\theta^2)]$\\
12. $-(37531769635-15573605854 \,\theta-775718120240 \,\theta^2+584810541056 \,\theta^3)$\\
\hglue6mm${}/[21208998746 (139+89 \,\theta-2344\,\theta^2)]$\\
13. $-(412004693345-27232465158 \,\theta-569003258104 \,\theta^2+355637709824 \,\theta^3)$\\
\hglue6mm${}/[21208998746 (1523+1752 \,\theta-496\,\theta^2)]$\\
14. $-16 (38206001451+5201757659 \,\theta-136458232122 \,\theta^2+88485633088 \,\theta^3)$\\
\hglue6mm${}/[10604499373 (4521+6014\,\theta-8944 \,\theta^2)]$\\
15. $-16 (38183907735+5954175099 \,\theta-134975025770 \,\theta^2+86628697024 \,\theta^3)$\\
\hglue6mm${}/[10604499373 (4521+6014\,\theta-8944 \,\theta^2)]$\\
16. $-16 (38177725884+6144969253 \,\theta-134661224482 \,\theta^2+86219431936 \,\theta^3)$\\
\hglue6mm${}/[10604499373 (4521+6014\,\theta-8944 \,\theta^2)]$\\
17. $-16 (38172202455+6308149751 \,\theta-134368573602 \,\theta^2+85846502272 \,\theta^3)$\\
\hglue6mm${}/[10604499373 (4521+6014\,\theta-8944 \,\theta^2)]$\\
18. $-48 (12722226342+2151879149 \,\theta-44713793770 \,\theta^2+28515771040 \,\theta^3)$\\
\hglue6mm${}/[10604499373 (4521+6014\,\theta-8944 \,\theta^2)]$\\
19. $-16 (38161814019+6583819441 \,\theta-133982425830 \,\theta^2+85327161248 \,\theta^3)$\\
\hglue6mm${}/[10604499373 (4521+6014\,\theta-8944 \,\theta^2)]$\\
20. $-16 (38156290590+6722999183 \,\theta-133781283774 \,\theta^2+85058406880 \,\theta^3)$\\
\hglue6mm${}/[10604499373 (4521+6014\,\theta-8944 \,\theta^2)]$\\
21. $-16 (38117564526+7669478639 \,\theta-133597538238 \,\theta^2+84374942176 \,\theta^3)$\\
\hglue6mm${}/[10604499373 (4521+6014\,\theta-8944 \,\theta^2)]$\\
22. $-48 (12703794225+2603782725 \,\theta-44460086614 \,\theta^2+28029340032 \,\theta^3)$\\
\hglue6mm${}/[10604499373 (4521+6014\,\theta-8944 \,\theta^2)]$\\
23. $-16 (38106517668+7914606307 \,\theta-133153181790 \,\theta^2+83812899296 \,\theta^3)$\\
\hglue6mm${}/[10604499373 (4521+6014\,\theta-8944 \,\theta^2)]$\\
24. $-16 (38100994239+8020554233 \,\theta-132909967398 \,\theta^2+83519610784 \,\theta^3)$\\
\hglue6mm${}/[10604499373 (4521+6014\,\theta-8944 \,\theta^2)]$\\
25. $-48 (12698490270+2706064735 \,\theta-44230934414 \,\theta^2+27752252352 \,\theta^3)$\\
\hglue6mm${}/[10604499373 (4521+6014\,\theta-8944 \,\theta^2)]$\\
26. $-16 (38089947381+8200141375 \,\theta-132541097674 \,\theta^2+83067643840 \,\theta^3)$\\
\hglue6mm${}/[10604499373 (4521+6014\,\theta-8944 \,\theta^2)]$\\
27. $-16 (38089947381+8198260639 \,\theta-132509414282 \,\theta^2+83037179328 \,\theta^3)$\\
\hglue6mm${}/[10604499373 (4521+6014\,\theta-8944 \,\theta^2)]$\\
28. $-16 (38083765530+8281974497 \,\theta-132328214258 \,\theta^2+82817027456 \,\theta^3)$\\
\hglue6mm${}/[10604499373 (4521+6014\,\theta-8944 \,\theta^2)]$\\
29. $-48 (12692966841+2781538655 \,\theta-44081581834 \,\theta^2+27568983840 \,\theta^3)$\\
\hglue6mm${}/[10604499373 (4521+6014\,\theta-8944 \,\theta^2)]$\\
30. $-16 (38054941278+8614852225 \,\theta-131791208962 \,\theta^2+82139008000 \,\theta^3)$\\
\hglue6mm${}/[10604499373 (4521+6014\,\theta-8944 \,\theta^2)]$\\
31. $-16 (38030982033+8876780531 \,\theta-131290586018 \,\theta^2+81528362624 \,\theta^3)$\\
\hglue6mm${}/[10604499373 (4521+6014\,\theta-8944 \,\theta^2)]$\\
32. $-48 (12675152868+2977934613 \,\theta-43739010386 \,\theta^2+27143517920 \,\theta^3)$\\
\hglue6mm${}/[10604499373 (4521+6014\,\theta-8944 \,\theta^2)]$\\
33. $-16 (38019276753+8993516941 \,\theta-131127339958 \,\theta^2+81314577184 \,\theta^3)$\\
\hglue6mm${}/[10604499373 (4521+6014\,\theta-8944 \,\theta^2)]$\\
34. $-16 (37995317508+9198212675 \,\theta-130749290142 \,\theta^2+80856709600 \,\theta^3)$\\
\hglue6mm${}/[10604499373 (4521+6014\,\theta-8944 \,\theta^2)]$\\
35. $-16 (37899480528+9917300163 \,\theta-128964600590 \,\theta^2+78805363552 \,\theta^3)$\\
\hglue6mm${}/[10604499373 (4521+6014\,\theta-8944 \,\theta^2)]$\\
36. $-16 (37875521283+10105379989 \,\theta-128638651246 \,\theta^2+78408365536 \,\theta^3)$\\
\hglue6mm${}/[10604499373 (4521+6014\,\theta-8944 \,\theta^2)]$\\
37. $-16 (37779684303+10824467477 \,\theta-127146508254 \,\theta^2+76649566048 \,\theta^3)$\\
\hglue6mm${}/[10604499373 (4521+6014\,\theta-8944 \,\theta^2)]$\\
38. $-16 (37757590587+10952865261 \,\theta-126579798526 \,\theta^2+76038454880 \,\theta^3)$\\
\hglue6mm${}/[10604499373 (4521+6014\,\theta-8944 \,\theta^2)]$\\
39. $-16 (37733631342+11092020469 \,\theta-126277235434 \,\theta^2+75690663232 \,\theta^3)$\\
\hglue6mm${}/[10604499373 (4521+6014\,\theta-8944 \,\theta^2)]$\\
40. $-16 (3986587163-50089552645 \,\theta+71238575350 \,\theta^2-27184176960 \,\theta^3)$\\
\hglue6mm${}/[10604499373 (481-5570\,\theta+3056 \,\theta^2)]$\\
41. $-16 (3986653060-50070045283 \,\theta+71085090302 \,\theta^2-27084759168 \,\theta^3)$\\
\hglue6mm${}/[10604499373 (481-5570\,\theta+3056 \,\theta^2)]$\\
42. $-16 (4000633325-49261121250 \,\theta+83098340592 \,\theta^2-39112507584 \,\theta^3)$\\
\hglue6mm${}/[10604499373 (481-5454\,\theta+4584 \,\theta^2)]$\\
43. $-32 (2000151920-24620476825 \,\theta+41477367966 \,\theta^2-19487587992 \,\theta^3)$\\
\hglue6mm${}/[10604499373 (481-5454\,\theta+4584 \,\theta^2)]$\\ 
44. $-(138779541091-180227257793 \,\theta+67215533982 \,\theta^2)/[3732783779296 (3-\,\theta)]$\\
45. $-32 (155846869069-1946669313673 \,\theta+3420246980781 \,\theta^2-1624641679152 \,\theta^3)$\\
\hglue6mm${}/[31813498119 (12451-143846\,\theta+133968 \,\theta^2)]$\\
46. $-16 (311778825413-3893700862096 \,\theta+6839063336762 \,\theta^2-3247562279904 \,\theta^3)$\\
\hglue6mm${}/[31813498119 (12451-143846\,\theta+133968 \,\theta^2)]$\\
47. $-(33993402201005-689235951232512 \,\theta+1456513482893394 \,\theta^2$\\
\hglue6mm${}-1308473239136408 \,\theta^3 +490002654108096 \,\theta^4)/[31813498119(89735-1705504 \,\theta$\\
\hglue6mm${}+2143084 \,\theta^2-760352 \,\theta^3)]$\\
48. $-(8770816751099+288640723789268 \,\theta-976291100477746 \,\theta^2$\\
\hglue6mm${}+1078180277518752 \,\theta^3-397430091205248 \,\theta^4)/[10604499373(69459+2238452 \,\theta$\\
\hglue6mm${}-5258404 \,\theta^2+3026528 \,\theta^3)]$\\
49. $-16 (43893867121-453580147988 \,\theta+651507327262 \,\theta^2-249551002512 \,\theta^3)$\\
\hglue6mm${}/[116649493103 (481-4514\,\theta+2800 \,\theta^2)]$\\
50. $-16 (6013459795577-75287766753692 \,\theta+140463071223342 \,\theta^2-82499629362720 \,\theta^3$\\
\hglue6mm${}+11325699600384 \,\theta^4)/[10604499373(724867-8403814 \,\theta+8914512 \,\theta^2-1173504 \,\theta^3)]$\\
51. $-(8770816751099-378080880576242 \,\theta+815969259822614 \,\theta^2-667951000381788 \,\theta^3$\\
\hglue6mm${}+214567660924192 \,\theta^4)/[31813498119(23153-949486 \,\theta+1190332 \,\theta^2-355104 \,\theta^3)]$\\
52. $-32 (21941365264-226656470836 \,\theta+325074509485 \,\theta^2-124377848344 \,\theta^3)$\\
\hglue6mm${}/[116649493103 (481-4514\,\theta+2800 \,\theta^2)]$\\
53. $-32 (3005967041168-37624249835880 \,\theta+70122663500069 \,\theta^2-41148951145056 \,\theta^3$\\
\hglue6mm${}+5652930035712 \,\theta^4)/[10604499373(724867-8403814 \,\theta+8914512 \,\theta^2-1173504 \,\theta^3)]$\\
54. $-(7836646919621+306358408375442 \,\theta-950316659171094 \,\theta^2+981992427224988 \,\theta^3$\\
\hglue6mm${}-344000196972832 \,\theta^4)/[10604499373(62061+2369578 \,\theta-4947956 \,\theta^2+2592352 \,\theta^3)]$\\
55. $-16 (43892680975-453351873116 \,\theta+649999468758 \,\theta^2-248662552784 \,\theta^3)$\\
\hglue6mm${}/[116649493103 (481-4514\,\theta+2800 \,\theta^2)]$\\
56. $-16 (43892680975-496774406644 \,\theta+434696771870 \,\theta^2-14821588992 \,\theta^3)$\\
\hglue6mm${}/[116649493103 (481-5002\,\theta)]$\\
57. $-16 (4106383875+117669499647 \,\theta-241040766397 \,\theta^2+144247786600 \,\theta^3)$\\
\hglue6mm${}/[583247465515 (1-\theta)(9+280 \,\theta)]$\\
58. $-16 (4105452960+117617182712 \,\theta-240627631107 \,\theta^2+143887899160 \,\theta^3)$\\
\hglue6mm${}/[583247465515 (1-\theta)(9+280 \,\theta)]$\\
59. $-16 (4104411075+117525878997 \,\theta-239952017347 \,\theta^2+143304631000 \,\theta^3)$\\
\hglue6mm${}/[583247465515 (1-\theta)(9+280 \,\theta)]$\\
60. $-16 (4104411075+23039957103 \,\theta-56052048598 \,\theta^2+38900841910 \,\theta^3)$\\
\hglue6mm${}/[583247465515 (1-\theta)(9+70 \,\theta)]$\\ 
61. $-16 (3211590381-34836061778 \,\theta+29378412470 \,\theta^2)/[4082732258605 (1-10 \,\theta)]$\\
62. $-16 (485269611-103072262120 \,\theta+84236525960 \,\theta^2)/[583247465515 (1-232 \,\theta)]$\\ 
63. $-16 (495852651-102358150280\,\theta+79832883080 \,\theta^2)/[583247465515 (1-232 \,\theta)]$\\
64. $-32 (247617508-51142256085 \,\theta+39737548940 \,\theta^2)/[583247465515 (1-232 \,\theta)]$\\ 
65. $-16 (495131581-102260001920\,\theta+79356005320 \,\theta^2)/[583247465515 (1-232 \,\theta)]$\\
66. $-96 (82418991-17030603810 \,\theta+13166266940 \,\theta^2)/[583247465515 (1-232 \,\theta)]$\\ 
67. $-(311880897984+15811173883543\,\theta-26679094516830 \,\theta^2+5412559516400 \,\theta^3)$\\
\hglue6mm${}/[1166494931030 (1-\theta)(27+1160 \,\theta)]$\\
68. $-(1791574243056-4705817549217 \,\theta+4165762914218 \,\theta^2-2898522685440 \,\theta^3)$\\
\hglue6mm${}/[116649493103 (1443-3776 \,\theta+3840\,\theta^2)]$\\
69. $-32 (71768589237+19533928574 \,\theta-109166266375 \,\theta^2+61438170624 \,\theta^3)$\\
\hglue6mm${}/[116649493103 (1591+2220\,\theta-2304 \,\theta^2)]$\\
70. $-96 (2270491719-198406986650 \,\theta+97360820404 \,\theta^2+31675226112 \,\theta^3)$\\
\hglue6mm${}/[116649493103 (151-13656\,\theta-3072 \,\theta^2)]$\\
71. $-16 (13590037297-1188123266214 \,\theta+575293938312 \,\theta^2+187926018048 \,\theta^3)$\\
\hglue6mm${}/[116649493103 (151-13656\,\theta-3072 \,\theta^2)]$\\
72. $-16 (13470631933-1179543531806 \,\theta+542973152328 \,\theta^2+180355946496 \,\theta^3)$\\
\hglue6mm${}/[116649493103 (151-13656\,\theta-3072 \,\theta^2)]$\\
73. $-32 (6728540974-589282372355 \,\theta+269645984484 \,\theta^2+89747966976 \,\theta^3)$\\
\hglue6mm${}/[116649493103 (151-13656\,\theta-3072 \,\theta^2)]$\\
74. $-32 (35603339326-218273522859 \,\theta+104761790337 \,\theta^2+26521330688 \,\theta^3)$\\
\hglue6mm${}/[116649493103 (799-4296\,\theta-1024 \,\theta^2)]$\\ 
75. $-32 (464056790-298455787 \,\theta)/1283144424133$\\
76. $-16 (926326887-589910792 \,\theta)/1283144424133$\\ 
77. $-32 (462522772-292685813 \,\theta)/1283144424133$\\ 
78. $-32 (456494932-271475013\,\theta)/1283144424133$\\ 
79. $-16 (910783169-535383252 \,\theta)/1283144424133$\\
80. $-16 (907840909-525959036 \,\theta)/1283144424133$\\ 
81. $-32 (453299997-261013145 \,\theta)/1283144424133$\\ 
82. $-16 (903908163-513554378\,\theta)/1283144424133$\\ 
83. $-64 (225656705-127395301 \,\theta)/1283144424133$\\
84. $-64 (224036216-122516989 \,\theta)/1283144424133$\\ 
85. $-16 (882685709-450538356 \,\theta)/1283144424133$\\ 
86. $-16 (879112323-440449324\,\theta)/1283144424133$\\ 
87. $-32 (439188379-219216208 \,\theta)/1283144424133$\\
88. $-32 (429980635-194825504 \,\theta)/1283144424133$\\ 
89. $-32 (426739657-186241292 \,\theta)/1283144424133$\\ 
90. $-112 (121541069-52211686\,\theta)/1283144424133$\\ 
91. $-16 (837756191-331594154 \,\theta)/1283144424133$\\
92. $-16 (818601887-282073930 \,\theta)/1283144424133$\\ 
93. $-128 (101988757-34404919 \,\theta)/1283144424133$\\ 
94. $-16 (814123363-270760828\,\theta)/1283144424133$\\ 
95. $-1632 (7964085-2610809 \,\theta)/1283144424133$\\
96. $-80 (162219151-52654982 \,\theta)/1283144424133$\\ 
97. $-64 (201959483-63927034 \,\theta)/1283144424133$\\
98. $-32 (186882616486-253015679184 \,\theta+144088046097 \,\theta^2-3019419648 \,\theta^3)$\\
\hglue6mm${}/[116649493103 (4643-3136\,\theta)]$\\
99. $-32 (185476855376-252605631128 \,\theta+150137699473 \,\theta^2-1869164544 \,\theta^3)$\\
\hglue6mm${}/[116649493103 (4643-3136\,\theta)]$\\
100. $-16 (369108410761-500689757702 \,\theta+298969002402 \,\theta^2-2300510208 \,\theta^3)$\\
\hglue6mm${}/[116649493103 (4643-3136\,\theta)]$\\
101. $-16 (369108410761-529665954406 \,\theta+358753324194 \,\theta^2+12652806144 \,\theta^3)$\\
\hglue6mm${}/[116649493103 (4643-3136\,\theta)]$\\ 
102. $-16 (33801094747-51304659314 \,\theta+34947794470 \,\theta^2)/[583247465515 (85-64 \,\theta)]$\\
103. $-48 (21823585-258373988 \,\theta+7942888 \,\theta^2)/[116649493103 (1-12 \,\theta)]$\\ 
104. $-(248577529200+2688091349739\,\theta-2194172765150 \,\theta^2-113813274624 \,\theta^3)$\\
\hglue6mm${}/[116649493103 (799+196 \,\theta+512 \,\theta^2)]$\\
105. $-32 (5977292569-440436505944 \,\theta-85878929748 \,\theta^2+2324975616 \,\theta^3)$\\
\hglue6mm${}/[116649493103 (151-13656\,\theta-3072 \,\theta^2)]$\\
106. $-16 (11904626033-877406359116 \,\theta-178433064456 \,\theta^2+2998923264 \,\theta^3)$\\
\hglue6mm${}/[116649493103 (151-13656\,\theta-3072 \,\theta^2)]$\\
107. $-16 (1068219149-82481312482 \,\theta+91562829360 \,\theta^2-1378780800 \,\theta^3)$\\
\hglue6mm${}/[10604499373 (151-14100\,\theta+15456 \,\theta^2)]$\\
108. $-80 (7333327695-1557446902148 \,\theta+1581024008528 \,\theta^2-23715029760 \,\theta^3)$\\
\hglue6mm${}/[31813498119 (901-443072\,\theta+443072 \,\theta^2)]$\\
109. $-320 (1824776445-387874428487 \,\theta+390822719812 \,\theta^2-2965596480 \,\theta^3)$\\
\hglue6mm${}/[31813498119 (901-443072\,\theta+443072 \,\theta^2)]$\\
110. $-80 (2421541645-515181045616 \,\theta+515181045616 \,\theta^2)$\\
\hglue6mm${}/[10604499373 (901-443072 \,\theta+443072 \,\theta^2)]$
\end{small}\bigskip

\newpage

\section*{Appendix B.  Independent cooperative equilibrium}

\begin{small}
We list the 130 points of discontinuity in $(0,1/2)$ of Banker's best response to the Players' maximin strategy in the independent cooperative equilibrium.  All are roots of polynomials of degree 13 or less.\medskip

\noindent$\theta_1\approx0.0172597$ (root of $1196071 - 70128084\,\theta + 48076640\,\theta^2$)\\
$\theta_2\approx0.0182052$ (root of $13439 - 747528\,\theta + 512560\,\theta^2$)\\
$\theta_3\approx0.0187178$ (root of $1196071 - 64730946\,\theta + 44388440\,\theta^2$)\\
$\theta_4\approx0.0192602$ (root of $1196071 - 62931900\,\theta + 43159040\,\theta^2$)\\
$\theta_5\approx0.0194839$ (root of $15041 - 782610\,\theta + 546040\,\theta^2$)\\
$\theta_6\approx0.0232232$ (root of $1196071 - 52337518\,\theta + 35919240\,\theta^2$)\\
$\theta_7\approx0.0240640$ (root of $1196071 - 50538472\,\theta + 34689840\,\theta^2$)\\
$\theta_8\approx0.0244412$ (root of $150499 - 6259726\,\theta + 4178760\,\theta^2$)\\
$\theta_9\approx0.0245604$ (root of $1338649 - 55459816\,\theta + 38898960\,\theta^2$)\\
$\theta_{10}\approx0.0255056$ (root of $1053493 - 42019036\,\theta + 28021920\,\theta^2$)\\
$\theta_{11}\approx0.0269967$ (root of $1196071 - 45141334\,\theta + 31001640\,\theta^2$)\\
$\theta_{12}\approx0.0273040$ (root of $1196071 - 44720186\,\theta + 33494208\,\theta^2$)\\
$\theta_{13}\approx0.0289972$ (root of $1338649 - 47203820\,\theta + 35835264\,\theta^2$)\\
$\theta_{14}\approx0.0311422$ (root of $1196071 - 39323048\,\theta + 29421312\,\theta^2$)\\
$\theta_{15}\approx0.0341664$ (root of $1196071 - 35924850\,\theta + 26856896\,\theta^2$)\\
$\theta_{16}\approx0.0360182$ (root of $1196071 - 34125804\,\theta + 25499264\,\theta^2$)\\
$\theta_{17}\approx0.0362616$ (root of $1053493 - 29843124\,\theta + 21800576\,\theta^2$)\\
$\theta_{18}\approx0.0376433$ (root of $1338649 - 36609438\,\theta + 27840320\,\theta^2$)\\
$\theta_{19}\approx0.0403979$ (root of $13439 - 343008\,\theta + 256000\,\theta^2$)\\
$\theta_{20}\approx0.0425911$ (root of $7097 - 172107\,\theta + 128568\,\theta^2$)\\
$\theta_{21}\approx0.0431219$ (root of $1697705233 - 179649339338\,\theta + 3456239837408\,\theta^2$\\
\hglue1cm${} - 4784841897088\,\theta^3 + 1712192778240\,\theta^4$)\\
$\theta_{22}\approx0.0433678$ (root of $399 - 9500\,\theta + 6909\,\theta^2$)\\
$\theta_{23}\approx0.0440036$ (root of $13037615396176416 - 320553761853928046\,\theta$\\
\hglue1cm${} + 520229303579284217\,\theta^2+ 785092514491791916\,\theta^3 - 1716296986200855713\,\theta^4 $\\
\hglue1cm${} + 733490688027658450\,\theta^5$)\\
$\theta_{24}\approx0.0449134$ (root of $453 - 10436\,\theta + 7791\,\theta^2$)\\
$\theta_{25}\approx0.0459910$ (root of $1196071 - 26929620\,\theta + 20068736\,\theta^2$)\\
$\theta_{26}\approx0.0494121$ (root of $1196071 - 25130574\,\theta + 18711104\,\theta^2$)\\
$\theta_{27}\approx0.0569262$ (root of $37906880 - 777815991\,\theta + 2034322957\,\theta^2 - 1198987872\,\theta^3$)\\
$\theta_{28}\approx0.0577965$ (root of $151 - 2730\,\theta + 2031\,\theta^2$)\\
$\theta_{29}\approx0.0581390$ (root of $169 - 3042\,\theta + 2325\,\theta^2$)\\
$\theta_{30}\approx0.0629881$ (root of $280721548800000 + 292930382236800\,\theta $\\
\hglue1cm${} - 74821515301320409\,\theta^2-79544313812216314\,\theta^3 + 1223456960296730959\,\theta^4$\\
\hglue1cm${} - 1756896199569915776\,\theta^5 + 697311824652591360\,\theta^6$)\\
$\theta_{31}\approx0.0799072$ (root of $528987825 - 33799834000\,\theta + 363014976832\,\theta^2$\\
\hglue1cm${} - 287912894464\,\theta^3 + 20926218240\,\theta^4$)\\
$\theta_{32}\approx0.0830993$ (root of $7053171 - 1027543564\,\theta + 12119103820\,\theta^2- 9367547392\,\theta^3$\\
\hglue1cm${}  + 462910464\,\theta^4$)\\
$\theta_{33}\approx0.0840033$ (root of $28113633 + 9276709160\,\theta - 122239114352\,\theta^2+ 93270978560\,\theta^3$\\
\hglue1cm${}  - 1804308480\,\theta^4$)\\
$\theta_{34}\approx0.0842924$ (root of $6934 - 88199\,\theta + 70442\,\theta^2$)\\
$\theta_{35}\approx0.0894740$ (root of $18849764 - 225910055\,\theta + 170294834\,\theta^2$)\\
$\theta_{36}\approx0.0939539$ (root of $11738571 - 3303674192\,\theta + 36713503424\,\theta^2 -  30864293888\,\theta^3$\\
\hglue1cm${} + 2177187840\,\theta^4$)\\
$\theta_{37}\approx0.0991012$ (root of $176329275 + 7376528780\,\theta - 99541141972\,\theta^2 + 72124260864\,\theta^3$\\
\hglue1cm${} + 511349760\,\theta^4$)\\
$\theta_{38}\approx0.0992228$ (root of $342436 - 3712083\,\theta + 2629436\,\theta^2$)\\
$\theta_{39}\approx0.100952$ (root of $2767375 - 29557970\,\theta + 21249219\,\theta^2$)\\
$\theta_{40}\approx0.103087$ (root of $7609292 - 79982713\,\theta + 59838606\,\theta^2$)\\
$\theta_{41}\approx0.105963$ (root of $39757085 - 406643392\,\theta + 296755596\,\theta^2$)\\
$\theta_{42}\approx0.106917$ (root of $135927651 - 1379098100\,\theta + 1007903500\,\theta^2$)\\
$\theta_{43}\approx0.107512$ (root of $3490495 - 35169356\,\theta + 25143276\,\theta^2$)\\
$\theta_{44}\approx0.109362$ (root of $89621179 - 891209305\,\theta + 655761490\,\theta^2$)\\
$\theta_{45}\approx0.112297$ (root of $9720495 - 94305370\,\theta + 68964683\,\theta^2$)\\
$\theta_{46}\approx0.112497$ (root of $1203144810881164339109220 - 48404500533224883973436836\,\theta$\\
\hglue1cm${} + 217206087418384854871702492\,\theta^2 + 2737716198663238142118733236\,\theta^3$\\
\hglue1cm${} - 1704798984054721412226067967\,\theta^4 - 124726766579831129650984558129\,\theta^5$\\
\hglue1cm${} + 24642895302856208693137779264\,\theta^6 + 323782264888609988784356977916\,\theta^7$\\
\hglue1cm${} - 313067602354574224207541990912\,\theta^8 + 81909086956457405877690716160\,\theta^9$)\\
$\theta_{47}\approx0.113072$ (root of $867789312 - 142483182448\,\theta + 999401311661\,\theta^2$\\
\hglue1cm${} + 1809875037880\,\theta^3- 923533355696\,\theta^4$)\\
$\theta_{48}\approx0.113700$ (root of $249777662135923802400 - 2330012213168443332880\,\theta $\\
\hglue1cm${} - 778322088674843368431\,\theta^2 + 16819673068522365489019\,\theta^3$\\
\hglue1cm${} + 4688579092023614656568\,\theta^4- 16780160518394139755028\,\theta^5$\\
\hglue1cm${} + 9461677057640338596256\,\theta^6 - 2264669144525844514368\,\theta^7$)\\
$\theta_{49}\approx0.114111$ (root of $2407125 - 23211335\,\theta + 18550286\,\theta^2$)\\
$\theta_{50}\approx0.114203$ (root of $6948788000 - 76666790013\,\theta + 147785705330\,\theta^2$\\
\hglue1cm${} - 81031247077\,\theta^3$)\\
$\theta_{51}\approx0.116342$ (root of $31878574183465640992163781929041$\\
\hglue1cm${} + 1669944974294610524788838645431620\,\theta$\\
\hglue1cm${} - 67530347158686182373798647100054978\,\theta^2$\\
\hglue1cm${} + 38177974301164442216941931951267116\,\theta^3$\\
\hglue1cm${} + 7836919606291076454356455180144690449\,\theta^4$\\
\hglue1cm${} - 17021612879852750311143191669458133976\,\theta^5$\\
\hglue1cm${} - 222165609541195429736784131296049150376\,\theta^6$\\
\hglue1cm${} + 221372997571832024084486940273881214432\,\theta^7$\\
\hglue1cm${} + 1369274491179181866442106236836520735632\,\theta^8$\\
\hglue1cm${} - 983091590295807945310282253520667561472\,\theta^9$\\
\hglue1cm${} - 622364322007457479455410707968203022336\,\theta^{10}$\\
\hglue1cm${}  + 280798074725443967846764626985292464128\,\theta^{11}$\\
\hglue1cm${} + 109147049291319385597103181202584764416\,\theta^{12}$)\\
$\theta_{52}\approx0.116377$ (root of $101880794910234933565921 - 7788787645490499989587116\,\theta$\\
\hglue1cm${} + 208257615405383055853101059\,\theta^2 - 3376861941757007607789768168\,\theta^3$\\
\hglue1cm${} + 24013663696725662357989680215\,\theta^4 + 20952663850106367747714734708\,\theta^5$\\
\hglue1cm${} - 811036808713452381113545783575\,\theta^6 + 923619557495752957614773322200\,\theta^7$\\
\hglue1cm${} + 6493253708045298491062971879240\,\theta^8$\\
\hglue1cm${} - 3775810143130181144728177847480\,\theta^9$\\
\hglue1cm${} - 3366137091487634221868815703980\,\theta^{10}$\\
\hglue1cm${} + 1221509511981970863677153710256\,\theta^{11}$\\
\hglue1cm${} + 395553778602382026493910482576\,\theta^{12}$)\\
$\theta_{53}\approx0.116645$ (root of $19773845 - 1380111260\,\theta + 7824344475\,\theta^2 +  23353862962\,\theta^3 $\\
\hglue1cm${} - 12500510326\,\theta^4$)\\
$\theta_{54}\approx0.117024$ (root of $19773845 - 1378268060\,\theta + 7770645915\,\theta^2 + 23359392562\,\theta^3 $\\
\hglue1cm${} - 12454184566\,\theta^4$)\\
$\theta_{55}\approx0.119390$ (root of $18752428782800868766627489245150625$\\
\hglue1cm${} - 483813925550481919691720268616919500\,\theta$\\
\hglue1cm${} - 1027025755915511968203463407272278850\,\theta^2$\\
\hglue1cm${} + 77226566286031539424130317242727073420\,\theta^3$\\
\hglue1cm${} - 379450456019531010523709830056727811999\,\theta^4$\\
\hglue1cm${} - 247011381562756434020077521640337029288\,\theta^5$\\
\hglue1cm${} + 2056689648424333633548344472714350544952\,\theta^6$\\
\hglue1cm${} - 1836711615929676877612957506562024135840\,\theta^7$\\
\hglue1cm${} - 269029282628576913365925509498032229744\,\theta^8$\\
\hglue1cm${} + 674642330808747981042734705177986909440\,\theta^9$\\
\hglue1cm${} + 26011425693973073002002243759976386048\,\theta^{10}$\\
\hglue1cm${} - 86776671632310773471020896545001836544\,\theta^{11}$\\
\hglue1cm${} - 14059780820025475772443417499679289344\,\theta^{12}$)\\
$\theta_{56}\approx0.119622$ (root of $41629610355987300400 + 1716774076566150210960\,\theta$\\
\hglue1cm${} - 16363792955265443808499\,\theta^2 - 10868567664718647029295\,\theta^3$\\
\hglue1cm${} + 28207512146724258923727\,\theta^4 + 317871169419755921339\,\theta^5 $\\
\hglue1cm${} - 5521452055307275801048\,\theta^6 - 931795923734199334224\,\theta^7$)\\
$\theta_{57}\approx0.122567$ (root of $22925 - 207210\,\theta + 164557\,\theta^2$)\\
$\theta_{58}\approx0.127445$ (root of $2788980871911465873944400 - 54602359795543790717551505\,\theta$\\
\hglue1cm${} + 256853207961211029447962549\,\theta^2 + 89906421920889229547327678\,\theta^3$\\
\hglue1cm${} - 752455769614548359339949574\,\theta^4 + 255352508719115525529947003\,\theta^5$\\
\hglue1cm${} + 442191963735884001737980625\,\theta^6 - 239218537833107616911584376\,\theta^7$)\\
$\theta_{59}\approx0.131175$ (root of $49375943 - 415750391\,\theta + 299888348\,\theta^2$)\\
$\theta_{60}\approx0.132459$ (root of $143513022619746018669526873837491600$\\
\hglue1cm${} + 15203910160924318659900866458957308720\,\theta$\\
\hglue1cm${} - 1411114194312455670913597321350939595784\,\theta^2$\\
\hglue1cm${} + 8364732971583765411812649830782376974801\,\theta^3$\\
\hglue1cm${} + 18158386978388110676635348720752450625636\,\theta^4$\\
\hglue1cm${} - 57382235622240037875378374203508631176046\,\theta^5$\\
\hglue1cm${} - 35033606577187313236105278528986696655804\,\theta^6$\\
\hglue1cm${} + 146980085520612377730283888489854048198033\,\theta^7$\\
\hglue1cm${} - 36399257357980852171910885132072002975200\,\theta^8$\\
\hglue1cm${} - 120723967403035949262443029071550255858332\,\theta^9$\\
\hglue1cm${} + 99871251458308105088806500088716932262312\,\theta^{10}$\\
\hglue1cm${} - 20607307655490616016545880781929743936352\,\theta^{11}$\\
\hglue1cm${} - 2399201882479034085378887614400221417760\,\theta^{12}$\\
\hglue1cm${} + 578864372257931116292868638079322994176\,\theta^{13}$)\\
$\theta_{61}\approx0.135985$ (root of $453 - 3698\,\theta + 2697\,\theta^2$)\\
$\theta_{62}\approx0.141109$ (root of $42031503195645968 - 287371985069366399\,\theta$\\
\hglue1cm${} - 154893412202865613\,\theta^2 + 583091512457455351\,\theta^3- 68486965475961355\,\theta^4 $\\
\hglue1cm${} - 136576631328641552\,\theta^5$)\\
$\theta_{63}\approx0.142653$ (root of $67193207 - 524545673\,\theta + 375179318\,\theta^2$)\\
$\theta_{64}\approx0.144811$ (root of $5895 - 45304\,\theta + 31737\,\theta^2$)\\
$\theta_{65}\approx0.145171$ (root of $22095466997 - 255539273500\,\theta + 834844019492\,\theta^2$\\
\hglue1cm${} - 890947089792\,\theta^3 + 299815899904\,\theta^4$)\\
$\theta_{66}\approx0.145310$ (root of $15650525033 - 413450240736\,\theta + 2600512122484\,\theta^2$\\
\hglue1cm${} - 3626566188080\,\theta^3 + 1447256560384\,\theta^4$)\\
$\theta_{67}\approx0.146912$ (root of $123985 - 3905116\,\theta + 20836819\,\theta^2$)\\
$\theta_{68}\approx0.147245$ (root of $81112271982771456 - 3450350938797389552\,\theta $\\
\hglue1cm${} + 46940285515777139336\,\theta^2 - 205036405356966452675\,\theta^3$\\
\hglue1cm${} + 104229490423164281702\,\theta^4 + 239570955153413384545\,\theta^5$\\
\hglue1cm${} - 180782164276324080448\,\theta^6 + 32098656047492471808\,\theta^7$)\\
$\theta_{69}\approx0.148149$ (root of $856748372818023504 - 17834677044742215743\,\theta $\\
\hglue1cm${} + 130772281989721194986\,\theta^2 - 382313919483453580679\,\theta^3 $\\
\hglue1cm${} + 299664189704368875492\,\theta^4 + 235160875621920624132\,\theta^5$\\
\hglue1cm${} - 264479813418800856320\,\theta^6 + 14909300130971796928\,\theta^7$)\\
$\theta_{70}\approx0.152502$ (root of $3811606140251090932648904222557341 $\\
\hglue1cm${} - 48659723325357212127573509312325281\,\theta$\\
\hglue1cm${} - 5559748714014470227347123791278703915\,\theta^2$\\
\hglue1cm${} + 58551832449956278963723525304686434812\,\theta^3$\\
\hglue1cm${} - 130955154186533119929354058122354955912\,\theta^4$\\
\hglue1cm${} - 141121091350991485099365617874071584624\,\theta^5$\\
\hglue1cm${} + 667410142027922297849569442082870073584\,\theta^6$\\
\hglue1cm${} - 228329077049470850859126664736673855872\,\theta^7$\\
\hglue1cm${} - 956540102362576165086832083540348188224\,\theta^8$\\
\hglue1cm${} + 832239035092726029898030206291108825600\,\theta^9$\\
\hglue1cm${} + 290295269117768212491290172593959914496\,\theta^{10}$\\
\hglue1cm${} - 549751317905943940582569497496521026560\,\theta^{11}$\\
\hglue1cm${} + 163798995725649560992667028889170631680\,\theta^{12}$)\\
$\theta_{71}\approx0.153828$ (root of $18074088461 - 489486828168\,\theta + 2740068166648\,\theta^2 $\\
\hglue1cm${} - 2125555458848\,\theta^3+ 216860715472\,\theta^4$)\\
$\theta_{72}\approx0.154737$ (root of $1532584065 - 29711728576\,\theta + 153512273652\,\theta^2 $\\
\hglue1cm${} - 173198601776\,\theta^3 + 54045846912\,\theta^4$)\\
$\theta_{73}\approx0.154894$ (root of $7693595 - 56442557\,\theta + 43723254\,\theta^2$)\\
$\theta_{74}\approx0.158434$ ($=35646/224989$)\\
$\theta_{75}\approx0.159386$ ($=1828/11469$)\\
$\theta_{76}\approx0.161238$ (root of $43162535 - 723232274\,\theta + 2825250168\,\theta^2$)\\
$\theta_{77}\approx0.162957$ (root of $11136593 - 242747284\,\theta + 1070261476\,\theta^2$)\\
$\theta_{78}\approx0.165599$ (root of $11136593 - 32083714\,\theta - 212361744\,\theta^2$)\\
$\theta_{79}\approx0.166815$ ($=5621/33696$)\\
$\theta_{80}\approx0.168091$ ($=14749/87744$)\\
$\theta_{81}\approx0.170083$ (root of $26417573 + 78007830\,\theta - 1371860048\,\theta^2$)\\
$\theta_{82}\approx0.175738$ ($=1290144/7341293$)\\
$\theta_{83}\approx0.176747$ (root of $3496103 - 200622112\,\theta + 1023165504\,\theta^2$)\\
$\theta_{84}\approx0.181200$ ($=6213/34288$)\\
$\theta_{85}\approx0.201498$ ($=20687/102666$)\\
$\theta_{86}\approx0.204058$ (root of $2709997 - 172154050\,\theta + 778569360\,\theta^2$)\\
$\theta_{87}\approx0.210558$ (root of $13549985 - 884171324\,\theta + 3893547552\,\theta^2$)\\
$\theta_{88}\approx0.215651$ (root of $29851341 - 1952936938\,\theta + 8414116560\,\theta^2$)\\
$\theta_{89}\approx0.215841$ (root of $314526381 - 1543673374\,\theta + 400585240\,\theta^2$)\\
$\theta_{90}\approx0.216321$ ($=20149/93144$)\\
$\theta_{91}\approx0.232970$ (root of $32913017 - 2145707382\,\theta + 8603825328\,\theta^2$)\\
$\theta_{92}\approx0.235148$ (root of $29851341 - 1957006690\,\theta + 7782577296\,\theta^2$)\\
$\theta_{93}\approx0.235426$ (root of $13549985 - 891632536\,\theta + 3542848128\,\theta^2$)\\
$\theta_{94}\approx0.241681$ ($=1676576/6937133$)\\
$\theta_{95}\approx0.254573$ ($=9588/37663$)\\
$\theta_{96}\approx0.255213$ ($=1591/6234$)\\
$\theta_{97}\approx0.282286$ ($=1141/4042$)\\
$\theta_{98}\approx0.284187$ ($=5796/20395$)\\
$\theta_{99}\approx0.291630$ ($=655/2246$)\\
$\theta_{100}\approx0.312202$ ($=655/2098$)\\
$\theta_{101}\approx0.315534$ ($=65/206$)\\
$\theta_{102}\approx0.317736$ ($=2397/7544$)\\
$\theta_{103}\approx0.322499$ ($=1141/3538$)\\
$\theta_{104}\approx0.332182$ ($=481/1448$)\\
$\theta_{105}\approx0.340483$ ($=2397/7040$)\\
$\theta_{106}\approx0.354185$ ($=1591/4492$)\\
$\theta_{107}\approx0.364699$ ($=655/1796$)\\
$\theta_{108}\approx0.377510$ ($=44268/117263$)\\
$\theta_{109}\approx0.377551$ ($=37/98$)\\
$\theta_{110}\approx0.384504$ ($=799/2078$)\\
$\theta_{111}\approx0.384544$ ($=2901/7544$)\\
$\theta_{112}\approx0.386798$ ($=46044/119039$)\\
$\theta_{113}\approx0.393855$ ($=141/358$)\\
$\theta_{114}\approx0.398947$ ($=1591/3988$)\\
$\theta_{115}\approx0.400756$ ($=1591/3970$)\\
$\theta_{116}\approx0.409866$ ($=1105/2696$)\\
$\theta_{117}\approx0.430543$ ($=967/2246$)\\
$\theta_{118}\approx0.432877$ ($=55716/128711$)\\
$\theta_{119}\approx0.433117$ (root of $702880555 - 559542428\,\theta - 2454987584\,\theta^2$)\\
$\theta_{120}\approx0.436721$ (root of $1845299991 - 4448512682\,\theta + 510999680\,\theta^2$)\\
$\theta_{121}\approx0.437496$ ($=56773/129768$)\\
$\theta_{122}\approx0.449657$ (root of $16528913 - 1049687266\,\theta + 2252670096\,\theta^2$)\\
$\theta_{123}\approx0.451530$ (root of $16653035 - 1057502284\,\theta + 2260360992\,\theta^2$)\\
$\theta_{124}\approx0.494544$ (root of $16528913 - 1043243492\,\theta + 2041923424\,\theta^2$)\\
$\theta_{125}\approx0.495084$ (root of $3864427485 - 3803622812493\,\theta + 42247177557968\,\theta^2$\\
\hglue1cm${} - 69847029193024\,\theta^3$)\\
$\theta_{126}\approx0.495197$ (root of $182368106747395200 + 881216426468870946\,\theta $\\
\hglue1cm${} - 4780752501776107548\,\theta^2 - 13041255680021548143\,\theta^3$\\
\hglue1cm${} + 64736996048499666656\,\theta^4 - 73515304998471077576\,\theta^5$\\
\hglue1cm${} + 37718067426331853440\,\theta^6 - 16799618113355584000\,\theta^7$)\\
$\theta_{127}\approx0.4958750$ (root of $1844123161 - 71232145332\,\theta + 284721068896\,\theta^2 $\\
\hglue1cm${} - 383623527168\,\theta^3 + 169415497728\,\theta^4$)\\
$\theta_{128}\approx0.4958752$ (root of $1207110151 - 30427610592\,\theta + 90295083136\,\theta^2 $\\
\hglue1cm${} - 75785852928\,\theta^3 + 15200022528\,\theta^4$)\\
$\theta_{129}\approx0.496125$ (root of $48577946559495994080 - 604139391703113490287\,\theta$\\
\hglue1cm${} + 2298236071942812885678\,\theta^2 - 1728960962420377483391\,\theta^3$\\
\hglue1cm${} - 5064706681585975840544\,\theta^4 + 1279180872137074137472\,\theta^5$\\
\hglue1cm${}  + 19263694030415189268480\,\theta^6 - 16527190457227168051200\,\theta^7$)\\
$\theta_{130}\approx0.496212$ (root of $139587367297173051168 - 135782955080259495272577\,\theta $\\
\hglue1cm${}+ 1873662221810254572134561\,\theta^2 - 6306085403811779286897552\,\theta^3 $\\
\hglue1cm${}+ 3889005095548638820291008\,\theta^4 + 4461194696950345843390464\,\theta^5$\\
\hglue1cm${} + 13439789031937874190594048\,\theta^6 - 35092289343041827660431360\,\theta^7$\\
\hglue1cm${} + 17873545461788238682521600\,\theta^8$)\\
\end{small}
\bigskip

\begin{small}
Banker's best response to the Players' maximin strategy in the independent cooperative equilibrium.  (Not Banker's minimax strategy, for which see Appendix A.)  Pairs followed by a colon are Player~1's and Player~2's third-card values (10 = stand, 11 = natural).  Entries are maximum Banker drawing totals (e.g., 5 means that Banker draws on 5 or less and stands on 6 or 7).  Plus signs indicate that Banker is indifferent on next higher total (e.g., 5+ means Banker draws on 5 or less, is indifferent on 6, and stands on 7).  Two plus signs indicate that Banker is indifferent on next two higher totals.  Number ranges in parentheses refer to the 131 intervals in $(0,1/2]$, with interval $i$ being $(\theta_{i-1},\theta_i)$; here $\theta_0=0$ and $\theta_{131}=1/2$.\medskip

$(0,0)$: 3.  
$(0,1)$: 3.  
$(0,2)$: 4 (1--119), 3 (120--131).  
$(0,3)$: 4.  
$(0,4)$: 5 (1--10), 4 (11--131).  
$(0,5)$: 5 (1--105), 4 (106--131).  
$(0,6)$: 6 (1--11), 5+ (12), 5 (13--111), 4 (112--131).  
$(0,7)$: 6 (1--37), 5+ (38--40), 5 (41--42), 4+ (43--44), 4 (45--46), 4+ (47--49), 4 (50--100), 3 (101--131).  
$(0,8)$: 2 (1--33), 2+ (34--35), 3 (36--40), 2+ (41--42), 3 (43--131).  
$(0,9)$: 3.  
$(0,10)$: 5.  
$(0,11)$: 3.

$(1,0)$: 3.  
$(1,1)$: 3.  
$(1,2)$: 4.  
$(1,3)$: 4.  
$(1,4)$: 5 (1--8), 4 (9--131).  
$(1,5)$: 5 (1--113), 4 (114--131).  
$(1,6)$: 6 (1--4), 5 (5--120), 4 (121--131).  
$(1,7)$: 6 (1--31), 5 (32--41), 4 (42--43), 5 (44--45), 4+ (46), 5 (47--52), 4+ (53--55), 5 (56), 4+ (57), 4 (58--131).  
$(1,8)$: 2 (1--26), 3 (27--131).  
$(1,9)$: 3.  
$(1,10)$: 5.  
$(1,11)$: 3.

$(2,0)$: 3.  
$(2,1)$: 3 (1--66), 3+ (67), 3 (68--69), 3+ (70--71), 4 (72--74), 3 (75--77), 4 (78--131).  
$(2,2)$: 4.  
$(2,3)$: 4.  
$(2,4)$: 5 (1--17), 4 (18--131).  
$(2,5)$: 5 (1--124), 4 (125--126), 4+ (127), 4 (128--131).  
$(2,6)$: 6 (1--3), 5 (4--131).  
$(2,7)$: 6 (1--32), 5 (33--63), 4+ (64--65), 5 (66--86), 4 (87--131).  
$(2,8)$: 2 (1--14), 3 (15--131).  
$(2,9)$: 3.  
$(2,10)$: 5.  
$(2,11)$: 4.

$(3,0)$: 3 (1--109), 4 (110--131).  
$(3,1)$: 3 (1--29), 4 (30--131).  
$(3,2)$: 4.  
$(3,3)$: 4.  
$(3,4)$: 5 (1--34), 4 (35--131).  
$(3,5)$: 5.  
$(3,6)$: 6 (1--7), 5 (8--131).  
$(3,7)$: 6 (1--39), 5 (40--107), 4 (108--131).  
$(3,8)$: 2 (1-4), 3 (5--115), 4 (116--131).  
$(3,9)$: 3.  
$(3,10)$: 5.  
$(3,11)$: 4.

$(4,0)$: 3 (1--92), 4 (93--131).  
$(4,1)$: 3 (1--13), 4 (14--131).  
$(4,2)$: 4.  
$(4,3)$: 4.  
$(4,4)$: 5.  
$(4,5)$: 5.  
$(4,6)$: 6 (1--15), 5 (16--131).  
$(4,7)$: 6 (1--59), 5+ (60), 5 (61--131).  
$(4,8)$: 2 (1), 3 (2--96), 4 (97--131).  
$(4,9)$: 3 (1--105), 4 (106--131).  
$(4,10)$: 5.  
$(4,11)$: 4 (1--20), 5 (21--22), 4 (23), 4+ (24), 4 (25--27), 4+ (28--30), 5 (31--61), 4+ (62), 4 (63--64), 5 (65--131).

$(5,0)$: 3 (1--88), 3+ (89), 4 (90--131).  
$(5,1)$: 3 (1--9), 4 (10--131).  
$(5,2)$: 4.  
$(5,3)$: 4 (1--103), 5 (104--131).  
$(5,4)$: 5.  
$(5,5)$: 5.  
$(5,6)$: 6 (1--28), 5 (29--131).  
$(5,7)$: 6 (1--93), 5 (94--131).  
$(5,8)$: 2 (1--6), 3 (7--91), 4 (92--114), 5 (115--131).  
$(5,9)$: 3 (1--102), 4 (103--131).  
$(5,10)$: 6 (1--12), 5+ (13--46), 5 (47--58), 5+ (59--66), 5 (67--131).  
$(5,11)$: 5.

$(6,0)$: 3 (1--104), 4 (105--131).  
$(6,1)$: 3 (1--18), 4 (19--131).  
$(6,2)$: 4 (1--123), 5 (124--131).  
$(6,3)$: 4 (1--97), 5 (98--131).  
$(6,4)$: 5.  
$(6,5)$: 5.  
$(6,6)$: 6.  
$(6,7)$: 6.  
$(6,8)$: 2 (1--16), 3 (17--106), 5 (107--131).  
$(6,9)$: 3 (1--122), 4 (123--131).  
$(6,10)$: 6 (1--127), 5+ (128), 6 (129--130), 5+ (131).  
$(6,11)$: 6.

$(7,0)$: 3.  
$(7,1)$: 3 (1--35), 3+ (36), 4 (37--78), 3 (79--81), 4 (82),  3+ (83), 4 (84--131).  
$(7,2)$: 4.  
$(7,3)$: 4.  
$(7,4)$: 5.  
$(7,5)$: 5.  
$(7,6)$: 6.  
$(7,7)$: 6.  
$(7,8)$: 2 (1--38), 3 (39--131).  
$(7,9)$: 3.  
$(7,10)$: 6.  
$(7,11)$: 6.

$(8,0)$: 3.  
$(8,1)$: 3.  
$(8,2)$: 4 (1--116), 3 (117--131).  
$(8,3)$: 4.  
$(8,4)$: 5 (1--46), 4+ (47), 4 (48--51), 4+ (52), 5 (53), 4+ (54), 4 (55--58), 5 (59--67), 4+ (68), 4 (69--131).  
$(8,5)$: 5.  
$(8,6)$: 6 (1--44), 5+ (45), 5 (46--58), 5+ (59), 5 (60), 5+ (61), 5 (62), 5+ (63), 5 (64--131).  
$(8,7)$: 6 (1--99), 3 (100--131).  
$(8,8)$: 2.  
$(8,9)$: 3.  
$(8,10)$: 6 (1--48), 5+ (49--51), 6 (52--55), 5+ (56--58), 6 (59--70), 5+ (71--73), 5 (74--131). 
$(8,11)$: 2. 

$(9,0)$: 3.  
$(9,1)$: 3.  
$(9,2)$: 4 (1--101), 3 (102--131).  
$(9,3)$: 4 (1--128), 3+ (129), 3 (130--131).  
$(9,4)$: 5 (1--21), 4+ (22--23), 4 (24--131).  
$(9,5)$: 5 (1--110), 4 (111--131).  
$(9,6)$: 6 (1--24), 5+ (25), 5 (26--117), 4 (118--131).  
$(9,7)$: 6 (1--72), 4 (73--75), 4++ (76), 4 (77--87), 3 (88--131).  
$(9,8)$: 2 (1--85), 3 (86--131).  
$(9,9)$: 3.  
$(9,10)$: 5+ (1--11), 5 (12--131).  
$(9,11)$: 2 (1--49), 2+ (50), 2 (51--57), 2+ (58), 2 (59--73), 2+ (74--75), 3 (76--131).

$(10,0)$: 3 (1--65), 3+ (66--70), 3 (71), 3+ (72--79), 4 (80--118), 4+ (119--121), 5 (122--131).  
$(10,1)$: 3 (1--5), 4 (6--108), 5 (109--131).  
$(10,2)$: 4 (1--98), 5 (99--131).  
$(10,3)$: 4 (1--80), 4+ (81--82), 5 (83--131).  
$(10,4)$: 5.  
$(10,5)$: 5.  
$(10,6)$: 6 (1--19), 5+ (20--21), 5 (22--79), 5+ (80), 5 (81--83), 5+ (84--88), 6 (89--90), 5+ (91--94), 6 (95--121), 5+ (122--131).  
$(10,7)$: 6.  
$(10,8)$: 2 (1--2), 3 (3--84), 4 (85--89), 4+ (90), 5 (91--131).  
$(10,9)$: 3 (1--95), 4 (96--112), 5 (113--131).  
$(10,10)$: 5 (1--30), 5+ (31--33), 5 (34--36), 5+ (37), 5 (38--94), 5+ (95--118), 5 (119--131).  
$(10,11)$: 5 (1--94), 5+ (95--118), 5 (119--131).

$(11,0)$: 3.  
$(11,1)$: 3.  
$(11,2)$: 4.  
$(11,3)$: 4.  
$(11,4)$: 5.  
$(11,5)$: 5.  
$(11,6)$: 6.  
$(11,7)$: 6.  
$(11,8)$: 2.  
$(11,9)$: 3.  
$(11,10)$: 6 (1--76), 5+ (77--125), 5 (126--131).
\end{small}\bigskip

The information sets on which Banker is indifferent in his best response to the Players' maximin strategy in the independent cooperative equilibrium.
\par\medskip

\tabcolsep=1.5mm
\begin{center}
\begin{small}
\begin{tabular}{cc}
\hline
\noalign{\smallskip}
intervals & information sets \\
\noalign{\smallskip} \hline
\noalign{\smallskip}
1--11 & $(9,10,6)$ \\
12 & $(0,6,6)$ \\
13--19, 26--27 & $(5,10,6)$ \\
20--21 & $(5,10,6)$, $(10,6,6)$ \\
22--23 & $(5,10,6)$, $(9,4,5)$ \\
24, 28--30 & $(4,11,5)$, $(5,10,6)$ \\
25 & $(5,10,6)$, $(9,6,6)$ \\
31--33, 37 & $(5,10,6)$, $(10,10,6)$ \\
34--35, 41--42 & $(0,8,3)$, $(5,10,6)$ \\
\noalign{\smallskip}
\hline
\end{tabular}
\end{small}
\end{center}

\tabcolsep=1.5mm
\begin{center}
\begin{small}
\begin{tabular}{cc}
\hline
\noalign{\smallskip}
intervals & information sets \\
\noalign{\smallskip} \hline
\noalign{\smallskip}
36 & $(5,10,6)$, $(7,1,4)$ \\ 
38--40 & $(0,7,6)$, $(5,10,6)$ \\ 
43--44 & $(0,7,5)$, $(5,10,6)$ \\ 
45, 59, 61, 63 & $(5,10,6)$, $(8,6,6)$ \\ 
46 & $(1,7,5)$, $(5,10,6)$ \\ 
47 & $(0,7,5)$, $(8,4,5)$ \\ 
48 & $(0,7,5)$ \\ 
49 & $(0,7,5)$, $(8,10,6)$ \\ 
50, 58 & $(8,10,6)$, $(9,11,3)$ \\
51, 56 & $(8,10,6)$ \\ 
52 & $(8,4,5)$ \\ 
53, 55 & $(1,7,5)$ \\ 
54 & $(1,7,5)$, $(8,4,5)$ \\
57 & $(1,7,5)$, $(8,10,6)$ \\ 
60 & $(4,7,6)$, $(5,10,6)$ \\
62 & $(4,11,5)$, $(5,10,6)$ \\ 
64--65 & $(2,7,5)$, $(5,10,6)$ \\ 
66 & $(5,10,6)$, $(10,0,4)$ \\
67, 70 & $(2,1,4)$, $(10,0,4)$ \\
68 & $(8,4,5)$, $(10,0,4)$ \\
69 & $(10,0,4)$ \\
71 & $(2,1,4)$, $(8,10,6)$ \\
72--73 & $(8,10,6)$, $(10,0,4)$ \\
74--75 & $(9,11,3)$, $(10,0,4)$ \\
76 & $(9,7,5)$, $(9,7,6)$, $(10,0,4)$ \\
77--79 & $(10,0,4)$, $(11,10,6)$ \\
80, 84--88, 91--94, 122--125 & $(10,6,6)$, $(11,10,6)$ \\
81--82 & $(10,3,5)$, $(11,10,6)$ \\
83 & $(7,1,4)$, $(11,10,6)$ \\
89 & $(5,0,4)$, $(11,10,6)$ \\
90 & $(10,8,5)$, $(11,10,6)$ \\
95--118 & $(10,10,6)$, $(10,11,6)$, $(11,10,6)$ \\
119--121 & $(10,0,5)$, $(11,10,6)$ \\
126, 130 & $(10,6,6)$ \\
127 & $(2,5,5)$, $(10,6,6)$ \\
128, 131 & $(6,10,6)$, $(10,6,6)$ \\
129 & $(9,3,4)$, $(10,6,6)$ \\
\noalign{\smallskip}
\hline
\end{tabular}
\end{small}
\end{center}
\bigskip

\begin{small}

The Players' maximin strategies in the independent cooperative equilibrium.  $p_1$ (resp., $p_2$) is the probability that Player 1 (resp., Player 2) draws on 5.\medskip

\noindent(intervals 1--11)
\begin{eqnarray*}
p_1&=&0, \quad p_2=9 (89 - 12 \,\theta)/[11(89-68\,\theta)].
\end{eqnarray*}

\noindent(interval 12)
\begin{eqnarray*}
p_1&=&0, \quad p_2=89 (53 - 462 \,\theta)/[16 (267 + 142 \,\theta)].
\end{eqnarray*}

\noindent(intervals 13--19, 26--27)
\begin{eqnarray*}
p_1&=&0, \quad p_2=3(267 - 88\,\theta)/(979-800 \,\theta).
\end{eqnarray*}

\noindent(intervals 20--21)
\begin{eqnarray*}
p_1&=&[-6191 + 962008 \,\theta - 705792 \,\theta^2 - (427951969 - 54744421840 \,\theta \\
&&{}+ 1809650847808 \,\theta^2  - 2714674851840 \,\theta^3 + 1043139723264 \,\theta^4)^{1/2}]\\
&&{}/[32 (151 - 12746 \,\theta + 11088 \,\theta^2)],\\
p_2&=&[139055 - 1171096\,\theta + 945408\,\theta^2 + (427951969 - 54744421840\,\theta \\
&&{}+ 1809650847808\,\theta^2 - 2714674851840\,\theta^3 + 1043139723264\,\theta^4)^{1/2}]\\
&&{}/[32 (4521 + 3992 \,\theta - 6400 \,\theta^2)].
\end{eqnarray*}

\noindent(intervals 22--23)
\begin{eqnarray*}
p_1&=&[-11837 + 596844 \,\theta - 458830 \,\theta^2 - 30  \,\theta (114287524 - 161786452 \,\theta\\
&&{}+ 57427729 \,\theta^2)^{1/2}]/[16(133 - 10326 \,\theta + 8960 \,\theta^2)],\\
p_2&=&[-8006 + 7063 \,\theta+ (114287524 - 161786452 \,\theta + 57427729 \,\theta^2)^{1/2}]\\
&&{}/[32 (103-80 \,\theta)].
\end{eqnarray*}

\noindent(intervals 24, 28--30)
\begin{eqnarray*}
p_1&=&1/16, \quad p_2=3(3-\theta)/(11-9 \,\theta).
\end{eqnarray*}

\noindent(interval 25)
\begin{eqnarray*}
p_1&=&[-13439 + 618540 \,\theta - 478924 \,\theta^2 - 30 \,\theta (109743865 - 151394620 \,\theta\\
&&{} + 51579556 \,\theta^2)^{1/2}]/[16(151 - 10500 \,\theta + 9116 \,\theta^2)],\\
p_2&=& [-7637 + 6694 \,\theta + (109743865 - 151394620 \,\theta + 51579556 \,\theta^2)^{1/2}]\\
&&{}/[32 (109 - 86 \,\theta)].
\end{eqnarray*}

\noindent(intervals 31--33, 37)
\begin{eqnarray*}
p_1&=& (-22+\sqrt{649})/10,\\
p_2&=&3 [4521 + 408 \,\theta - 2816 \,\theta^2 - 128\sqrt{649} \,\theta (1-\theta)]\\
&&{}/(16577 - 18304 \,\theta + 3840 \,\theta^2).
\end{eqnarray*}

\noindent(intervals 34--35, 41--42)
\begin{eqnarray*}
p_1&=&\{-13439 + 528700 \,\theta - 400044 \,\theta^2 - 10\,\theta [11(148467731 - 217532180 \,\theta\\
&&{} + 79180940 \,\theta^2)]^{1/2}\}/[16 (151 - 10500 \,\theta + 9116 \,\theta^2)],\\
p_2&=&\{-29591 + 26434\,\theta + [11(148467731 - 217532180\,\theta + 79180940\,\theta^2)]^{1/2}\}\\
&&{}/[32 (415 - 338\,\theta)].
\end{eqnarray*}

\noindent(interval 36)
\begin{eqnarray*}
p_1&=& [-15041 - 118912 \,\theta + 81208 \,\theta^2 + 30 \,\theta (122583865 - 197073880 \,\theta\\
&&{}+ 79914256 \,\theta^2)^{1/2}]/[16 (169 + 5048 \,\theta- 4532 \,\theta^2)],\\
p_2&=&[-8597 + 7900 \,\theta + (122583865 - 197073880 \,\theta + 79914256 \,\theta^2)^{1/2}]\\
&&{}/[32 (94-77 \,\theta)].
\end{eqnarray*}

\noindent(intervals 38--40)
\begin{eqnarray*}
p_1&=&[-58295 + 911362 \,\theta - 690174 \,\theta^2 - 30\,\theta(105051121 - 132870958 \,\theta\\
&&{}+ 40507681 \,\theta^2)^{1/2}]/[16 (655 - 13778 \,\theta + 11616 \,\theta^2)],\\
p_2&=&[-7385 + 6319\,\theta + (105051121 - 132870958\,\theta + 40507681\,\theta^2)^{1/2}]\\
&&{}/[64 (56 - 43\,\theta)].
\end{eqnarray*}

\noindent(intervals 43--44)
\begin{eqnarray*}
p_1&=&[-58295 + 411627 \,\theta - 313054 \,\theta^2 + 5 \,\theta (1470229561 - 2720513020 \,\theta \\
&&{}+ 1260212260 \,\theta^2)^{1/2}]/[16 (655 - 2548 \,\theta + 1756 \,\theta^2)],\\
p_2&=&[-30005 + 29062 \,\theta + (1470229561 - 2720513020 \,\theta  + 1260212260 \,\theta^2)^{1/2}]\\
&&{}/[32 (307-284 \,\theta)].
\end{eqnarray*} 

\noindent(intervals 45, 59, 61, 63)
\begin{eqnarray*}
p_1&=&\{-13439 + 418646 \,\theta - 315746 \,\theta^2 - 30 \,\theta  [5(21948773 - 32055266 \,\theta\\
&&{}+ 11608013 \,\theta^2)]^{1/2}\}/[16(151 - 8254 \,\theta + 7144 \,\theta^2)],\\
p_2&=&\{-7637 + 6817 \,\theta + [5(21948773 - 32055266 \,\theta + 11608013 \,\theta^2)]^{1/2}\}\\
&&{}/[32 (109-89 \,\theta)].
\end{eqnarray*}

\noindent(interval 46)
\begin{eqnarray*}
p_1&=&[-58295 + 21946 \,\theta- 20232 \,\theta^2 + 120 \,\theta (25153300 - 40965940 \,\theta \\
&&{}+ 16732321 \,\theta^2)^{1/2}]/[16(655 + 6436 \,\theta - 6132 \,\theta^2)],\\
p_2&=&[-3790 +3503 \,\theta + (25153300 - 40965940 \,\theta + 16732321 \,\theta^2)^{1/2}]\\
&&{}/[8 (185 - 157 \,\theta)].
\end{eqnarray*}

\noindent(interval 47)
\begin{eqnarray*}
p_1&=&(-178+ 1295 \,\theta)/[32 (1 - 35 \,\theta)], \quad
p_2= (7093-32102 \,\theta)/[16 (253 + 28 \,\theta)].
\end{eqnarray*}

\noindent(interval 48)
\begin{eqnarray*}
p_1&=&\{(-445 + 36 \,\theta) (340880 - 429027 \,\theta - 1074404 \,\theta^2) + 40 [114 \,\theta (1 - \theta) (340880 \\
&&{}- 429027 \,\theta - 1074404 \,\theta^2) (2824835 + 4392257 \,\theta - 1954788 \,\theta^2)]^{1/2}\}\\
&&{}/[16 (5 - 4 \,\theta) (340880 - 429027 \,\theta - 1074404 \,\theta^2)],\\
p_2&=&\{(125 - 214 \,\theta) (2824835 + 4392257 \,\theta - 1954788 \,\theta^2) - 40 [114 \,\theta (1 - \theta) (340880 \\
&&{}- 429027 \,\theta - 1074404 \,\theta^2) (2824835 + 4392257 \,\theta - 1954788 \,\theta^2)]^{1/2}\}\\
&&{} /[16 (5 - 4 \,\theta) (2824835 + 4392257 \,\theta - 1954788 \,\theta^2)].
\end{eqnarray*}

\noindent(interval 49)
\begin{eqnarray*}
p_1&=&[-58295 + 435087 \,\theta - 177046 \,\theta^2 + 5 \,\theta (208536601 + 2854478852 \,\theta \\
&&{}- 1580996444 \,\theta^2)^{1/2}]/[16 (655 - 3238 \,\theta + 1624 \,\theta^2)],\\
p_2&=&[-5915 - 5606 \,\theta + (208536601 + 2854478852 \,\theta - 1580996444 \,\theta^2)^{1/2}]\\
&&{}/[32 (637-356 \,\theta)].
\end{eqnarray*}

\noindent(intervals 50, 58)
\begin{eqnarray*}
p_1&=&71/176, \quad p_2=9 (35-23  \,\theta)/(385-349 \,\theta).
\end{eqnarray*}

\noindent(intervals 51, 56)
\begin{eqnarray*}
p_1&=&\{(-979 + 926 \,\theta) (103136 - 2222227 \,\theta - 1082268 \,\theta^2) - 48 [-95 \,\theta (1 - \theta) (103136 \\
&&{}- 2222227 \,\theta - 1082268 \,\theta^2) (6214637 + 4078519 \,\theta - 6682468 \,\theta^2)]^{1/2}\}\\
&&{}/[16 (11 - 4 \,\theta) (103136 - 2222227 \,\theta - 1082268 \,\theta^2)],\\
p_2&=& 3\{(3 + 4 \,\theta) (6214637 + 4078519 \,\theta - 6682468 \,\theta^2) - [-95 \,\theta (1 - \theta) (103136 \\
&&{}- 2222227 \,\theta - 1082268 \,\theta^2) (6214637 + 4078519 \,\theta - 6682468 \,\theta^2)]^{1/2}\}\\
&&{}/[(11 - 4 \,\theta) (6214637 + 4078519 \,\theta - 6682468 \,\theta^2)].
\end{eqnarray*}

\noindent(interval 52)
\begin{eqnarray*}
p_1&=& \{(-89 + 142 \,\theta) (88112 - 4797177 \,\theta - 3428792 \,\theta^2)- 30 [-57 \,\theta (1 - \theta) (88112 - \\
&&{}4797177 \,\theta - 3428792 \,\theta^2) (564967 + 105761 \,\theta - 4281416 \,\theta^2)]^{1/2}\}\\
&&{}/[16 (1 - 8 \,\theta) (88112 - 4797177 \,\theta - 3428792 \,\theta^2)],\\
p_2&=&\{(1 + 622 \,\theta) (564967 + 105761 \,\theta - 4281416 \,\theta^2)- 30 [-57 \,\theta (1 - \theta) (88112 \\
&&{}- 4797177 \,\theta - 3428792 \,\theta^2) (564967 + 105761 \,\theta - 4281416 \,\theta^2)]^{1/2}\}\\
&&{}/[16 (1- 8 \,\theta) (564967 + 105761 \,\theta - 4281416 \,\theta^2)].
\end{eqnarray*}
 
\noindent(interval 53)
\begin{eqnarray*}
p_1&=& \{(-445 + 18 \,\theta) (342800 + 4906909 \,\theta + 2888148 \,\theta^2) + 30 [665 \,\theta (1 - \theta) (342800 \\
&&{}+ 4906909 \,\theta + 2888148 \,\theta^2) (2824835 + 4597481 \,\theta - 5411244 \,\theta^2)]^{1/2}\}\\
&&{}/[16 (5- 12 \,\theta) (342800 + 4906909 \,\theta + 2888148 \,\theta^2)],\\
p_2&=&\{(125 + 498 \,\theta) (2824835 + 4597481 \,\theta - 5411244 \,\theta^2)- 30 [665 \,\theta (1 - \theta) (342800 \\
&&{}+ 4906909 \,\theta + 2888148 \,\theta^2) (2824835 + 4597481 \,\theta - 5411244 \,\theta^2)]^{1/2}\}\\
&&{}/[16 (5 - 12 \,\theta) (2824835 + 4597481 \,\theta - 5411244 \,\theta^2)].
\end{eqnarray*}

\noindent(interval 54)
\begin{eqnarray*}
p_1&=&(-89 +661 \,\theta)/[16 (1 - 29 \,\theta)],\quad p_2=(585 - 3077 \,\theta)/[16 (15 + 13 \,\theta)].
\end{eqnarray*}

\noindent(interval 55)
\begin{eqnarray*}
p_1&=& \{(-445 + 18 \,\theta) (340880 + 4964509 \,\theta + 2832468 \,\theta^2) + 30 [665 \,\theta (1 - \theta) (340880 \\
&&{}+ 4964509 \,\theta + 2832468 \,\theta^2) (2824835 + 4626281 \,\theta - 5386284 \,\theta^2)]^{1/2}\}\\
&&{}/[16 (5 - 12 \,\theta) (340880 + 4964509 \,\theta + 2832468 \,\theta^2)],\\
p_2&=& \{(125 + 498 \,\theta) (2824835 + 4626281 \,\theta - 5386284 \,\theta^2)- 30 [665 \,\theta (1 - \theta) (340880 \\
&&{}+ 4964509 \,\theta + 2832468 \,\theta^2) (2824835 + 4626281 \,\theta - 5386284 \,\theta^2)]^{1/2}\}\\
&&{}/[16 (5 - 12 \,\theta) (2824835 + 4626281 \,\theta - 5386284 \,\theta^2)].
\end{eqnarray*} 

\noindent(interval 57)
\begin{eqnarray*}
p_1&=& [-58295 + 45406 \,\theta + 192222 \,\theta^2 + 30 \,\theta (284799825 - 251288370 \,\theta \\
&&{}+ 25348129 \,\theta^2)^{1/2}] /[16 (655 + 5746 \,\theta + 312 \,\theta^2)],\\
p_2&=&[-11145 + 8849 \,\theta + (284799825 - 251288370 \,\theta + 25348129 \,\theta^2)^{1/2}]\\
&&{}/[256 (30-23 \,\theta))].
\end{eqnarray*}

\noindent(interval 60)
\begin{eqnarray*}
p_1&=& [-58295 + 262268 \,\theta - 194246 \,\theta^2 + 10 \,\theta (587196736 - 1020241072 \,\theta \\
&&{}+ 444774961 \,\theta^2)^{1/2}]/[16 (655 - 302 \,\theta - 216 \,\theta^2)],\\
p_2&=&[5(-3752 + 3547 \,\theta) +  (587196736 - 1020241072 \,\theta + 444774961 \,\theta^2)^{1/2}]\\
&&{}/[32 (203-178 \,\theta)].
\end{eqnarray*}

\noindent(interval 62)
\begin{eqnarray*}
p_1&=&1/16, \quad p_2=3 (3- \theta)/(11-9 \,\theta).
\end{eqnarray*}

\noindent(intervals 64--65)
\begin{eqnarray*}
p_1&=&[-11659 + 8432 \,\theta - 5130 \,\theta^2 + 2 \,\theta (1781786944 - 2909303536 \,\theta \\
&&{}+ 1192851481 \,\theta^2)^{1/2}]/[16 (131 + 838 \,\theta - 832 \,\theta^2)],\\
p_2&=&[-32120  + 29701 \,\theta+ (1781786944 - 2909303536 \,\theta + 1192851481 \,\theta^2)^{1/2}]\\
&&{}/[32  (379-320 \,\theta)].
\end{eqnarray*}

\noindent(interval 66)
\begin{eqnarray*}
p_1&=&[-19721 + 221498 \,\theta - 143552 \,\theta^2 - (4342414609 - 75288376852 \,\theta \\
&&{}+ 382769338276 \,\theta^2  - 478895079168 \,\theta^3 + 174977224704 \,\theta^4)^{1/2}]\\
&&{}/[32 (481 - 4392 \,\theta + 3500 \,\theta^2)],\\
p_2&=&[184265 - 633466 \,\theta + 422592 \,\theta^2 +(4342414609 - 75288376852 \,\theta \\
&&{}+ 382769338276 \,\theta^2 - 478895079168 \,\theta^3 + 174977224704 \,\theta^4)^{1/2}]\\
&&{}/[32(4521 - 4192 \,\theta + 320 \,\theta^2)].
\end{eqnarray*}

\noindent(intervals 67, 70)
\begin{eqnarray*}
p_1&=&[-533 + 9778 \,\theta -(3171961 - 38476724 \,\theta + 150204804 \,\theta^2)^{1/2}]\\
&&{}/[32 (13 - 35 \,\theta)],\\
p_2&=&[61787 - 246656 \,\theta + 125684 \,\theta^2 + 10(1-\theta) (3171961 - 38476724 \,\theta \\
&&{}+ 150204804 \,\theta^2)^{1/2}]/[16 (137-4 \,\theta) (21-16 \,\theta)].
\end{eqnarray*}

\noindent(interval 68)
\begin{eqnarray*}
p_1&=&[-164 - 4091 \,\theta +(300304 - 9930856 \,\theta + 76356121 \,\theta^2)^{1/2}]/[128 (1 - 35 \,\theta)],\\
p_2&=&[5891 + 7330 \,\theta + 70706 \,\theta^2 - 10 (1- \theta) (300304 - 9930856 \,\theta + 76356121 \,\theta^2)^{1/2}]\\
&&{}/[16 (411 - 1610 \,\theta + 256 \,\theta^2)].
\end{eqnarray*}

\noindent(interval 69)
\begin{eqnarray*}
p_1&=&\{2 (9 - 2 \,\theta) (257904 - 1987157 \,\theta - 1758400 \,\theta^2)+ 5 [-7 \,\theta (1 - \theta) (1047291 \\
&&{}- 819531 \,\theta - 160768 \,\theta^2) (257904 - 1987157 \,\theta - 1758400 \,\theta^2)]^{1/2}\}\\
&&{}/[6 (257904 - 1987157 \,\theta - 1758400 \,\theta^2)],\\
p_2&=& \{(83 - 350 \,\theta) (1047291 - 819531 \,\theta - 160768 \,\theta^2) + 40 [-7 \,\theta (1 - \theta) (1047291 \\
&&{}- 819531 \,\theta - 160768 \,\theta^2) (257904 -1987157 \,\theta - 1758400 \,\theta^2)]^{1/2}\}\\
&&{}/[48 (1047291 - 819531 \,\theta - 160768 \,\theta^2)].
\end{eqnarray*}

\noindent(interval 71)
\begin{eqnarray*}
p_1&=&[-15041 + 377004 \,\theta - 327028 \,\theta^2 - 10 \,\theta (927192865 - 1849866100 \,\theta \\
&&{} + 922692004\,\theta^2)^{1/2}]/[16 (169 - 7556 \,\theta + 2592 \,\theta^2)],\\
p_2&=&[-29743 + 29702 \,\theta + (927192865 - 1849866100 \,\theta + 922692004 \,\theta^2)^{1/2}]\\
&&{}/[32 (26-25 \,\theta)].
\end{eqnarray*}

\noindent(intervals 72--73)
\begin{eqnarray*}
p_1&=&[-19721 + 196244 \,\theta - 262148 \,\theta^2 - (4342414609 - 76590765160 \,\theta \\
&&{}+ 378797212888 \,\theta^2 - 429088870560 \,\theta^3 + 139230322704 \,\theta^4)^{1/2}]\\
&&{}/[32 (481 - 4758 \,\theta + 1400 \,\theta^2)],\\
p_2&=& [184265 - 596692 \,\theta + 373764 \,\theta^2 + (4342414609 - 76590765160 \,\theta \\
&&{}+ 378797212888 \,\theta^2- 429088870560 \,\theta^3 + 139230322704 \,\theta^4)^{1/2}]\\
&&{}/[32 (4521 - 3706 \,\theta + 128 \,\theta^2)].
\end{eqnarray*}

\noindent(intervals 74--75)
\begin{eqnarray*}
p_1&=&71/176, \quad p_2=(37931-128622  \,\theta)/[16 (1371 - 352 \,\theta)].
\end{eqnarray*}

\noindent(interval 76)
\begin{eqnarray*}
p_1&=&[-82 + 8017 \,\theta - (75076 + 288796 \,\theta + 36666529 \,\theta^2)^{1/2}]/[64 (1 + 35 \,\theta)],\\
p_2&=&[54115 - 229866 \,\theta + 24362 \,\theta^2 + 10 (1- \theta) (75076 + 288796 \,\theta \\
&&{}+ 36666529 \,\theta^2)^{1/2}]/[16(2055-226 \,\theta -128 \,\theta^2)].
\end{eqnarray*}

\noindent(intervals 77--79)
\begin{eqnarray*}
p_1&=&(1443-9304 \,\theta)/(481-3850 \,\theta), \quad p_2= 9/11.
\end{eqnarray*}

\noindent(intervals 80, 84--88, 91--94, 122--125)
\begin{eqnarray*}
p_1&=& 3 (151 - 3520 \,\theta)/(151-12504 \,\theta), \quad p_2=9/11.
\end{eqnarray*}

\noindent(intervals 81--82)
\begin{eqnarray*}
p_1&=&(3423-15776 \,\theta)/(1141-18 \,\theta), \quad p_2=9/11.
\end{eqnarray*}

\noindent(interval 83)
\begin{eqnarray*}
p_1&=&(-15041+155416 \,\theta)/[16 (169 + 5446 \,\theta)], \quad p_2=9/11.
\end{eqnarray*}

\noindent(interval 89)
\begin{eqnarray*}
p_1&=&481 (-267+1390 \,\theta)/[16 (1443 - 320 \,\theta)], \quad p_2=9/11.
\end{eqnarray*}

\noindent(interval 90)
\begin{eqnarray*}
p_1&=&(4773-17126 \,\theta)/(1591-468 \,\theta), \quad p_2=9/11.
\end{eqnarray*}

\noindent(intervals 95--118)
\begin{eqnarray*}
p_1&=&9/11, \quad p_2=9/11.
\end{eqnarray*}

\noindent(intervals 119--121)
\begin{eqnarray*}
p_1&=&(13929-26282 \,\theta)/(4643-3520\,\theta), \quad p_2=9/11.
\end{eqnarray*}

\noindent(interval 126)
\begin{eqnarray*}
p_1&=& 3 \{3 (1910976 - 1870549 \,\theta - 9828920 \,\theta^2)+ 4 [10 \,\theta (1 - \theta) (508455 \\
&&{}+ 490401 \,\theta - 3262312 \,\theta^2) (1910976 - 1870549 \,\theta - 9828920 \,\theta^2)]^{1/2}\}\\
&&{}/[(3 + 8 \,\theta) (1910976 - 1870549 \,\theta - 9828920 \,\theta^2)],\\
p_2&=& \{(53 - 1032 \,\theta) (508455 + 490401 \,\theta - 3262312 \,\theta^2) - 64 [10 \,\theta (1 - \theta) (508455 \\
&&{}+ 490401 \,\theta - 3262312 \,\theta^2) (1910976 - 1870549 \,\theta - 9828920 \,\theta^2)]^{1/2}\}\\
&&{}/[16 (3 + 8 \,\theta) (508455 + 490401 \,\theta - 3262312 \,\theta^2)].
\end{eqnarray*}

\noindent(interval 127)
\begin{eqnarray*}
p_1&=& [-41 - 114 \,\theta + (18769 + 110148 \,\theta - 118716 \,\theta^2)^{1/2}]/[32 (1 + 7 \,\theta)],\\
p_2&=&[51101 - 27284 \,\theta - 238752 \,\theta^2  +320 (1- \theta) (18769 + 110148 \,\theta \\
&&{}- 118716 \,\theta^2)^{1/2}]/[16(411 + 436 \,\theta + 1568 \,\theta^2)].
\end{eqnarray*}

\noindent(intervals 128, 131)
\begin{eqnarray*}
p_1&=& \{-6191 + 932160 \,\theta - 1065024 \,\theta^2 - [(20687 - 1065024 \,\theta + 1065024 \,\theta^2) (20687 \\
&&{}- 1556544 \,\theta + 1556544 \,\theta^2)]^{1/2}\}/[32 (151 - 12928 \,\theta + 8256 \,\theta^2)],\\
p_2&=& \{139055 - 1197888 \,\theta + 1065024 \,\theta^2 +[(20687 - 1065024 \,\theta + 1065024 \,\theta^2) (20687 \\
&&{} - 1556544 \,\theta + 1556544 \,\theta^2)]^{1/2}\}/ [32 (4521 + 3584 \,\theta - 8256 \,\theta^2)].
\end{eqnarray*}

\noindent(interval 129)
\begin{eqnarray*}
p_1&=&3 [-943+6424 \,\theta-(9928801 - 53390544 \,\theta + 87503424 \,\theta^2 )^{1/2}]\\
&&{}/[32 (69-136 \,\theta)],\\
p_2&=&[-8494 - 59587 \,\theta +139192 \,\theta^2 + 5 (1- \theta)(9928801 - 53390544 \,\theta \\
&&{}+ 87503424 \,\theta^2)^{1/2}]/ [16(411 + 118 \,\theta- 1328 \,\theta^2)].
\end{eqnarray*}

\noindent(interval 130)
\begin{eqnarray*}
p_1&=&\{9 (633312 - 612583 \,\theta - 3283560 \,\theta^2) + 4 [30 \,\theta (1 - \theta) (515031 + 492289 \,\theta \\
&&{}- 3283560 \,\theta^2) (633312 - 612583 \,\theta - 3283560 \,\theta^2)]^{1/2}\}\\
&&{} /[(3 + 8 \,\theta) (633312 - 612583 \,\theta - 3283560 \,\theta^2)], \\
p_2&=&\{(53 - 1032 \,\theta) (515031 + 492289 \,\theta - 3283560 \,\theta^2) - 64 [30 \,\theta (1 - \theta) (515031\\
&&{} + 492289 \,\theta - 3283560 \,\theta^2) (633312 - 612583 \,\theta - 3283560 \,\theta^2)]^{1/2}\}\\
&&{}/[16 (3 + 8 \,\theta) (515031 + 492289 \,\theta - 3283560 \,\theta^2)].
\end{eqnarray*}\medskip
\end{small}

\begin{small}
The lower value of the game (to the Players) as a function of $\theta$ (listed by interval number), assuming the independent cooperative equilibriums.  See Appendix A for the upper value. \medskip

\noindent 1. $-16 (8304873109-17736605719 \,\theta+8304947412 \,\theta^2)/[116649493103 (89-68 \,\theta)]$ \\
2. $-16 (8303677038-17666477635 \,\theta+8256870772 \,\theta^2)/[116649493103 (89-68 \,\theta)]$ \\
3. $-16 (8295291102-17200020163 \,\theta+7937033332 \,\theta^2)/[116649493103 (89-68 \,\theta)]$ \\
4. $-16 (8294095031-17135289217 \,\theta+7892644892 \,\theta^2)/[116649493103 (89-68 \,\theta)]$ \\
5. $-16 (8291702889-17009425417 \,\theta+7806326812 \,\theta^2)/[116649493103 (89-68 \,\theta)]$ \\
6. $-16 (8282317305-16521076777 \,\theta+7465597852 \,\theta^2)/[116649493103 (89-68 \,\theta)]$ \\
7. $-16 (8281121234-16468739259 \,\theta+7429678612 \,\theta^2)/[116649493103 (89-68 \,\theta)]$ \\
8. $-16 (8279925163-16418200787 \,\theta+7394988772 \,\theta^2)/[116649493103 (89-68 \,\theta)]$ \\
9. $-16 (8278871670-16374382705 \,\theta+7365737452 \,\theta^2)/[116649493103 (89-68 \,\theta)]$ \\
10. $-16 (8277533021-16318922889 \,\theta+7326838492 \,\theta^2)/[116649493103 (89-68 \,\theta)]$ \\
11. $-16 (8273319049-16150846745 \,\theta+7214750812 \,\theta^2)/[116649493103 (89-68 \,\theta)]$ \\
12. $-(36110836949-24752039219\,\theta+2625422702 \,\theta^2)/[10604499373 (267+142 \,\theta)]$ \\
13. $-16 (8268534765-16412250157 \,\theta+7670254664 \,\theta^2)/[10604499373 (979-800 \,\theta)]$ \\
14. $-16 (8267196116-16365046337 \,\theta+7634419400 \,\theta^2)/[10604499373 (979-800 \,\theta)]$ \\
15. $-16 (8266000045-16325723289 \,\theta+7604998088 \,\theta^2)/[10604499373 (979-800 \,\theta)]$ \\
16. $-48 (2754934658-5429932813 \,\theta+2526047064 \,\theta^2)/[10604499373 (979-800 \,\theta)]$ \\
17. $-16 (8263607903-16255672635 \,\theta+7552641928 \,\theta^2)/[10604499373 (979-800 \,\theta)]$ \\
18. $-16 (8262554410-16225829511 \,\theta+7530841352 \,\theta^2)/[10604499373 (979-800 \,\theta)]$ \\
19. $-48 (2753738587-5396406691 \,\theta+2501000344 \,\theta^2)/[10604499373 (979-800 \,\theta)]$ \\
20. $-[182273159958241-15306516372417582 \,\theta+6788491004299806 \,\theta^2 $\\
\hglue6mm${}+34767918791185696 \,\theta^3-38018026998518400 \,\theta^4+10893384025141248 \,\theta^5$\\
\hglue6mm${}-(104084497-7627696486 \,\theta-5690345714 \,\theta^2+10137326544 \,\theta^3)(427951969 $\\
\hglue6mm${}-54744421840 \,\theta+1809650847808 \,\theta^2-2714674851840 \,\theta^3+1043139723264 \,\theta^4)^{1/2}]$\\
\hglue6mm${}/[21208998746 (4521+3992 \,\theta-6400 \,\theta^2)(151-12746 \,\theta+11088 \,\theta^2)]$ \\
21. $-[182273159958241-15306711758319150 \,\theta+6818034109207070 \,\theta^2 $\\
\hglue6mm${}+34772528150205728 \,\theta^3-38079435575710336 \,\theta^4+10924394297192448 \,\theta^5$\\
\hglue6mm${}-(104084497-7597604710 \,\theta-5663774962 \,\theta^2 +10094728144 \,\theta^3)(427951969$\\
\hglue6mm${} -54744421840 \,\theta+1809650847808 \,\theta^2-2714674851840 \,\theta^3+1043139723264 \,\theta^4)^{1/2}]$\\
\hglue6mm${}/[21208998746 (4521+3992 \,\theta-6400 \,\theta^2)(151-12746 \,\theta+11088 \,\theta^2)]$ \\
22. $-[4489462838770 -354830408369689 \,\theta+778611736746975 \,\theta^2-575165852183354 \,\theta^3$\\
\hglue6mm${}+140999683659200\,\theta^4-(75140611-5687500861 \,\theta-5069198242 \,\theta^2$\\
\hglue6mm${}+8099636800 \,\theta^3)(114287524-161786452\,\theta +57427729 \,\theta^2)^{1/2}]$\\
\hglue6mm${}/[21208998746 (103-80 \,\theta) (133-10326 \,\theta+8960 \,\theta^2)]$ \\
23. $-[4489462838770 -354813995871769 \,\theta+777766476530655 \,\theta^2-573877836710714 \,\theta^3$\\
\hglue6mm${}+140501506020800\,\theta^4- (75140611-5687500861 \,\theta -5110332322 \,\theta^2$\\
\hglue6mm${}+8131585600 \,\theta^3)(114287524-161786452\,\theta+57427729 \,\theta^2)^{1/2}]$\\
\hglue6mm${}/[21208998746 (103-80 \,\theta) (133-10326 \,\theta+8960 \,\theta^2)]$ \\
24. $-(1483508960-2870126013 \,\theta+1320321327\,\theta^2)/[10604499373 (11-9 \,\theta)]$\\
25. $-[5322710051893 -377702268115357 \,\theta+840314998731738 \,\theta^2-634018822538972 \,\theta^3$\\
\hglue6mm${}+160148994678440\,\theta^4-(85310017-5767884733 \,\theta-5222494708 \,\theta^2$\\
\hglue6mm${}+8313962452 \,\theta^3)(109743865-151394620\,\theta+51579556 \,\theta^2)^{1/2}]$\\
\hglue6mm${}/[21208998746 (109-86 \,\theta) (151-10500 \,\theta+9116 \,\theta^2)]$ \\
26. $-16 (8250580261-15923807475 \,\theta+7305460616 \,\theta^2)/[10604499373 (979-800 \,\theta)]$ \\
27. $-48 (2749794730-5299558967 \,\theta+2428916504 \,\theta^2)/[10604499373 (979-800 \,\theta)]$ \\
28. $-(1483074080-2860466813 \,\theta+1313113647 \,\theta^2)/[10604499373 (11-9 \,\theta)]$ \\
29. $-(1482856640-2856535613 \,\theta+1310189007 \,\theta^2)/[10604499373 (11-9 \,\theta)]$ \\
30. $-(1482613280-2852155133 \,\theta+1306841007 \,\theta^2)/[10604499373 (11-9 \,\theta)]$ \\
31. $-[32 (348822690425 - 777921507937 \,\theta + 479648251884 \,\theta^2 -   64270064640 \,\theta^3) $\\
\hglue6mm${}+ 8 \sqrt{649} (50635200 + 912502649 \,\theta + 3285164672 \,\theta^2 - 3652485120 \,\theta^3)]$\\
\hglue6mm${}/[53022496865 (16577-18304 \,\theta+3840 \,\theta^2)]$ \\
32. $-[16 (697379854985 - 1551021662842 \,\theta + 954693068760 \,\theta^2 - 
127795740672 \,\theta^3) $\\
\hglue6mm${}+8 \sqrt{649} (34841840+ 1031176697 \,\theta + 3154719360 \,\theta^2 - 3611025408 \,\theta^3)]$\\
\hglue6mm${}/[53022496865 (16577-18304 \,\theta+3840\,\theta^2)]$ \\
33. $-[128 (87139291140 - 192958144129 \,\theta + 118491730235 \,\theta^2 - 
15874634304 \,\theta^3)$\\
\hglue6mm${} + 8 \sqrt{649} (19048480 + 933228537 \,\theta + 3211879040 \,\theta^2 - 3577602048 \,\theta^3)]$\\
\hglue6mm${}/[53022496865 (16577-18304 \,\theta+3840\,\theta^2)]$ \\
34. $-\{20309243048887 -1440900935249679 \,\theta+3228035909983198 \,\theta^2$\\
\hglue6mm${}-2454540613665204 \,\theta^3 +626572756041016 \,\theta^4- (85310017-5784669213 \,\theta$\\
\hglue6mm${}-5045929108 \,\theta^2+8237965332 \,\theta^3) [11 (148467731-217532180\,\theta +79180940 \,\theta^2)]^{1/2}\} $\\
\hglue6mm${}/[21208998746 (415-338 \,\theta) (151-10500 \,\theta+9116 \,\theta^2)]$ \\
35. $-\{20309243048887 -1440793288505679 \,\theta+3226528516783198 \,\theta^2 $\\
\hglue6mm${}-2452281458433204 \,\theta^3+625678662153016\,\theta^4-(85310017-5787285213 \,\theta $\\
\hglue6mm${}-5013913108 \,\theta^2+8211853332 \,\theta^3)[11(148467731-217532180\,\theta+79180940 \,\theta^2)]^{1/2}\} $\\
\hglue6mm${}/[21208998746 (415-338 \,\theta) (151-10500 \,\theta+9116 \,\theta^2)]$ \\
36. $-[5329412012587 +152061563180829 \,\theta-354738015046902 \,\theta^2+264529548416428 \,\theta^3$\\
\hglue6mm${}-64526904997632 \,\theta^4 - (95479423+3012939849 \,\theta+2376347652 \,\theta^2 $\\
\hglue6mm${}-4151180264 \,\theta^3) (122583865-197073880\,\theta+79914256 \,\theta^2)^{1/2}] $\\
\hglue6mm${}/[21208998746 (94-77 \,\theta) (169+5048 \,\theta-4532 \,\theta^2)]$ \\
37. $-[32 (348373526061 - 772940260925 \,\theta + 474517411564 \,\theta^2 - 
63071502336 \,\theta^3)  $\\
\hglue6mm${}-8 \sqrt{649} (2796992- 1752524185 \,\theta - 2479235712 \,\theta^2 +3512721408 \,\theta^3)] $\\
\hglue6mm${}/[53022496865 (16577-18304 \,\theta+3840\,\theta^2)]$ \\
38. $-[23553006170785 -527851874536246 \,\theta+1112934522007627 \,\theta^2 $\\
\hglue6mm${}-821157098602170 \,\theta^3+204411986805960\,\theta^4- (370053385-7164995911 \,\theta $\\
\hglue6mm${}-6597893474 \,\theta^2+10227090552 \,\theta^3) (105051121-132870958\,\theta+40507681 \,\theta^2)^{1/2}] $\\
\hglue6mm${}/[42417997492 (56-43 \,\theta) (655-13778 \,\theta+11616 \,\theta^2)]$ \\
39. $-[23553006170785 -527756742748726 \,\theta+1111714661263627 \,\theta^2 $\\
\hglue6mm${}-819355115214330 \,\theta^3+203708029342920\,\theta^4-(370053385-7172188231 \,\theta $\\
\hglue6mm${}-6529352354 \,\theta^2+10173238392 \,\theta^3) (105051121-132870958\,\theta+40507681 \,\theta^2)^{1/2}] $\\
\hglue6mm${}/[42417997492 (56-43 \,\theta) (655-13778 \,\theta+11616 \,\theta^2)]$ \\
40. $-[23553006170785 -527585747924726 \,\theta+1109586504147467 \,\theta^2 $\\
\hglue6mm${}-816239604678970 \,\theta^3+202500335146440\,\theta^4- (370053385-7185812231 \,\theta $\\
\hglue6mm${}-6397645794 \,\theta^2+10069009272 \,\theta^3) (105051121-132870958\,\theta+40507681 \,\theta^2)^{1/2}] $\\
\hglue6mm${}/[42417997492 (56-43 \,\theta) (655-13778 \,\theta+11616 \,\theta^2)]$ \\
41. $-\{20309243048887 -1439027882333519 \,\theta+3203634487118238 \,\theta^2 $\\
\hglue6mm${}-2417704373375924 \,\theta^3+611792243241016\,\theta^4- (85310017-5829004253 \,\theta $\\
\hglue6mm${}-4616563668 \,\theta^2+7897739412 \,\theta^3)[11(148467731-217532180\,\theta+79180940 \,\theta^2)]^{1/2}\} $\\
\hglue6mm${}/[21208998746 (415-338 \,\theta) (151-10500 \,\theta+9116 \,\theta^2)]$ \\
42. $-\{20309243048887 -1438587801175119 \,\theta+3198103084836958 \,\theta^2 $\\
\hglue6mm${}-2409347924803764 \,\theta^3+608426760201016\,\theta^4- (85310017-5839349853 \,\theta $\\
\hglue6mm${}-4525632148 \,\theta^2+7826666772 \,\theta^3)[11(148467731-217532180\,\theta+79180940 \,\theta^2)]^{1/2}\} $\\
\hglue6mm${}/[21208998746 (415-338 \,\theta) (151-10500 \,\theta+9116 \,\theta^2)]$ \\
43. $-[68173002801085 -361842641497321 \,\theta+593242021943304 \,\theta^2 $\\
\hglue6mm${}-386877538159088 \,\theta^3+87113744858408\,\theta^4-(370053385-798445141 \,\theta $\\
\hglue6mm${}-1479064094 \,\theta^2+1473437652 \,\theta^3) (1470229561-2720513020\,\theta+1260212260 \,\theta^2)^{1/2}] $\\
\hglue6mm${}/[21208998746 (307-284 \,\theta) (655-2548 \,\theta+1756 \,\theta^2)]$ \\
44. $-[68173002801085 -362031155003721 \,\theta+595663749431624 \,\theta^2 $\\
\hglue6mm${}-390766555977328 \,\theta^3+88792272698408\,\theta^4- (370053385-790794741 \,\theta $\\
\hglue6mm${}-1548256414 \,\theta^2+1527767892 \,\theta^3)\ (1470229561-2720513020\,\theta+1260212260 \,\theta^2)^{1/2}] $\\
\hglue6mm${}/[21208998746 (307-284 \,\theta) (655-2548 \,\theta+1756 \,\theta^2)]$ \\
45. $-\{5322710051893 -297802681147878 \,\theta+658095314726559 \,\theta^2-494848115914946 \,\theta^3$\\
\hglue6mm${}+124921046322012\,\theta^4- (85310017-4578641331 \,\theta-3395743050 \,\theta^2 $\\
\hglue6mm${}+5999348908 \,\theta^3)[5(21948773-32055266 \,\theta+11608013 \,\theta^2)]^{1/2}\} $\\
\hglue6mm${}/[21208998746 (109-89 \,\theta) (151-8254 \,\theta+7144 \,\theta^2)]$ \\
46. $-[20000227674700 +167647279856975 \,\theta-442588081074173 \,\theta^2 $\\
\hglue6mm${}+346607935245942 \,\theta^3-88242237094272\,\theta^4- 2 (370053385 + 4365509747 \,\theta  $\\
\hglue6mm${}+ 1974347548 \,\theta^2 - 5032341864 \,\theta^3) (25153300-40965940 \,\theta+16732321 \,\theta^2)^{1/2}] $\\
\hglue6mm${}/[10604499373 (185-157 \,\theta) (655+6436 \,\theta-6132 \,\theta^2)]$ \\
47. $-(63788826426-2377311861310 \,\theta+2151107563969 \,\theta^2-953147513920 \,\theta^3)$\\
\hglue6mm${}/[21208998746 (253+28\,\theta) (1-35 \,\theta)]$ \\
48. $-\{3194496825-9925798980 \,\theta+8807429906 \,\theta^2-1561359016 \,\theta^3  $\\
\hglue6mm${}+ 80 [114\,\theta (1 - \theta) (340880 - 429027 \,\theta - 1074404 \,\theta^2) (2824835 + 4392257 \,\theta  $\\
\hglue6mm${}- 1954788 \,\theta^2)]^{1/2}\}/[10604499373 (5-4 \,\theta)^2]$\\
49. $-[120603536706835 -782453418819761 \,\theta+1302773788384418 \,\theta^2 $\\
\hglue6mm${}-716442562381600 \,\theta^3+118117725349440\,\theta^4-(370053385-1392215891 \,\theta $\\
\hglue6mm${}-2507482208 \,\theta^2+1749863488 \,\theta^3) (208536601+2854478852\,\theta-1580996444 \,\theta^2)^{1/2}] $\\
\hglue6mm${}/[21208998746 (637-356 \,\theta) (655-3238 \,\theta+1624 \,\theta^2)]$ \\
50. $-(565851624800-1094725361473 \,\theta+522472722365 \,\theta^2) $\\
\hglue6mm${}/[116649493103 (385-349 \,\theta)]$ \\
51. $-\{16275482080-32107361753 \,\theta+14336234534 \,\theta^2+389285192 \,\theta^3  $\\
\hglue6mm${}+ 96 [-95 \,\theta (1 - \theta) (103136- 2222227 \,\theta - 
   1082268 \,\theta^2) (6214637 + 4078519 \,\theta  $\\
\hglue6mm${}- 6682468 \,\theta^2)]^{1/2}\}/[10604499373 (11-4 \,\theta)^2]$ \\
52. $-\{141339081-3032859092 \,\theta+10751305280 \,\theta^2-5503495424 \,\theta^3 $\\
\hglue6mm${}+ 60 [-57 \,\theta (1 - \theta) (88112- 4797177 \,\theta - 3428792 \,\theta^2) (564967 + 105761 \,\theta  $\\
\hglue6mm${}- 4281416 \,\theta^2)]^{1/2}\}/[10604499373 (1-8 \,\theta)^2]$ \\
53. $-\{3194496825-25668523940 \,\theta+34159562304 \,\theta^2-13235416704 \,\theta^3 $\\
\hglue6mm${}+ 60 [665 \,\theta (1 - \theta) (342800 + 4906909 \,\theta + 2888148 \,\theta^2) (2824835 + 4597481 \,\theta  $\\
\hglue6mm${}- 5411244 \,\theta^2)]^{1/2}\}/[10604499373 (5-12 \,\theta)^2]$\\
54. $-(1798055025-55645273722 \,\theta+4892388025 \,\theta^2-464320928\,\theta^3) $\\
\hglue6mm${}/[10604499373 (15+13 \,\theta) (1-29 \,\theta)]$ \\
55. $-\{3194496825-25682923940 \,\theta+34101271104 \,\theta^2-13162725504 \,\theta^3 $\\
\hglue6mm${}+60 [665 \,\theta (1 - \theta) (340880 + 4964509 \,\theta + 2832468 \,\theta^2) (2824835 + 4626281 \,\theta  $\\
\hglue6mm${}- 5386284 \,\theta^2)]^{1/2}\}/[10604499373 (5-12 \,\theta)^2]$\\
56. $-\{16275482080-32107361753 \,\theta+14336234534 \,\theta^2+389285192 \,\theta^3 $\\
\hglue6mm${}+ 96 [-95 \,\theta (1 - \theta) (103136- 2222227 \,\theta - 1082268 \,\theta^2) (6214637 + 4078519 \,\theta  $\\
\hglue6mm${}- 6682468 \,\theta^2)]^{1/2}\}/[10604499373 (11-4 \,\theta)^2]$\\
57. $-[48738877667025 +345584169500420 \,\theta-596777706348187 \,\theta^2 $\\
\hglue6mm${}+195701232893070 \,\theta^3+28132085064168 \,\theta^4- (370053385+3758736917 \,\theta $\\
\hglue6mm${}+263034346 \,\theta^2 -2465827176 \,\theta^3)(284799825-251288370\,\theta+25348129 \,\theta^2)^{1/2}] $\\
\hglue6mm${}/[169671989968 (30-23 \,\theta) (655+5746 \,\theta+312 \,\theta^2)]$ \\
58. $-(565466484800-1091244233473 \,\theta+519708164765 \,\theta^2) $\\
\hglue6mm${}/[116649493103 (385-349 \,\theta)]$ \\
59. $-\{5322710051893 -297705423490278 \,\theta+656923430851359 \,\theta^2 $\\
\hglue6mm${}-493079836081346 \,\theta^3+124207212277212\,\theta^4- (85310017-4587310131 \,\theta $\\
\hglue6mm${}-3324426570 \,\theta^2+5944066348 \,\theta^3)[5(21948773-32055266\,\theta+11608013 \,\theta^2)]^{1/2}\} $\\
\hglue6mm${}/[21208998746 (109-89 \,\theta) (151-8254 \,\theta+7144 \,\theta^2)]$ \\
60. $-[44678744987240 -83114516505921 \,\theta+41077048614881 \,\theta^2 $\\
\hglue6mm${}+4331079222774 \,\theta^3-5718669106624\,\theta^4-(370053385+487409781 \,\theta $\\
\hglue6mm${}-597999906 \,\theta^2-130739312 \,\theta^3) (587196736-1020241072\,\theta+444774961 \,\theta^2)^{1/2}] $\\
\hglue6mm${}/[21208998746 (203-178 \,\theta) (655-302 \,\theta-216 \,\theta^2)]$ \\
61. $-\{5322710051893 -297626357346278 \,\theta+656086970692959 \,\theta^2 $\\
\hglue6mm${}-491838177874946 \,\theta^3+123709134101212 \,\theta^4-(85310017-4594432691 \,\theta $\\
\hglue6mm${}-3273415690 \,\theta^2+5905219628 \,\theta^3) [5(21948773-32055266\,\theta+11608013 \,\theta^2)]^{1/2}\} $\\
\hglue6mm${}/[21208998746 (109-89 \,\theta) (151-8254 \,\theta+7144 \,\theta^2)]$ \\
62. $-(1468612160-2701497693 \,\theta+1198549167 \,\theta^2)/[10604499373 (11-9 \,\theta)]$ \\
63. $-\{5322710051893 -297606638147558 \,\theta+655451434380639 \,\theta^2 $\\
\hglue6mm${}-490863884442626 \,\theta^3+123326584361692 \,\theta^4- (85310017-4594432691 \,\theta $\\
\hglue6mm${}-3316945930 \,\theta^2+5940762668 \,\theta^3)[5(21948773-32055266 \,\theta+11608013 \,\theta^2)]^{1/2}\} $\\
\hglue6mm${}/[21208998746 (109-89 \,\theta) (151-8254 \,\theta+7144 \,\theta^2)]$ \\
64. $-[16468011103544 +82114052344397 \,\theta-236988968770905 \,\theta^2 $\\
\hglue6mm${}+188294230451990 \,\theta^3-47827705938112 \,\theta^4 -(74010677+612854261 \,\theta $\\
\hglue6mm${}+242478922 \,\theta^2-686992448 \,\theta^3) (1781786944-2909303536 \,\theta+1192851481 \,\theta^2)^{1/2}] $\\
\hglue6mm${}/[21208998746 (379-320 \,\theta) (131+838 \,\theta-832 \,\theta^2)]$ \\
65. $-[16468011103544 +82054568870477 \,\theta-236900431643993 \,\theta^2 $\\
\hglue6mm${}+188240196751126 \,\theta^3-47809397145792 \,\theta^4- (74010677+612854261 \,\theta $\\
\hglue6mm${}+232388426 \,\theta^2-678472768 \,\theta^3)(1781786944-2909303536 \,\theta+1192851481 \,\theta^2)^{1/2}] $\\
\hglue6mm${}/[21208998746 (379-320 \,\theta) (131+838 \,\theta-832 \,\theta^2)]$ \\
66. $-[563792175786385-6159103136740412 \,\theta+14053311335243944 \,\theta^2 $\\
\hglue6mm${}-14177675117512308 \,\theta^3+7181287047464888 \,\theta^4- 1527132941140224 \,\theta^5 $\\
\hglue6mm${}-(221620279-1747864990 \,\theta-2155536852 \,\theta^2+2934396812 \,\theta^3) (4342414609 $\\
\hglue6mm${}-75288376852 \,\theta+382769338276 \,\theta^2-478895079168 \,\theta^3+174977224704 \,\theta^4)^{1/2}] $\\
\hglue6mm${}/[21208998746 (4521-4192 \,\theta+320 \,\theta^2) (481-4392 \,\theta+3500 \,\theta^2)]$ \\
67. $-[9725507907518-40967030149881 \,\theta+38590577420808 \,\theta^2-33500923350152 \,\theta^3 $\\
\hglue6mm${}+17855101713272 \,\theta^4 - (157222316-1736012551 \,\theta-598081184 \,\theta^2 $\\
\hglue6mm${} +1403775004 \,\theta^3)(3171961-38476724 \,\theta+150204804\,\theta^2)^{1/2}] $\\
\hglue6mm${}/[21208998746 (137-4 \,\theta) (21-16 \,\theta)(13-35 \,\theta)]$ \\
68. $-[419244579992-16896227541388 \,\theta+71414079216111 \,\theta^2-70075472766898 \,\theta^3 $\\
\hglue6mm${}+20550223116464 \,\theta^4+(7405088+616604531 \,\theta-912386282 \,\theta^2 $\\
\hglue6mm${}+1384662256 \,\theta^3)(300304-9930856 \,\theta+76356121 \,\theta^2)^{1/2}]$\\
\hglue6mm${}/[84835994984 (1-35 \,\theta) (411-1610 \,\theta +256 \,\theta^2)]$ \\
69. $-\{1242458193-1686200619 \,\theta+294156074 \,\theta^2+112537600 \,\theta^3 $\\
\hglue6mm${}+ 80 [-7 \,\theta (1 - \theta) (1047291 - 819531 \,\theta - 160768 \,\theta^2) (257904 - 1987157 \,\theta  $\\
\hglue6mm${}- 1758400 \,\theta^2)]^{1/2}\}/95440494357 $ \\
70. $-[9724336721918-40984421740921 \,\theta+39026873707208 \,\theta^2-34254130456072 \,\theta^3 $\\
\hglue6mm${}+18190575309432 \,\theta^4-(156564716-1695187591 \,\theta-665898464 \,\theta^2 $\\
\hglue6mm${}+1431424924 \,\theta^3)(3171961-38476724 \,\theta+150204804\,\theta^2)^{1/2}]$\\
\hglue6mm${}/[21208998746 (137-4 \,\theta) (21-16 \,\theta) (13-35 \,\theta)]$ \\
71. $-[4086897252113 -185126390646691 \,\theta+171125184966904 \,\theta^2 $\\
\hglue6mm${}+122491341123472 \,\theta^3-112709050922328 \,\theta^4-(95479423-4182725339 \,\theta$\\
\hglue6mm${}-2610627708 \,\theta^2 +4295381964 \,\theta^3) (927192865-1849866100 \,\theta+922692004 \,\theta^2)^{1/2}] $\\
\hglue6mm${}/[21208998746 (26-25 \,\theta) (169-7556 \,\theta+2592 \,\theta^2)]$ \\
72. $-[563661307507441-6540403995935380 \,\theta+11741670233184172 \,\theta^2$\\
\hglue6mm${}-8611941499123352 \,\theta^3+3593299382911616 \,\theta^4-875197441225952 \,\theta^5-(219634327$\\
\hglue6mm${}-1470282400 \,\theta-1990210056 \,\theta^2+2219021624 \,\theta^3) (4342414609-76590765160 \,\theta$\\
\hglue6mm${}+378797212888 \,\theta^2-429088870560 \,\theta^3+139230322704 \,\theta^4)^{1/2}]$\\
\hglue6mm${}/[21208998746 (4521-3706 \,\theta+128 \,\theta^2) (481-4758 \,\theta+1400 \,\theta^2)]$ \\
73. $-[563093633847121-6533620178412692 \,\theta+11731921432283052 \,\theta^2$\\
\hglue6mm${}-8662080296298136 \,\theta^3+3659622858278016 \,\theta^4-  894593269785824 \,\theta^5-(211019767$\\
\hglue6mm${}-1300824096 \,\theta-2152931976 \,\theta^2+2273100088 \,\theta^3) (4342414609-76590765160 \,\theta$\\
\hglue6mm${}+378797212888 \,\theta^2-429088870560 \,\theta^3+139230322704 \,\theta^4)^{1/2}]$\\
\hglue6mm${}/[21208998746 (4521-3706 \,\theta+128 \,\theta^2) (481-4758 \,\theta+1400 \,\theta^2)]$ \\
74. $-(1982685137441-2426532275238 \,\theta+831447933002 \,\theta^2)$\\
\hglue6mm${}/[116649493103 (1371-352 \,\theta)]$ \\
75. $-(1980688961441-2411936715238 \,\theta+818848549002 \,\theta^2)$\\
\hglue6mm${}/[116649493103 (1371-352 \,\theta)]$ \\
76. $-[1067203846700+37427990066498 \,\theta-44096101156681 \,\theta^2-1266084963690 \,\theta^3 $\\
\hglue6mm${}+4733022672840 \,\theta^4-(106705720-592538297 \,\theta-1277647946 \,\theta^2$\\
\hglue6mm${}+249823560 \,\theta^3)(75076+288796 \,\theta+36666529 \,\theta^2)^{1/2}]$\\
\hglue6mm${}/[42417997492 (1+35 \,\theta) (2055-226 \,\theta-128 \,\theta^2)]$ \\
77. $-32 (21941365264-196646459517 \,\theta+169425586011 \,\theta^2)$\\
\hglue6mm${}/[116649493103 (481-3850 \,\theta)]$ \\
78. $-16 (43871593935-393050171750 \,\theta+337780910546 \,\theta^2)$\\
\hglue6mm${}/[116649493103 (481-3850 \,\theta)]$ \\
79. $-32 (21941365264-196541127732 \,\theta+168784274401 \,\theta^2)$\\
\hglue6mm${}/[116649493103 (481-3850 \,\theta)]$ \\
80. $-128 (1722009436-140774817947 \,\theta+119228476035 \,\theta^2)$\\
\hglue6mm${}/[116649493103 (151-12504 \,\theta)]$ \\
81. $-32 (52048020304-61553110724 \,\theta+57822685369 \,\theta^2)/[116649493103 (1141-18 \,\theta)]$ \\
82. $-16 (104069623035-123184229278 \,\theta+117017230786 \,\theta^2)$\\
\hglue6mm${}/[116649493103 (1141-18 \,\theta)]$ \\
83. $-3 (72252932336+2580286294207 \,\theta-2318275081608 \,\theta^2)$\\
\hglue6mm${}/[116649493103 (169+5446 \,\theta)]$ \\
84. $-16 (13772579385-1119107660632 \,\theta+911813453784 \,\theta^2)$\\
\hglue6mm${}/[116649493103 (151-12504 \,\theta)]$ \\
85. $-16 (13772579385-1110400116280 \,\theta+863758684632 \,\theta^2)$\\
\hglue6mm${}/[116649493103 (151-12504 \,\theta)]$ \\
86. $-32 (6884727824-555062971733 \,\theta+431237474484 \,\theta^2)$\\
\hglue6mm${}/[116649493103 (151-12504 \,\theta)]$ \\
87. $-16 (13755905663-1109265173216 \,\theta+858582102168 \,\theta^2)$\\
\hglue6mm${}/[116649493103 (151-12504 \,\theta)]$ \\
88. $-96 (2290392613-184730166982 \,\theta+142448092436 \,\theta^2)$\\
\hglue6mm${}/[116649493103 (151-12504 \,\theta)]$ \\
89. $-(1791574243056+71575277151 \,\theta-4401093023830 \,\theta^2)$\\
\hglue6mm${}/[116649493103 (1443-320 \,\theta)]$ \\
90. $-32 (71768589237-97002675970 \,\theta+74256652793 \,\theta^2)/[116649493103 (1591-468 \,\theta)]$ \\
91. $-96 (2270491719-181859381018 \,\theta+129586365364 \,\theta^2)$\\
\hglue6mm${}/[116649493103 (151-12504 \,\theta)]$ \\
92. $-16 (13590037297-1089010578726 \,\theta+768914366856 \,\theta^2)$\\
\hglue6mm${}/[116649493103 (151-12504 \,\theta)]$ \\
93. $-16 (13470631933-1081182551966 \,\theta+737784057672 \,\theta^2)$\\
\hglue6mm${}/[116649493103 (151-12504 \,\theta)]$ \\
94. $-32 (6728540974-540145459715 \,\theta+367120604772 \,\theta^2)$\\
\hglue6mm${}/[116649493103 (151-12504 \,\theta)]$ \\
95. $-32 (470039702-321957499 \,\theta)/1283144424133 $\\
96. $-32 (464056790-298455787 \,\theta)/1283144424133 $\\
97. $-16 (926326887-589910792 \,\theta)/1283144424133 $\\
98. $-32 (462522772-292685813 \,\theta)/1283144424133 $\\
99. $-32 (456494932-271475013 \,\theta)/1283144424133 $\\
100. $-16 (910783169-535383252 \,\theta)/1283144424133 $\\
101. $-16 (907840909-525959036 \,\theta)/1283144424133 $\\
102. $-32 (453299997-261013145 \,\theta)/1283144424133 $\\
103. $-16 (903908163-513554378 \,\theta)/1283144424133 $\\
104. $-64 (225656705-127395301 \,\theta)/1283144424133 $\\
105. $-64 (224036216-122516989 \,\theta)/1283144424133 $\\
106. $-16 (882685709-450538356 \,\theta)/1283144424133 $\\
107. $-16 (879112323-440449324 \,\theta)/1283144424133 $\\
108. $-32 (439188379-219216208 \,\theta)/1283144424133 $\\
109. $-32 (429980635-194825504 \,\theta)/1283144424133 $\\
110. $-32 (426739657-186241292\,\theta)/1283144424133 $\\
111. $-112 (121541069-52211686 \,\theta)/1283144424133 $\\
112. $-16 (837756191-331594154 \,\theta)/1283144424133 $\\
113. $-16 (818601887-282073930 \,\theta)/1283144424133 $\\
114. $-128 (101988757-34404919 \,\theta)/1283144424133 $\\
115. $-16 (814123363-270760828 \,\theta)/1283144424133 $\\
116. $-1632 (7964085-2610809 \,\theta)/1283144424133 $\\
117. $-80 (162219151-52654982 \,\theta)/1283144424133$\\
118. $-64 (201959483-63927034 \,\theta)/1283144424133 $\\
119. $-32 (186882616486-271223952336 \,\theta+156733883409 \,\theta^2)$\\
\hglue6mm${}/[116649493103 (4643-3520 \,\theta)]$ \\
120. $-32 (185476855376-270104867480 \,\theta+161643858577 \,\theta^2)$\\
\hglue6mm${}/[116649493103 (4643-3520 \,\theta)]$ \\
121. $-16 (369108410761-535761222278 \,\theta+322776717474 \,\theta^2)$\\
\hglue6mm${}/[116649493103 (4643-3520 \,\theta)]$ \\
122. $-16 (12004171877-810611806806 \,\theta+30230619912 \,\theta^2)$\\
\hglue6mm${}/[116649493103 (151-12504 \,\theta)]$ \\
123. $-32 (5977292569-403731372504 \,\theta+11736304812 \,\theta^2)$\\
\hglue6mm${}/[116649493103 (151-12504 \,\theta)]$ \\
124. $-16 (11904626033-804290238156 \,\theta+16691526648 \,\theta^2)$\\
\hglue6mm${}/[116649493103 (151-12504 \,\theta)]$ \\
125. $-32 (5927519647-400580253840 \,\theta+5282878188 \,\theta^2)$\\
\hglue6mm${}/[116649493103 (151-12504 \,\theta)]$ \\
126. $-\{1006830711+5740109913 \,\theta+6863843632 \,\theta^2-10141233600 \,\theta^3$\\
\hglue6mm${}+ 384 [10 \,\theta (1 - \theta) (508455 + 490401 \,\theta - 3262312 \,\theta^2) (1910976 - 1870549 \,\theta $\\
\hglue6mm${}- 9828920 \,\theta^2)]^{1/2}\}/[10604499373 (3 + 8 \,\theta)^2]$ \\
127. $-[11508761994+364954528953 \,\theta+2056469568736 \,\theta^2+3341872958672 \,\theta^3$\\
\hglue6mm${}-3883623730560 \,\theta^4-7 (9086784 + 296462939 \,\theta - 1056711628 \,\theta^2$\\
\hglue6mm${} + 1058164640 \,\theta^3)(18769+110148 \,\theta-118716\,\theta^2)^{1/2}]$\\
\hglue6mm${}/[21208998746 (1+7 \,\theta) (411+436 \,\theta+1568 \,\theta^2)]$ \\
128. $-\{3 (31584643268923 - 2222801600784000 \,\theta - 
   8399101487371008 \,\theta^2 $\\
\hglue6mm${}+ 21271409223426048 \,\theta^3 - 10663533870379008 \,\theta^4 + 13811163070464 \,\theta^5)$\\
\hglue6mm${}-(2956617537-14919507136 \,\theta+14844772672 \,\theta^2+43327488 \,\theta^3)[(20687 $\\
\hglue6mm${}- 1065024 \,\theta + 1065024 \,\theta^2) (20687 - 1556544 \,\theta + 1556544 \,\theta^2)]^{1/2}\}$\\
\hglue6mm${}/[21208998746 (151-12928 \,\theta+8256 \,\theta^2) (4521+3584 \,\theta-8256 \,\theta^2)]$ \\
129. $-[4975694427636+2036609450871 \,\theta-65883906612104 \,\theta^2+89984612598976 \,\theta^3$\\
\hglue6mm${}-15104796295680 \,\theta^4-3 (378748362 - 1425747809 \,\theta + 680992304 \,\theta^2 $\\
\hglue6mm${}+ 1077007680 \,\theta^3)(9928801-53390544 \,\theta+87503424 \,\theta^2)^{1/2}]$\\
\hglue6mm${}/[21208998746 (69-136 \,\theta) (411+118 \,\theta-1328 \,\theta^2)]$ \\
130. $-\{1002293271+5756543673 \,\theta+6876615472 \,\theta^2-10165901760 \,\theta^3$\\
\hglue6mm${}+384 [30 \,\theta (1 - \theta) (515031 + 492289 \,\theta - 3283560 \,\theta^2) (633312 - 612583 \,\theta$\\
\hglue6mm${} - 3283560 \,\theta^2)]^{1/2}\}/[10604499373 (3+8 \,\theta)^2]$ \\
131. $-\{94430296089921-6646323952883456 \,\theta-25262343281817856 \,\theta^2 $\\
\hglue6mm${}+63817334469402624 \,\theta^3-31908667234701312 \,\theta^4-3 (980324411 - 4975425984 \,\theta$\\
\hglue6mm${} + 4975425984 \,\theta^2) [(20687 - 1065024 \,\theta + 1065024 \,\theta^2) (20687 - 1556544 \,\theta $\\
\hglue6mm${}+ 1556544 \,\theta^2)]^{1/2}\}/[21208998746 (151-12928 \,\theta+8256 \,\theta^2) (4521+3584 \,\theta-8256 \,\theta^2)]$
\end{small}

\newpage
\section*{Appendix C.  Nash equilibrium}

\begin{small}
We list the 102 points of discontinuity in $(0,1/2)$ of the Nash equilibrium.  All are roots of polynomials of degree 4 or less. \medskip

\noindent$\theta_1\approx0.0172597$ (root of $1196071 - 70128084\,\theta + 48076640\,\theta^2$)\\
$\theta_2\approx0.0182052$ (root of $13439 - 747528\,\theta + 512560\,\theta^2$)\\
$\theta_3\approx0.0187178$ (root of $1196071 - 64730946\,\theta + 44388440\,\theta^2$)\\
$\theta_4\approx0.0192602$ (root of $1196071 - 62931900\,\theta + 43159040\,\theta^2$)\\
$\theta_5\approx0.0194839$ (root of $15041 - 782610\,\theta + 546040\,\theta^2$)\\
$\theta_6\approx0.0232232$ (root of $1196071 - 52337518\,\theta + 35919240\,\theta^2$)\\
$\theta_7\approx0.0234507$ (root of $1196071 - 51916370\,\theta + 38924736\,\theta^2$)\\
$\theta_8\approx0.0243083$ (root of $1196071 - 50117324\,\theta + 37567104\,\theta^2$)\\
$\theta_9\approx0.0243618$ (root of $1053493 - 44035598\,\theta + 32510784\,\theta^2$)\\
$\theta_{10}\approx0.0250848$ (root of $1338649 - 54400004\,\theta + 41265792\,\theta^2$)\\
$\theta_{11}\approx0.0254193$ (root of $1053493 - 42236552\,\theta + 31153152\,\theta^2$)\\
$\theta_{12}\approx0.0273040$ (root of $1196071 - 44720186\,\theta + 33494208\,\theta^2$)\\
$\theta_{13}\approx0.0289972$ (root of $1338649 - 47203820\,\theta + 35835264\,\theta^2$)\\
$\theta_{14}\approx0.0311422$ (root of $1196071 - 39323048\,\theta + 29421312\,\theta^2$)\\
$\theta_{15}\approx0.0341664$ (root of $1196071 - 35924850\,\theta + 26856896\,\theta^2$)\\
$\theta_{16}\approx0.0360182$ (root of $1196071 - 34125804\,\theta + 25499264\,\theta^2$)\\
$\theta_{17}\approx0.0362616$ (root of $1053493-29843124\,\theta+21800576\,\theta^2$)\\
$\theta_{18}\approx0.0376433$ (root of $1338649-36609438\,\theta+27840320\,\theta^2$)\\
$\theta_{19}\approx0.0403979$ (root of $13439 - 343008\,\theta + 256000\,\theta^2$)\\
$\theta_{20}\approx0.0425911$ (root of $7097 - 172107\,\theta + 128568\,\theta^2$)\\
$\theta_{21}\approx0.0431219$ (root of $1697705233 - 179649339338\,\theta + 3456239837408\,\theta^2$\\
\hglue1cm${} - 4784841897088\,\theta^3 + 1712192778240\,\theta^4$)\\
$\theta_{22}\approx0.0442392$ (root of $59351003 - 6091261480\,\theta + 115203984016\,\theta^2- 180482084480\,\theta^3 $\\
\hglue1cm${} + 73466830848\,\theta^4$)\\
$\theta_{23}\approx0.0471300$ (root of $128073217 - 12321650132\,\theta + 220062541748\,\theta^2$\\
\hglue1cm${} - 352366745632\,\theta^3 + 146605449216\,\theta^4$)\\
$\theta_{24}\approx0.0573572$ (root of $33780227 - 2723282412\,\theta + 45247029888\,\theta^2- 145206995712\,\theta^3$\\
\hglue1cm${}  + 89048375296\,\theta^4$)\\
$\theta_{25}\approx0.0591403$ (root of $3123737 - 259968848\,\theta + 3687520832\,\theta^2- 3144634368\,\theta^3$\\
\hglue1cm${}  + 325632000\,\theta^4$)\\
$\theta_{26}\approx0.0652644$ (root of $31536199 - 508575192\,\theta + 388700160\,\theta^2$)\\
$\theta_{27}\approx0.0806166$ (root of $31430175 - 414868312\,\theta + 310060800\,\theta^2$)\\
$\theta_{28}\approx0.0812522$ (root of $7245735 - 95040097\,\theta + 72172800\,\theta^2$)\\
$\theta_{29}\approx0.0834260$ (root of $9595069 - 40934083\,\theta - 1023367628\,\theta^2 + 1623091968\,\theta^3$)\\
$\theta_{30}\approx0.0840033$ (root of $28113633 + 9276709160\,\theta - 122239114352\,\theta^2+ 93270978560\,\theta^3 $\\
\hglue1cm${} - 1804308480\,\theta^4$)\\
$\theta_{31}\approx0.0844782$ (root of $21810537 - 5495456260\,\theta + 66071964772\,\theta^2- 48365877760\,\theta^3 $\\
\hglue1cm${} + 1329454080\,\theta^4$)\\
$\theta_{32}\approx0.0847479$ (root of $299335 - 3770281\,\theta + 2810880\,\theta^2$)\\
$\theta_{33}\approx0.0866700$ (root of $6382005 - 78597901\,\theta + 57254400\,\theta^2$)\\
$\theta_{34}\approx0.0986936$ (root of $10476725 - 114510661\,\theta + 84672000\,\theta^2$)\\
$\theta_{35}\approx0.0987732$ (root of $959548 - 29449633\,\theta + 209415867\,\theta^2 - 97344000\,\theta^3$)\\
$\theta_{36}\approx0.0991012$ (root of $176329275 + 7376528780\,\theta - 99541141972\,\theta^2+ 72124260864\,\theta^3 $\\
\hglue1cm${} + 511349760\,\theta^4$)\\
$\theta_{37}\approx0.0992228$ (root of $342436 - 3712083\,\theta + 2629436\,\theta^2$)\\
$\theta_{38}\approx0.100952$ (root of $2767375 - 29557970\,\theta + 21249219\,\theta^2$)\\
$\theta_{39}\approx0.103087$ (root of $7609292 - 79982713\,\theta + 59838606\,\theta^2$)\\
$\theta_{40}\approx0.105963$ (root of $39757085 - 406643392\,\theta + 296755596\,\theta^2$)\\
$\theta_{41}\approx0.111672$ (root of $40843060 - 399958783\,\theta + 306406206\,\theta^2$)\\
$\theta_{42}\approx0.111727$ (root of $140032092 - 1365758465\,\theta + 1006186750\,\theta^2$)\\
$\theta_{43}\approx0.123634$ (root of $4077 - 36744\,\theta + 30475\,\theta^2$)\\
$\theta_{44}\approx0.135877$ (root of $1196071 - 9917626\,\theta + 8206024\,\theta^2$)\\
$\theta_{45}\approx0.135889$ (root of $5188255 - 43001952\,\theta + 35482916\,\theta^2$)\\
$\theta_{46}\approx0.150032$ (root of $1037651 - 7880772\,\theta + 6429244\,\theta^2$)\\
$\theta_{47}\approx0.151654$ (root of $128427 - 966182\,\theta + 786928\,\theta^2$)\\
$\theta_{48}\approx0.151865$ (root of $21806448509 - 253203995032\,\theta + 865393272728\,\theta^2$\\
\hglue1cm${} - 1001495027424\,\theta^3 + 367605859408\,\theta^4$)\\
$\theta_{49}\approx0.152021$ (root of $67821 - 508597\,\theta + 410910\,\theta^2$)\\
$\theta_{50}\approx0.152087$ (root of $99286229 - 3347506740\,\theta + 22154149028\,\theta^2 -  30943670400\,\theta^3 $\\
\hglue1cm${} + 11672260544\,\theta^4$)\\
$\theta_{51}\approx0.152117$ (root of $44952851 - 1492318649\,\theta + 9470092670\,\theta^2 - 11072285484\,\theta^3 $\\
\hglue1cm${} + 3536867064\,\theta^4$)\\
$\theta_{52}\approx0.153828$ (root of $18074088461 - 489486828168\,\theta + 2740068166648\,\theta^2 $\\
\hglue1cm${} - 2125555458848\,\theta^3 + 216860715472\,\theta^4$)\\
$\theta_{53}\approx0.154737$ (root of $1532584065 - 29711728576\,\theta + 153512273652\,\theta^2$\\
\hglue1cm${} - 173198601776\,\theta^3 + 54045846912\,\theta^4$)\\
$\theta_{54}\approx0.154894$ (root of $7693595 - 56442557\,\theta + 43723254\,\theta^2$)\\
$\theta_{55}\approx0.156773$ (root of $190178742 - 3285318023\,\theta + 15032889412\,\theta^2- 11720012508\,\theta^3 $\\
\hglue1cm${} + 917652480\,\theta^4$)\\
$\theta_{56}\approx0.162831$ (root of $70318 - 485511\,\theta + 329568\,\theta^2$)\\
$\theta_{57}\approx0.165259$ (root of $126177 + 1644824 \,\theta - 14573056 \,\theta^2$)\\
$\theta_{58}\approx0.166815$ ($=5621/33696$)\\
$\theta_{59}\approx0.167161$ ($=1127/6742$)\\
$\theta_{60}\approx0.168027$ (root of $9950447-22050316\,\theta-221208996\,\theta^2$)\\
$\theta_{61}\approx0.169559$ (root of $9950447-331023552\,\theta+1606163556\,\theta^2$)\\
$\theta_{62}\approx0.170532$ ($=5772/33847$)\\
$\theta_{63}\approx0.184790$ ($=6364/34439$)\\
$\theta_{64}\approx0.187345$ ($=151/806$)\\
$\theta_{65}\approx0.205097$ ($=169/824$)\\
$\theta_{66}\approx0.207326$ ($=19092/92087$)\\
$\theta_{67}\approx0.213355$ ($=131/614$)\\
$\theta_{68}\approx0.214613$ ($=655/3052$)\\
$\theta_{69}\approx0.218968$ ($=1443/6590$)\\
$\theta_{70}\approx0.236123$ ($=1591/6738$)\\
$\theta_{71}\approx0.236633$ ($=655/2768$)\\
$\theta_{72}\approx0.237102$ ($=1443/6086$)\\
$\theta_{73}\approx0.254573$ ($=9588/37663$)\\
$\theta_{74}\approx0.255213$ ($=1591/6234$)\\
$\theta_{75}\approx0.282286$ ($=1141/4042$)\\
$\theta_{76}\approx0.284187$ ($=5796/20395$)\\
$\theta_{77}\approx0.291630$ ($=655/2246$)\\
$\theta_{78}\approx0.312202$ ($=655/2098$)\\
$\theta_{79}\approx0.315534$ ($=65/206$)\\
$\theta_{80}\approx0.317736$ ($=2397/7544$)\\
$\theta_{81}\approx0.322499$ ($=1141/3538$)\\ 
$\theta_{82}\approx0.332182$ ($=481/1448$)\\
$\theta_{83}\approx0.340483$ ($=2397/7040$)\\
$\theta_{84}\approx0.354185$ ($=1591/4492$)\\
$\theta_{85}\approx0.364699$ ($=655/1796$)\\
$\theta_{86}\approx0.377510$ ($=44268/117263$)\\
$\theta_{87}\approx0.377551$ ($=37/98$)\\
$\theta_{88}\approx0.384504$ ($=799/2078$)\\
$\theta_{89}\approx0.384544$ ($=2901/7544$)\\
$\theta_{90}\approx0.386798$ ($=46044/119039$)\\
$\theta_{91}\approx0.393855$ ($=141/358$)\\
$\theta_{92}\approx0.398947$ ($=1591/3988$)\\
$\theta_{93}\approx0.400756$ ($=1591/3970$)\\
$\theta_{94}\approx0.409866$ ($=1105/2696$)\\
$\theta_{95}\approx0.430543$ ($=967/2246$)\\
$\theta_{96}\approx0.432877$ ($=55716/128711$)\\
$\theta_{97}\approx0.433673$ ($=85/196$)\\
$\theta_{98}\approx0.440212$ ($=2901/6590$)\\
$\theta_{99}\approx0.452435$ ($=799/1766$)\\
$\theta_{100}\approx0.454289$ ($=805/1772$)\\
$\theta_{101}\approx0.498116$ ($=793/1592$)\\
$\theta_{102}\approx0.498130$ ($=799/1604$)\\
\end{small}\bigskip

\begin{small}
Banker's Nash equilibrium strategy (excluding mixing probabilities).  Pairs followed by a colon are Player~1's and Player~2's third-card values (10 = stand, 11 = natural).  Entries are maximum Banker drawing totals (e.g., 5 means that Banker draws on 5 or less and stands on 6 or 7).  Plus signs indicate that Banker mixes on next higher total (e.g., 5+ means Banker draws on 5 or less, mixes on 6, and stands on 7).  Number ranges in parentheses refer to the 103 intervals in $(0,1/2]$, with interval $i$ being $(\theta_{i-1},\theta_i)$; here $\theta_0=0$ and $\theta_{103}=1/2$.\medskip

$(0,0)$: 3.  
$(0,1)$: 3.  
$(0,2)$: 4 (1--97), 3 (98--103).  
$(0,3)$: 4.  
$(0,4)$: 5 (1--11), 4 (12--103).  
$(0,5)$: 5 (1--83), 4 (84--103).  
$(0,6)$: 6 (1--12),  5 (13--89), 4 (90--103).  
$(0,7)$: 6 (1--34), 5+ (35--39), 5 (40--42), 4 (43--78), 3 (79--103).  
$(0,8)$: 2 (1--28), 2+ (29--30), 3 (31--39), 2+ (40--44), 3 (45--103).  
$(0,9)$: 3.  
$(0,10)$: 5.  
$(0,11)$: 3.

$(1,0)$: 3.  
$(1,1)$: 3.  
$(1,2)$: 4.  
$(1,3)$: 4.  
$(1,4)$: 5 (1--9), 4 (10--103).  
$(1,5)$: 5 (1--91), 4 (92--103).  
$(1,6)$: 6 (1--4), 5 (5--98), 4 (99--103).  
$(1,7)$: 6 (1--27), 5 (28--40), 4+ (41), 4 (42--103).  
$(1,8)$: 2 (1--23), 3 (24--103).  
$(1,9)$: 3.  
$(1,10)$: 5.  
$(1,11)$: 3.

$(2,0)$: 3.  
$(2,1)$: 3 (1--52), 4 (53--103).  
$(2,2)$: 4.  
$(2,3)$: 4.  
$(2,4)$: 5 (1--17), 4 (18--103).  
$(2,5)$: 5 (1--102), 4 (103).  
$(2,6)$: 6 (1--3), 5 (4--103).  
$(2,7)$: 6 (1--29), 5 (30--31), 5+ (32), 5 (33--46), 4 (47--48), 5 (49--67), 4 (68--103).  
$(2,8)$: 2 (1--14), 3 (15--103).  
$(2,9)$: 3.  
$(2,10)$: 5.  
$(2,11)$: 4.

$(3,0)$: 3 (1--87), 4 (88--103).  
$(3,1)$: 3 (1--24), 4 (25--103).  
$(3,2)$: 4.  
$(3,3)$: 4.  
$(3,4)$: 5 (1--33), 4 (34--103).  
$(3,5)$: 5.  
$(3,6)$: 6 (1--8), 5 (9--103).  
$(3,7)$: 6 (1--38), 5 (39--85), 4 (86--103).  
$(3,8)$: 2 (1--4), 3 (5--93), 4 (94--103).  
$(3,9)$: 3.  
$(3,10)$: 5.  
$(3,11)$: 4.

$(4,0)$: 3 (1--72), 4 (73--103).  
$(4,1)$: 3 (1--13), 4 (14--103).  
$(4,2)$: 4.  
$(4,3)$: 4.  
$(4,4)$: 5.  
$(4,5)$: 5.  
$(4,6)$: 6 (1--15), 5 (16--103).  
$(4,7)$: 6 (1--45), 5 (46--103).  
$(4,8)$: 2 (1), 3 (2--74), 4 (75--103).  
$(4,9)$: 3 (1--83), 4 (84--103).  
$(4,10)$: 5.  
$(4,11)$:  4 (1--20), 5 (21--43), 4 (44--49), 5 (50--103). 

$(5,0)$: 3 (1--69), 4 (70--103).  
$(5,1)$: 3 (1--10), 4 (11--103).  
$(5,2)$: 4.  
$(5,3)$: 4 (1--81), 5 (82--103).  
$(5,4)$: 5.  
$(5,5)$: 5.  
$(5,6)$: 6 (1--25), 5 (26--103).  
$(5,7)$: 6 (1--71), 5 (72--103).  
$(5,8)$: 2 (1--6), 2+ (7), 3 (8--70), 4 (71--92), 5 (93--103).  
$(5,9)$: 3 (1--80), 4 (81--103).  
$(5,10)$: 6 (1--7), 5+ (8--25), 5 (26--30), 5+ (31), 5 (32--36), 5+ (37--40), 5 (41--103).  
$(5,11)$: 5.

$(6,0)$: 3 (1--82), 4 (83--103).  
$(6,1)$: 3 (1--18), 4 (19--103).  
$(6,2)$: 4 (1--100), 5 (101--103).  
$(6,3)$: 4 (1--75), 5 (76--103).  
$(6,4)$: 5.  
$(6,5)$: 5.  
$(6,6)$: 6.  
$(6,7)$: 6.  
$(6,8)$: 2 (1--16), 3 (17--84), 5 (85--103).  
$(6,9)$: 3 (1--99), 4 (100--103).  
$(6,10)$: 6.  
$(6,11)$: 6.

$(7,0)$: 3.  
$(7,1)$: 3 (1--35), 4 (36--57), 3 (58--65), 4 (66--103).  
$(7,2)$: 4.  
$(7,3)$: 4.  
$(7,4)$: 5.  
$(7,5)$: 5.  
$(7,6)$: 6.  
$(7,7)$: 6.  
$(7,8)$: 2 (1--37), 3 (38--60), 2 (61--64), 3 (65--103).  
$(7,9)$: 3.  
$(7,10)$: 6.  
$(7,11)$: 6.

$(8,0)$: 3.  
$(8,1)$: 3.  
$(8,2)$: 4 (1--94), 3 (95--103).  
$(8,3)$: 4.  
$(8,4)$: 5 (1--50), 4 (51--103).  
$(8,5)$: 5.  
$(8,6)$: 6 (1--51), 5 (52--103).  
$(8,7)$: 6 (1--77), 3 (78--103).  
$(8,8)$: 2.  
$(8,9)$: 3.  
$(8,10)$: 6 (1--41), 5+ (42--55), 5 (56--103).  
$(8,11)$: 2.

$(9,0)$: 3.  
$(9,1)$: 3.  
$(9,2)$: 4 (1--79), 3 (80--103).  
$(9,3)$: 4 (1--101), 3 (102--103).  
$(9,4)$: 5 (1--21), 4 (22--103).  
$(9,5)$: 5 (1--88), 4 (89--103).  
$(9,6)$: 6 (1--22), 5 (23--95), 4 (96--103).  
$(9,7)$: 6 (1--53), 4 (54--68), 3 (69--103).  
$(9,8)$: 2 (1--61), 3 (62--103).  
$(9,9)$: 3.  
$(9,10)$: 5+ (1--6), 5 (7--103).  
$(9,11)$: 2 (1--26), 2+ (27--28), 2 (29--32), 2+ (33--34), 2 (35--54), 3 (55--103).

$(10,0)$: 3 (1--47), 3+ (48--62), 4 (63--96), 5 (97--103).  
$(10,1)$: 3 (1--5), 4 (6--86), 5 (87--103).  
$(10,2)$: 4 (1--76), 5 (77--103).  
$(10,3)$: 4 (1--59), 5 (60--103).  
$(10,4)$: 5.  
$(10,5)$: 5.  
$(10,6)$: 6 (1--19), 5+ (20--26), 5 (27--56), 5+ (57--58), 6 (59--103).  
$(10,7)$: 6.  
$(10,8)$: 2 (1--2), 3 (3--63), 4 (64--66), 5 (67--103).  
$(10,9)$: 3 (1--73), 4 (74--90), 5 (91--103).  
$(10,10)$: 5 (1--25), 5+ (26--36), 5 (37--55), 5+ (56), 5 (57--62), 5+ (63--103).  
$(10,11)$: 5 (1--62), 5+ (63--103).

$(11,0)$: 3.  
$(11,1)$: 3.  
$(11,2)$: 4.  
$(11,3)$: 4.  
$(11,4)$: 5.  
$(11,5)$: 5.  
$(11,6)$: 6.  
$(11,7)$: 6.  
$(11,8)$: 2.  
$(11,9)$: 3.  
$(11,10)$: 6 (1--58), 5+ (59--103).
\end{small}
\bigskip

The information sets on which Banker mixes in the Nash equilibrium.\medskip

\tabcolsep=1.5mm
\begin{center}
\begin{small}
\begin{tabular}{cc}
\hline
\noalign{\smallskip}
intervals & Banker mixes on \\
\noalign{\smallskip} \hline
\noalign{\smallskip}
1--6   & $(9,10,6)$ \\
7      & $(5,8,3)$ \\
8--19  & $(5,10,6)$ \\
20--25 & $(5,10,6)$, $(10,6,6)$ \\
26     & $(10,6,6)$, $(10,10,6)$ \\
27--28, 33--34 & $(9,11,3)$, $(10,10,6)$ \\
29--30 & $(0,8,3)$, $(10,10,6)$ \\
\noalign{\smallskip}
\hline
\end{tabular}
\end{small}
\end{center}
\tabcolsep=1.5mm
\begin{center}
\begin{small}
\begin{tabular}{cc}
\hline
\noalign{\smallskip}
intervals & Banker mixes on \\
\noalign{\smallskip} \hline
\noalign{\smallskip}
31     & $(5,10,6)$, $(10,10,6)$\\
32     & $(2,7,6)$, $(10,10,6)$ \\
35--36 & $(0,7,6)$, $(10,10,6)$ \\
37--39 & $(0,7,6)$, $(5,10,6)$ \\
40     & $(0,8,3)$, $(5,10,6)$ \\
41     & $(0,8,3)$, $(1,7,5)$ \\
42--44 & $(0,8,3)$, $(8,10,6)$ \\
45--47 & $(8,10,6)$ \\
48--55 & $(8,10,6)$, $(10,0,4)$ \\
56     & $(10,0,4)$, $(10,10,6)$ \\
57--58 & $(10,0,4)$, $(10,6,6)$ \\
59--62 & $(10,0,4)$, $(11,10,6)$ \\
63--103 & $(10,10,6)$, $(10,11,6)$, $(11,10,6)$ \\
\noalign{\smallskip}
\hline
\end{tabular}
\end{small}
\end{center}
\bigskip

\begin{small}
Banker's behavioral strategy in the Nash equilibrium, listed by interval (1--103).  For each interval, we have listed above the two or three information sets on which Banker mixes when $\theta$ belongs to that interval.  Below we list the probabilities that Banker draws in the case of each such information set.  For intervals 63--103, behavioral strategies are non unique, and the two extreme points are given.\medskip

\noindent1. $9299/15664$.\\ 

\noindent2. $116615/203632$.\\

\noindent3. $86663/203632$.\\

\noindent4. $82391/203632$.\\

\noindent5. $73847/203632$.\\

\noindent6. $3959/203632$.\\

\noindent7. $3959/4272$.\\

\noindent8. $203319/203632$.\\

\noindent9. $199047/203632$.\\

\noindent10. $200471/203632$.\\  

\noindent11. $190503/203632$.\\  

\noindent12. $196199/203632$.\\ 

\noindent13. $179111/203632$.\\

\noindent14. $13011/15664$.\\

\noindent15. $164871/203632$.\\

\noindent16. $160599/203632$.\\  

\noindent17. $156327/203632$.\\  

\noindent18. $14341/18512$.\\ 

\noindent19. $147783/203632$.\\

\noindent20.
\begin{eqnarray*}
&&[-2637864161208199 + 165946542578699760\,\theta + 13309178004407488\,\theta^2\\
&&{}- 107899586713936896\,\theta^3 - 14 (9081628501 + 8007062744\,\theta) (427951969 \\
&&{}- 54744421840\,\theta + 1809650847808\,\theta^2 - 2714674851840\,\theta^3 + 1043139723264\,\theta^4)^{1/2}]\\
&&{}/[85488 (91 + 1032\,\theta) (2503127 - 300213592\,\theta+ 262683264\,\theta^2)],\\
&&[4 (-180862924147939 + 10010839445109979\,\theta - 35860055053960408\,\theta^2\\
&&{}+ 21629527664623872\,\theta^3)- 77 (542832349 - 1353122180\,\theta) (427951969\\
&&{}- 54744421840\,\theta + 1809650847808\,\theta^2- 2714674851840\,\theta^3 + 1043139723264\,\theta^4)^{1/2}]\\
&&{}/[85488 (1123 - 896\,\theta) (2503127 - 300213592\,\theta + 262683264\,\theta^2)].
\end{eqnarray*}

\noindent21.
\begin{eqnarray*}
&&[-2627592769801351 + 165310760974835184\,\theta + 12718635757094080\,\theta^2\\
&&{}- 107072930867241984\,\theta^3 - 10 (12664628471 + 11122723528\,\theta) (427951969\\
&&{}- 54744421840\,\theta + 1809650847808\,\theta^2 - 2714674851840\,\theta^3+ 1043139723264\,\theta^4)^{1/2}]\\
&&{}/[85488 (91 + 1032\,\theta) (2503127 - 300213592\,\theta + 262683264\,\theta^2)],\\
&&[4 (-181718873431843 + 10160377178231899\,\theta - 35995096111169752\,\theta^2\\
&&{}+ 21646542336721152\,\theta^3) -55 (756956111 -1884566764\,\theta) (427951969 \\
&&{}- 54744421840\,\theta + 1809650847808\,\theta^2 - 2714674851840\,\theta^3 + 1043139723264\,\theta^4)^{1/2}]\\
&&{}/[85488 (1123 - 896\,\theta) (2503127 - 300213592\,\theta + 262683264\,\theta^2)].
\end{eqnarray*}

\noindent22.
\begin{eqnarray*}
&&[-2603821471987615 + 163767421736128368\,\theta + 15377831007461056\,\theta^2\\
&&{}- 108214447526568960\,\theta^3 - 2 (62751886879 + 56188136456\,\theta) (427951969 \\
&&{}- 54744421840\,\theta + 1809650847808\,\theta^2 - 2714674851840\,\theta^3+ 1043139723264\,\theta^4)^{1/2}]\\
&&{}/[85488 (91 + 1032\,\theta) (2503127 - 300213592\,\theta + 262683264\,\theta^2)],\\
&&[4 (-183602876710045 + 10494021498064759\,\theta - 36567577817322424\,\theta^2 \\
&&{}+ 21891747481896192\,\theta^3)- (41268298721 - 103372711300\,\theta) (427951969 \\
&&{}- 54744421840\,\theta + 1809650847808\,\theta^2 - 2714674851840\,\theta^3 + 1043139723264\,\theta^4)^{1/2}]\\
&&{}/[85488 (1123 - 896\,\theta) (2503127 - 300213592\,\theta + 262683264\,\theta^2)].
\end{eqnarray*}

\noindent23.
\begin{eqnarray*}
&&[-2580322521673303 + 162268998780855024\,\theta + 18537511815172288\,\theta^2 \\
&&{}- 109854062669128704\,\theta^3 - 2 (62174048827 + 56688004520\,\theta) (427951969 \\
&&{}- 54744421840\,\theta + 1809650847808\,\theta^2- 2714674851840\,\theta^3 + 1043139723264\,\theta^4)^{1/2}]\\
&&{}/[85488 (91 + 1032\,\theta) (2503127 - 300213592\,\theta + 262683264\,\theta^2)],\\
&&[4 (-185512999228951 + 10832048893656163\,\theta - 37133170411338520\,\theta^2 \\
&&{}+ 22129161641523456\,\theta^3)- (40898960969 - 103041458452\,\theta) (427951969\\
&&{}- 54744421840\,\theta + 1809650847808\,\theta^2 - 2714674851840\,\theta^3 + 1043139723264\,\theta^4)^{1/2}]\\
&&{}/[85488 (1123 - 896\,\theta) (2503127 - 300213592\,\theta + 262683264\,\theta^2)].
\end{eqnarray*}

\noindent24.
\begin{eqnarray*}
&&[-2562090951567631 + 161111044751770224\,\theta + 21078455591838400\,\theta^2\\
&&{}- 111212335304991744\,\theta^3 - 10 (12344704427 + 11412869800\,\theta) (427951969\\
&&{} - 54744421840\,\theta + 1809650847808\,\theta^2 - 2714674851840\,\theta^3 + 1043139723264\,\theta^4)^{1/2}]\\
&&{}/[85488 (91 + 1032\,\theta) (2503127 - 300213592\,\theta + 262683264\,\theta^2)],\\
&&[4 (-187003003283773 + 11095689602853679\,\theta - 37571927576883832\,\theta^2\\
&&{}+ 22312518601420032\,\theta^3)- 5 (8122171309 - 20555003732\,\theta) (427951969 \\
&&{}- 54744421840\,\theta + 1809650847808\,\theta^2 - 2714674851840\,\theta^3 + 1043139723264\,\theta^4)^{1/2}]\\
&&{}/[85488 (1123 - 896\,\theta) (2503127 - 300213592\,\theta + 262683264\,\theta^2)].
\end{eqnarray*}

\noindent25.
\begin{eqnarray*}
&&[-2554666489378663 + 160678944858495984\,\theta + 22882410863758528\,\theta^2\\
&&{}- 112506021410694144\,\theta^3 - 22 (5592821417 + 5192635960\,\theta) (427951969 \\
&&{}- 54744421840\,\theta + 1809650847808\,\theta^2 - 2714674851840\,\theta^3 + 1043139723264\,\theta^4)^{1/2}]\\
&&{}/[85488 (91 + 1032\,\theta) (2503127 - 300213592\,\theta + 262683264\,\theta^2)],\\
&&[4 (-187678889651131 + 11214940015021963\,\theta - 37750124773351000\,\theta^2\\
&&{}+ 22379970176307456\,\theta^3)- (40480168409 - 102585412852\,\theta) (427951969\\
&&{}- 54744421840\,\theta + 1809650847808\,\theta^2 - 2714674851840\,\theta^3 + 1043139723264\,\theta^4)^{1/2}]\\
&&{}/[85488 (1123 - 896\,\theta) (2503127 - 300213592\,\theta+ 262683264\,\theta^2)].
\end{eqnarray*}

\noindent26.
\begin{eqnarray*}
&&[-109667316291 - 138506682368\,\theta + 11 (703164835 + 889696256\,\theta) \sqrt{201}]\\
&&{}/(1101541376\,\theta),\\
&&[-338546558911 - 401794387496\,\theta + 11 (2172391441 + 2606191832\,\theta) \sqrt{201}]\\
&&{}/(157520416768\,\theta).
\end{eqnarray*}\smallskip

\noindent27. $(15351935 - 35936256\,\theta)/[73216 (457 - 384\,\theta)], \quad 687437/13593008$.\\

\noindent28. $(176493157 - 408201216\,\theta)/[805376 (457 - 384\,\theta)], \quad 603437/13593008$.\\

\noindent29.
\begin{eqnarray*}
&&[2 (-3401657391041 + 68001517760236\,\theta + 65549981371776\,\theta^2\\
&&{}- 18064055992320\,\theta^3) - (234296323 + 229472256\,\theta) (427951969 \\
&&{}- 29310831064\,\theta + 380780676496\,\theta^2 - 295717017600\,\theta^3 + 24787353600\,\theta^4)^{1/2}]\\
&&{}/[88618176 (20687 - 246356\,\theta - 58880\,\theta^2)],\\
&&{}[6992466801009 - 115483316705054\,\theta - 282906672700904\,\theta^2\\
&&{}+ 42627839009280\,\theta^3 + 2 (91888308 + 172409981\,\theta) (427951969\\
&&{}- 29310831064\,\theta + 380780676496\,\theta^2 - 295717017600\,\theta^3 + 24787353600\,\theta^4)^{1/2}]\\
&&{}/[12224784 (91 + 412\,\theta) (20687 - 246356\,\theta - 58880\,\theta^2)].
\end{eqnarray*}

\noindent30.
\begin{eqnarray*}
&&[2 (-1176322005507 + 24301595351492\,\theta + 20841715876992\,\theta^2 \\
&&{}- 6021351997440\,\theta^3) - (86304161 + 76490752\,\theta) (427951969 \\
&&{}- 29310831064\,\theta + 380780676496\,\theta^2 - 295717017600\,\theta^3+ 24787353600\,\theta^4)^{1/2}]\\
&&{}/[29539392 (20687- 246356\,\theta - 58880\,\theta^2)],\\
&&[2475929159803 - 42408289900618\,\theta - 89280316869368\,\theta^2\\
&&{} + 14473862837760\,\theta^3 + 2 (34136636 + 56629727\,\theta) (427951969 \\
&&{}- 29310831064\,\theta + 380780676496\,\theta^2 - 295717017600\,\theta^3+ 24787353600\,\theta^4)^{1/2}]\\
&&{}/[4074928 (91 + 412\,\theta) (20687 - 246356\,\theta - 58880\,\theta^2)].
\end{eqnarray*}

\noindent31.
\begin{eqnarray*}
&&[4 (4139832799 + 5911350336 \,\theta) - (652736883 + 860864512 \,\theta) \sqrt{649}]\\
&&{}/[4816963840 (1 - \theta)],\\
&&[294698354563 - 820517442368 \,\theta -8 (652736883 - 1650119038 \,\theta) \sqrt{649}]\\
&&{}/[5510606632960 (1 - \theta)].
\end{eqnarray*}

\noindent32.
\begin{eqnarray*}
&&[8 (65827819247 + 2722044586946\,\theta + 184409726400\,\theta^2 - 3568752525312\,\theta^3)\\
&&{}- (200226067 + 229472256\,\theta) (322094809 - 5158039588\,\theta \\
&&{}+ 33447454948\,\theta^2 - 17645912064\,\theta^3 + 15479341056\,\theta^4)^{1/2}]\\
&&{}/[12308080 (17947 - 106838\,\theta - 73728\,\theta^2)],\\
&&[-217276306691 + 576449396780\,\theta - 11798152185932\,\theta^2 - 8392742461440\,\theta^3 \\
&&{}+ 2 (12550397 + 5992520\,\theta) (322094809 - 5158039588\,\theta\\
&&{} + 33447454948\,\theta^2 - 17645912064\,\theta^3 + 15479341056\,\theta^4)^{1/2}]\\
&&{}/[4074928 (23 + 50\,\theta) (17947- 106838\,\theta - 73728\,\theta^2)].
\end{eqnarray*}\smallskip

\noindent33. $3 (10536967 - 105961472\,\theta)/[805376 (457 - 384\,\theta)], \quad 24449/1045616$.\\

\noindent34. $3 (16597943 - 110262272\,\theta)/[805376 (457 - 384\,\theta)], \quad 334637/13593008$.\\

\noindent35.
\begin{eqnarray*}
&&\{4384 (45018818755 - 426890812458\,\theta - 385188740352\,\theta^2) + 3 (280115699 \\
&&{}+ 229472256\,\theta)[137 (655- 6738\,\theta) (89735 - 1798626\,\theta + 875520\,\theta^2)]^{1/2}\}\\
&&{}/[49232320 (987085 - 8578230\,\theta - 1575936\,\theta^2)],\\
&&\{2603 (47933200645 - 261429995712\,\theta - 2560125795156\,\theta^2 - 33333700608\,\theta^3)\\
&&{}+ 2 (337872733 + 1269627654\,\theta) [137 (655 - 6738\,\theta) (89735 - 1798626\,\theta \\
&&{}+ 875520\,\theta^2)]^{1/2}\}/[4074928 (115 + 822\,\theta) (987085 - 8578230\,\theta - 1575936\,\theta^2)].
\end{eqnarray*}

\noindent36.
\begin{eqnarray*}
&&\{52608 (3860738335 - 37434676466 \,\theta - 31361654976 \,\theta^2) + 9 (107321057 \\
&&{}+ 76490752 \,\theta)[137 (655 - 6738 \,\theta) (89735 -1798626 \,\theta + 875520 \,\theta^2)]^{1/2}\}\\
&&{}/[49232320 (987085 - 8578230 \,\theta - 1575936 \,\theta^2)],\\
&&\{2603 (52765968805 - 331033911168 \,\theta - 2474953725780 \,\theta^2 - 33333700608 \,\theta^3)\\
&&{} + 2 (407966317 + 1231394790 \,\theta) [137 (655 - 6738 \,\theta) (89735 - 1798626 \,\theta \\
&&{}+ 875520 \,\theta^2)]^{1/2}\}/[4074928 (115 + 822 \,\theta) (987085 - 8578230 \,\theta - 1575936 \,\theta^2)].
\end{eqnarray*}

\noindent37.
\begin{eqnarray*}
&&[798754007281061 - 2562457837338161\,\theta + 1340745216377498\,\theta^2- (77854147501 \\
&&{}- 200906666222\,\theta)(105051121 - 132870958\,\theta + 40507681\,\theta^2)^{1/2}]\\
&&{}/[15782400\,\theta (1916806 - 1701305\,\theta)],\\
&&[-917863534375073 + 922812135152917\,\theta - 271230091643286\,\theta^2 + 7 (12348483719 \\
&&{}- 119039262\,\theta)(105051121 - 132870958\,\theta + 40507681\,\theta^2)^{1/2}]\\
&&{}/[64970880\,\theta (1916806- 1701305\,\theta)].
\end{eqnarray*}

\noindent38.
\begin{eqnarray*}
&&[798754007281061 - 2564493096580721\,\theta + 1342842314219738\,\theta^2- (77854147501 \\
&&{} - 200780801582\,\theta)(105051121 - 132870958\,\theta + 40507681\,\theta^2)^{1/2}]\\
&&{}/[15782400\,\theta (1916806 - 1701305\,\theta)],\\
&&[-917863534375073 + 919027663855957\,\theta - 267703749652566\,\theta^2+ 11 (7858126003 \\
&&{}- 82340214\,\theta)(105051121 - 132870958\,\theta + 40507681\,\theta^2)^{1/2}]\\
&&{}/[64970880\,\theta (1916806- 1701305\,\theta)].
\end{eqnarray*}

\noindent39.
\begin{eqnarray*}
&&[798754007281061 - 2564863112975921\,\theta + 1343521523613338\,\theta^2- (77854147501 \\
&&{} - 200628895982\,\theta)(105051121 - 132870958\,\theta + 40507681\,\theta^2)^{1/2}]\\
&&{}/[15782400\,\theta (1916806 - 1701305\,\theta)],\\
&&[-917863534375073 + 914460198497557\,\theta - 263447819663766\,\theta^2+ (86439386033 \\
&&{}- 993203154\,\theta)(105051121 - 132870958\,\theta + 40507681\,\theta^2)^{1/2}]\\
&&{}/[64970880\,\theta (1916806 - 1701305\,\theta)].
\end{eqnarray*}

\noindent40.
\begin{eqnarray*}
&&\{-284986850133089 + 946557575862016\,\theta - 546718746563388\,\theta^2+ (7051353127 \\
&&{} - 18162548178\,\theta)  [11 (148467731 - 217532180\,\theta + 79180940\,\theta^2)]^{1/2}\}\\
&&{}/[4734720\,\theta (982781 -866742\,\theta)],\\
&&\{-735125564999631 + 174455145453568\,\theta + 271003610668540\,\theta^2+ (18052390551 \\
&&{} + 10500144302\,\theta)[11 (148467731 - 217532180\,\theta + 79180940\,\theta^2)]^{1/2}\}\\
&&{}/[124812480\,\theta (982781 - 866742\,\theta)].
\end{eqnarray*}

\noindent41.
\begin{eqnarray*}
&&(43502459 - 483147725\,\theta + 2962380936\,\theta^2)/[317687040\,\theta (1 - \theta)],\\
&&(-6214637 + 42470875\,\theta + 131926140\,\theta^2)/[8824640\,\theta (1 - \theta)].
\end{eqnarray*}

\noindent42.
\begin{eqnarray*}
&&[-2334903194937404 + 657925124267365\,\theta + 1688047556447796\,\theta^2 + (52396734481 \\
&&{} + 37871699280\,\theta)(1985556496 - 3981633496\,\theta + 1996095769\,\theta^2)^{1/2}]\\
&&{}/[29986560\,\theta (125140 - 114591\,\theta)],\\
&&[281675545213036- 88166644468081\,\theta - 194185450056546\,\theta^2- (6325935499 \\
&&{} + 4202797098\,\theta)(1985556496 - 3981633496\,\theta + 1996095769\,\theta^2)^{1/2}]\\
&&{}/[41604160\,\theta (125140 - 114591\,\theta)].
\end{eqnarray*}

\noindent43.
\begin{eqnarray*}
&&[-2334903194937404 + 737140606492005\,\theta + 1608550893581876\,\theta^2+ (52396734481 \\
&&{} + 36075837520\,\theta)(1985556496 - 3981633496\,\theta + 1996095769\,\theta^2)^{1/2}]\\
&&{}/[29986560\,\theta (125140 - 114591\,\theta)],\\
&&[845026635639108 - 294352979210323\,\theta - 552789239356198\,\theta^2- (18977806497 \\
&&{} + 11981128574\,\theta)(1985556496 - 3981633496\,\theta + 1996095769\,\theta^2)^{1/2}]\\
&&{}/[124812480\,\theta (125140 - 114591\,\theta)]
\end{eqnarray*}

\noindent44.
\begin{eqnarray*}
&&[-2334903194937404 + 721419128102245\,\theta + 1624348604470836\,\theta^2+ (52396734481 \\
&&{} + 36437524560\,\theta)(1985556496 - 3981633496\,\theta + 1996095769\,\theta^2)^{1/2}]\\
&&{}/[29986560\,\theta (125140 - 114591\,\theta)],\\
&&[281675545213036 - 96287248377201\,\theta - 186102366699106\,\theta^2 - (6325935499 \\
&&{} + 4035819818\,\theta)(1985556496 - 3981633496\,\theta + 1996095769\,\theta^2)^{1/2}]\\
&&{}/[41604160\,\theta (125140 - 114591\,\theta)].
\end{eqnarray*}\smallskip

\noindent45. $184759/203632$.\\

\noindent46. $16149/18512$.\\

\noindent47. $170519/203632$.\\

\noindent48.
\begin{eqnarray*}
&&[2 (-22019125783642 + 140415421565879\,\theta + 263496530920338\,\theta^2\\
&&{}- 248956268391456\,\theta^3) - (650262513 + 1334577884\,\theta) (4342414609\\
&&{}- 76590765160\,\theta + 378797212888\,\theta^2 - 429088870560\,\theta^3 + 139230322704\,\theta^4)^{1/2}]\\
&&{}/[199472 (61 + 318\,\theta) (724867 - 2858024\,\theta + 712056\,\theta^2)],\\
&&[2 (903244069293748 - 6198100588189621\,\theta - 449327197747946\,\theta^2\\
&&{}+ 2718572967797496\,\theta^3) + (21439446889 + 13673385976\,\theta) (4342414609 \\
&&{}- 76590765160\,\theta + 378797212888\,\theta^2 - 429088870560\,\theta^3 + 139230322704\,\theta^4)^{1/2}]\\
&&{}/[7180992 (1123 - 734\,\theta) (724867 - 2858024\,\theta + 712056\,\theta^2)].
\end{eqnarray*}

\noindent49.
\begin{eqnarray*}
&&[6 (-7082293386194 + 44723312966413\,\theta + 91113047241366\,\theta^2\\
&&{}- 84728489228832\,\theta^3) - 3 (210501491 + 455548948\,\theta) (4342414609\\
&&{} - 76590765160\,\theta + 378797212888\,\theta^2 - 429088870560\,\theta^3 + 139230322704\,\theta^4)^{1/2}]\\
&&{}/[199472 (61 + 318\,\theta) (724867 - 2858024\,\theta + 712056\,\theta^2)],\\
&&[2 (915459705899648 - 6188482012904981\,\theta - 566736684317146\,\theta^2 \\
&&{}+ 2788009050700536\,\theta^3)+ 3 (7022899163 + 4706836712\,\theta) (4342414609\\
&&{}- 76590765160\,\theta + 378797212888\,\theta^2 - 429088870560\,\theta^3  + 139230322704\,\theta^4)^{1/2}]\\
&&{}/[7180992 (1123 - 734\,\theta) (724867 - 2858024\,\theta + 712056\,\theta^2)].
\end{eqnarray*}

\noindent50.
\begin{eqnarray*}
&&[6 (-6917046905938 + 43284182897933\,\theta + 92146623955094\,\theta^2\\
&&{}- 84770872932896\,\theta^3) - (616458585 + 1367072828\,\theta) (4342414609\\
&&{} - 76590765160\,\theta + 378797212888\,\theta^2 - 429088870560\,\theta^3+ 139230322704\,\theta^4)^{1/2}]\\
&&{}/[199472 (61 + 318\,\theta) (724867 - 2858024\,\theta + 712056\,\theta^2)],\\
&&[2 (928183684879360 - 6196629250920469\,\theta - 554930078113754\,\theta^2\\
&&{}+ 2780013337223928\,\theta^3) + (20682519697 + 14177472184\,\theta) (4342414609 \\
&&{}- 76590765160\,\theta + 378797212888\,\theta^2 - 429088870560\,\theta^3  + 139230322704\,\theta^4)^{1/2}]\\
&&{}/[7180992 (1123 - 734\,\theta) (724867 - 2858024\,\theta + 712056\,\theta^2)].
\end{eqnarray*}

\noindent51.
\begin{eqnarray*}
&&[2 (-19762514649202 + 120562180679207\,\theta + 290222158824114\,\theta^2\\
&&{} - 259869978196320\,\theta^3)-(587389377 + 1397331196\,\theta) (4342414609\\
&&{} - 76590765160\,\theta + 378797212888\,\theta^2 - 429088870560\,\theta^3 + 139230322704\,\theta^4)^{1/2}]\\
&&{}/[199472 (61 + 318\,\theta) (724867 - 2858024\,\theta + 712056\,\theta^2)],\\
&&[6 (316454142117932 - 2064637497306103\,\theta - 218479907076046\,\theta^2\\
&&{}+ 946225401975656\,\theta^3) + (20039736521 + 14638559384\,\theta) (4342414609\\
&&{}- 76590765160\,\theta + 378797212888\,\theta^2 - 429088870560\,\theta^3+ 139230322704\,\theta^4)^{1/2}]\\
&&{}/[7180992 (1123 - 734\,\theta) (724867 - 2858024\,\theta + 712056\,\theta^2)].
\end{eqnarray*}

\noindent52.
\begin{eqnarray*}
&&[6 (-6278522372746 + 36969117047917\,\theta + 101335400112886\,\theta^2\\
&&{} - 88083885950112\,\theta^3) - (556448201 + 1417830780\,\theta) (4342414609 \\
&&{}- 76590765160\,\theta + 378797212888\,\theta^2 - 429088870560\,\theta^3 + 139230322704\,\theta^4)^{1/2}]\\
&&{}/[199472 (61 + 318\,\theta) (724867 - 2858024\,\theta + 712056\,\theta^2)],\\
&&[2 (973259067979016 - 6198361897180373\,\theta - 713373605620858\,\theta^2\\
&&{}+ 2871640643055864\,\theta^3) + 3 (6438154667 + 4996574056\,\theta) (4342414609 \\
&&{}- 76590765160\,\theta + 378797212888\,\theta^2 - 429088870560\,\theta^3 + 139230322704\,\theta^4)^{1/2}]\\
&&{}/[7180992 (1123 - 734\,\theta) (724867 - 2858024\,\theta + 712056\,\theta^2)].
\end{eqnarray*}

\noindent53.
\begin{eqnarray*}
&&[2 (-18248357106122 + 103568483914519\,\theta + 313854074809554\,\theta^2\\
&&{}- 265953817914528\,\theta^3) - (532074257 + 1420623388\,\theta) (4342414609 \\
&&{}- 76590765160\,\theta + 378797212888\,\theta^2 - 429088870560\,\theta^3 + 139230322704\,\theta^4)^{1/2}]\\
&&{}/[199472 (61 + 318\,\theta) (724867 - 2858024\,\theta + 712056\,\theta^2)],\\
&&[6 (331209468345756 - 2070080604896135\,\theta - 234351938675022\,\theta^2\\
&&{} + 954660477078824\,\theta^3) + 11 (1699658779 + 1373370280\,\theta) (4342414609 \\
&&{}- 76590765160\,\theta + 378797212888\,\theta^2 - 429088870560\,\theta^3 + 139230322704\,\theta^4)^{1/2}]\\
&&{}/[7180992 (1123 - 734\,\theta) (724867 - 2858024\,\theta + 712056\,\theta^2)].
\end{eqnarray*}

\noindent54.
\begin{eqnarray*}
&&[6 (-5300007539814 + 26177838556189\,\theta + 116433251864358\,\theta^2 \\
&&{}- 92183059057120\,\theta^3) - (451441537 + 1467760156\,\theta) (4342414609 \\
&&{}- 76590765160\,\theta + 378797212888\,\theta^2 - 429088870560\,\theta^3 + 139230322704\,\theta^4)^{1/2}]\\
&&{}/[199472 (61 + 318\,\theta) (724867 - 2858024\,\theta + 712056\,\theta^2)],\\
&&[6 (352209999291332 - 2075259364504423\,\theta - 275899885540846\,\theta^2 \\
&&{}+ 978020769556136\,\theta^3)  + 223 (75265127 + 71512328\,\theta) (4342414609 \\
&&{}- 76590765160\,\theta + 378797212888\,\theta^2 - 429088870560\,\theta^3 + 139230322704\,\theta^4)^{1/2}]\\
&&{}/[7180992 (1123 - 734\,\theta) (724867 - 2858024\,\theta + 712056\,\theta^2)].
\end{eqnarray*}

\noindent55.
\begin{eqnarray*}
&&[6 (-3482296256998 + 10347407802909\,\theta + 127802595715366\,\theta^2\\
&&{}- 92649279801824\,\theta^3) - (285936769 + 1472445980\,\theta) (4342414609 \\
&&{}- 76590765160\,\theta + 378797212888\,\theta^2 - 429088870560\,\theta^3 + 139230322704\,\theta^4)^{1/2}]\\
&&{}/[199472 (61 + 318\,\theta) (724867 - 2858024\,\theta + 712056\,\theta^2)],\\
&&[2 (1196593766650828 - 6315397711683637\,\theta - 697826988385226\,\theta^2\\
&&{}+ 2846109460425720\,\theta^3) + (12536167609 + 16573831672\,\theta) (4342414609 \\
&&{}- 76590765160\,\theta + 378797212888\,\theta^2 - 429088870560\,\theta^3 + 139230322704\,\theta^4)^{1/2}]\\
&&{}/[7180992 (1123 - 734\,\theta) (724867 - 2858024\,\theta + 712056\,\theta^2)].
\end{eqnarray*}

\noindent56.
\begin{eqnarray*}
&&[2329241845 + 8390949192\,\theta - 10477043712\,\theta^2 - (801955 + 2285568\,\theta)(6604081 \\
&&{} - 15886512\,\theta + 21013056\,\theta^2)^{1/2}]/(19696435200\,\theta),\\
&&[-475850935 - 374893316\,\theta + 873086976\,\theta^2 + 3 (72905 + 63488\,\theta) (6604081\\
&&{}  - 15886512\,\theta + 21013056\,\theta^2)^{1/2}]/(6519884800\,\theta).
\end{eqnarray*}

\noindent57.
\begin{eqnarray*}
&&(-36087975 + 1526749327\,\theta - 422825752\,\theta^2)/[3718840320\,\theta (1 - \theta)],\\
&&(-3280725 + 1561229\,\theta + 357664696\,\theta^2)/[103301120\,\theta (1 - \theta)].
\end{eqnarray*}

\noindent58.
\begin{eqnarray*}
&&(-28492695 + 1664069887\,\theta - 567741592\,\theta^2)/[3718840320\,\theta (1 - \theta)],\\
&&3 (-863415 + 3477663\,\theta + 116034152\,\theta^2)/[103301120\,\theta (1 - \theta)].
\end{eqnarray*}\smallskip

\noindent59. $4783555/8409024, \quad (33087139 - 267345726\,\theta)/[73216 (481 - 3850\,\theta)]$.\\

\noindent60. $5017139/8409024, \quad (33087139 - 265045822\,\theta)/[73216 (481 - 3850\,\theta)]$.\\

\noindent61. $5106979/8409024, \quad (433295863 - 3466313478\,\theta)/[951808 (481 - 3850\,\theta)]$.\\

\noindent62. $589403/934336, \quad (430132807 - 3436396070\,\theta)/[951808 (481 - 3850\,\theta)]$.\\

\noindent63. $0, 3104397/10469888, 8706221/10469888;\quad 1034799/3807232, 0, 175057/327184$.\\
  
\noindent64. $0, 5206653/10469888, 9265325/10469888;\quad 1735551/3807232, 0, 253667/654368$.\\
  
\noindent65. $0, 5296493/10469888, 9211421/10469888;\quad 5296493/11421696, 0, 244683/654368$.\\

\noindent66. $0, 5386333/10469888, 9085645/10469888;\quad 5386333/11421696, 0, 231207/654368$.\\

\noindent67. $0, 5152749/10469888, 9644749/10469888;\quad 1717583/3807232, 0, 140375/327184$.\\

\noindent68. $0, 476599/951808, 734993/805376;\quad 476599/1038336, 0, 16845/40898$.\\

\noindent69. $0,5260557/10469888, 9465069/10469888;\quad 1753519/3807232, 0, 10107/25168$.\\

\noindent70. $0, 5332429/10469888, 8818221/10469888;\quad 5332429/11421696, 0, 108931/327184$.\\

\noindent71. $0, 411569/805376, 8943997/10469888;\quad 411569/878592, 0, 28075/81796$.\\

\noindent72. $0, 5332429/10469888, 681089/805376;\quad 5332429/11421696, 0, 55027/163592$.\\

\noindent73. $0, 5260557/10469888, 67829/86528;\quad 1753519/3807232,0, 46043/163592$.\\

\noindent74. $0, 7362813/10469888, 8127437/10469888;\quad 2454271/3807232, 0, 47789/654368$.\\

\noindent75. $0, 7344845/10469888, 8253213/10469888;\quad 7344845/11421696, 0, 56773/654368$.\\

\noindent76. $0, 7398749/10469888, 8199309/10469888;\quad 7398749/11421696, 0, 50035/654368$.\\

\noindent77. $0, 7165165/10469888, 7799949/10469888;\quad 7165165/11421696, 0, 19837/327184$.\\

\noindent78. $0, 6787837/10469888, 7530429/10469888;\quad 6787837/11421696, 0, 11603/163592$.\\

\noindent79. $0, 7434685/10469888, 7171069/10469888;\quad 7171069/11421696, 4119/163592, 0$.\\

\noindent80. $0, 573281/805376, 7260909/10469888;\quad 2420303/3807232, 749/40898, 0$.\\

\noindent81. $0, 7470621/10469888, 7242941/10469888;\quad 7242941/11421696, 7115/327184, 0$.\\

\noindent82. $0, 7488589/10469888, 7189037/10469888;\quad 7189037/11421696, 851/29744, 0$.\\

\noindent83. $0, 7704205/10469888, 6542189/10469888;\quad 6542189/11421696, 36313/327184, 0$.\\

\noindent84. $0, 692215/951808, 6452349/10469888;\quad 2150783/3807232, 36313/327184, 0$.\\

\noindent85. $0, 7722173/10469888, 6703901/10469888;\quad 6703901/11421696, 31821/327184, 0$.\\

\noindent86. $0, 7776077/10469888, 6614061/10469888;\quad 2204687/3807232, 36313/327184, 0$.\\

\noindent87. $0, 7542493/10469888, 6054957/10469888;\quad 2018319/3807232, 92971/654368, 0$.\\

\noindent88. $0, 7326877/10469888, 5408109/10469888;\quad 1802703/3807232, 119923/654368, 0$.\\

\noindent89. $0, 7165165/10469888, 5390141/10469888;\quad 5390141/11421696, 110939/654368, 0$.\\

\noindent90. $0, 7093293/10469888, 5174525/10469888;\quad 5174525/11421696, 119923/654368, 0$.\\

\noindent91. $0, 6859709/10469888, 5893373/10469888;\quad 5893373/11421696, 15099/163592, 0$.\\

\noindent92. $0, 537345/805376, 5875405/10469888;\quad 5875405/11421696, 17345/163592, 0$.\\

\noindent93. $0, 7003453/10469888, 6001181/10469888;\quad 6001181/11421696, 31321/327184, 0$.\\

\noindent94. $0, 6949549/10469888, 6126957/10469888;\quad 2042319/3807232, 12853/163592, 0$.\\

\noindent95. $0, 620343/951808, 6216797/10469888;\quad 6216797/11421696, 2371/40898, 0$.\\

\noindent96. $0, 6662061/10469888, 50933/86528;\quad 560263/1038336, 15599/327184, 0$.\\

\noindent97. $0, 5727725/10469888, 6482381/10469888;\quad 5727725/11421696, 0, 23583/327184$.\\

\noindent98. $0, 6374573/10469888, 6841741/10469888;\quad 6374573/11421696, 0, 1123/25168$.\\

\noindent99. $0, 6500349/10469888, 6787837/10469888;\quad 2166783/3807232, 0, 1123/40898$.\\

\noindent100. $0, 6554253/10469888, 6769869/10469888;\quad 2184751/3807232, 0, 3369/163592$.\\

\noindent101. $0, 6608157/10469888, 6680029/10469888;\quad 2202719/3807232, 0, 1123/163592$.\\

\noindent102. $0, 602375/951808, 6733933/10469888;\quad 602375/1038336, 0, 3369/327184$.\\

\noindent103. $0, 6715965/10469888, 6715965/10469888;\quad 2238655/3807232, 0, 0$.\bigskip

The Players' strategies in the Nash equilibrium. $p_1$ (resp., $p_2$) is the probability that Player 1 (resp., Player 2) draws on 5.\medskip

\noindent(intervals 1--6) 
\begin{eqnarray*}
p_1&=&0,\quad p_2=9(89 - 12\,\theta)/[11(89-68\,\theta)].
\end{eqnarray*}

\noindent(interval 7) 
\begin{eqnarray*}
p_1&=&0,\quad p_2=89 (53 - 534\,\theta)/[16(267+214\,\theta)].
\end{eqnarray*}

\noindent(intervals 8--19) 
\begin{eqnarray*}
p_1&=&0,\quad p_2=3(267-88\,\theta)/(979 - 800\,\theta).
\end{eqnarray*}

\noindent(intervals 20--25)
\begin{eqnarray*}
p_1&=&[-6191 + 962008 \,\theta - 705792 \,\theta^2 - (427951969 - 54744421840 \,\theta\\
&&{}+ 1809650847808 \,\theta^2 - 2714674851840 \,\theta^3+ 1043139723264 \,\theta^4)^{1/2}]\\
&&{}/[32 (151 - 12746 \,\theta + 11088 \,\theta^2)],\\
p_2&=&[139055 - 1171096\,\theta + 945408\,\theta^2 + (427951969 - 54744421840\,\theta \\
&&{}+ 1809650847808\,\theta^2 - 2714674851840\,\theta^3 + 1043139723264\,\theta^4)^{1/2}]\\
&&{} /[32 (4521 + 3992 \,\theta - 6400 \,\theta^2)].
\end{eqnarray*}

\noindent(interval 26) 
\begin{eqnarray*}
p_1&=&3 [151 - 8800 \,\theta + 4128 \,\theta^2 + 352 \sqrt{201}\,\theta(1-\theta)]\\
&&{}/(151 - 13656 \,\theta - 3072 \,\theta^2),\\
p_2&=&(-129 + 11 \sqrt{201})/32.
\end{eqnarray*}

\noindent(intervals 27--28, 33--34) 
\begin{eqnarray*}
p_1&=&71/176,\quad p_2=3(1371-568\,\theta)/[11(457-384\,\theta)].
\end{eqnarray*}

\noindent(intervals 29--30)
\begin{eqnarray*}
p_1&=&[-6191 - 317036 \,\theta + 252672 \,\theta^2 + (427951969 - 29310831064 \,\theta\\
&&{}+ 380780676496 \,\theta^2 - 295717017600 \,\theta^3 + 24787353600 \,\theta^4)^{1/2}]\\
&&{}/[32 (151 - 11410 \,\theta - 2304 \,\theta^2)],\\
p_2&=&[139055 - 476884\,\theta + 252672\,\theta^2 + (427951969 - 29310831064\,\theta \\
&&{}+ 380780676496\,\theta^2 - 295717017600\,\theta^3 + 24787353600\,\theta^4)^{1/2}]\\
&&{}/[32 (4521 - 140 \,\theta - 2304 \,\theta^2)].
\end{eqnarray*}

\noindent(interval 31) 
\begin{eqnarray*}
p_1&=&(-22 + \sqrt{649})/10,\\
p_2&=&[3 (4521 + 408\,\theta - 2816\,\theta^2) - 384\sqrt{649}\,\theta (1 - \theta)]\\
&&{}/(16577 - 18304\,\theta + 3840\,\theta^2).
\end{eqnarray*}

\noindent(interval 32)
\begin{eqnarray*}
p_1&=&[-5371 + 219734 \,\theta - 12288 \,\theta^2 - (322094809 - 5158039588 \,\theta + 33447454948 \,\theta^2\\
&&{} - 17645912064 \,\theta^3 + 15479341056 \,\theta^4)^{1/2}]/[32 (131 + 608 \,\theta + 768 \,\theta^2)],\\
p_2&=&[57403 - 93782\,\theta - 12288\,\theta^2 - (322094809 - 5158039588\,\theta + 33447454948\,\theta^2\\
&&{} - 17645912064\,\theta^3 + 15479341056\,\theta^4)^{1/2}]/[32 (1507 - 1088\,\theta + 768\,\theta^2)].
\end{eqnarray*}

\noindent(intervals 35--36)
\begin{eqnarray*}
p_1&=&\{-26855 + 16098 \,\theta + 410112 \,\theta^2 + [137 (655 - 6738 \,\theta) (89735 - 1798626 \,\theta \\
&&{}+ 875520 \,\theta^2)]^{1/2}\}/[32 (655 - 14928 \,\theta - 2304 \,\theta^2)],\\
p_2&=&\{287015 - 946530\,\theta + 410112\,\theta^2 + [137 (655 - 6738 \,\theta) (89735 - 1798626 \,\theta  \\
&&{}+ 875520 \,\theta^2)]^{1/2}\}/[32 (7535 + 852\,\theta - 2304\,\theta^2)].
\end{eqnarray*}

\noindent(intervals 37--39)
\begin{eqnarray*}
p_1&=&[-58295 + 911362 \,\theta - 690174 \,\theta^2 - 30\,\theta(105051121 - 132870958 \,\theta \\
&&{}+ 40507681 \,\theta^2)^{1/2}]/[16 (655 - 13778 \,\theta + 11616 \,\theta^2)],\\
p_2&=&[-7385 + 6319\,\theta + (105051121 - 132870958\,\theta + 40507681\,\theta^2)^{1/2}]\\
&&{}/[64\,(56 - 43\,\theta)].
\end{eqnarray*}

\noindent(interval 40)
\begin{eqnarray*}
p_1&=&\{-13439 +528700 \,\theta - 400044 \,\theta^2 - 10\,\theta [11(148467731 - 217532180 \,\theta\\
&&{}+ 79180940 \,\theta^2)]^{1/2}\}/[16 (151 - 10500 \,\theta + 9116 \,\theta^2)],\\
p_2&=&\{-29591  + 26434\,\theta +  [11(148467731 - 217532180\,\theta+ 79180940\,\theta^2)]^{1/2}\}\\
&&{}/[32 (415 - 338\,\theta)].
\end{eqnarray*}

\noindent(interval 41)
\begin{eqnarray*}
p_1&=&(-979 + 10596 \,\theta)/[176 (1 + 66 \,\theta)],\quad p_2=(7345 - 26658\,\theta)/[16(340-123\,\theta)].
\end{eqnarray*}

\noindent(intervals 42--44) 
\begin{eqnarray*}
p_1&=&[-13439 + 547264 \,\theta - 467106 \,\theta^2 - 10\,\theta (1985556496 - 3981633496 \,\theta\\
&&{}+ 1996095769 \,\theta^2)^{1/2}]/[16 (151 - 11046 \,\theta + 2264 \,\theta^2)],\\
p_2&=&[-36092 + 36133\,\theta + (1985556496 - 3981633496\,\theta+ 1996095769\,\theta^2)^{1/2}]\\
&&{}/[32 (316 - 317\,\theta)].
\end{eqnarray*}

\noindent(intervals 45--47)
\begin{eqnarray*}
p_1&=&0,\quad p_2=3(267-214\,\theta)/(979 - 926\,\theta).
\end{eqnarray*}

\noindent(intervals 48--55)  
\begin{eqnarray*}
p_1&=&[-19721 + 196244 \,\theta - 262148 \,\theta^2 - (4342414609 - 76590765160 \,\theta\\
&&{}+ 378797212888 \,\theta^2 - 429088870560 \,\theta^3 + 139230322704 \,\theta^4)^{1/2}]\\
&&{}/[32 (481 - 4758 \,\theta + 1400 \,\theta^2)],\\
p_2&=&[184265 - 596692\,\theta + 373764\,\theta^2 + (4342414609 - 76590765160\,\theta\\
&&{} + 378797212888\,\theta^2- 429088870560\,\theta^3 + 139230322704\,\theta^4)^{1/2}]\\
&&{}/[32 (4521 - 3706\,\theta + 128\,\theta^2)].
\end{eqnarray*}

\noindent(interval 56) 
\begin{eqnarray*}
p_1&=&[1443 - 13151 \,\theta + 4584 \,\theta^2 + \theta(6604081 - 15886512 \,\theta + 21013056 \,\theta^2)^{1/2}]\\
&&{}/(481 - 5002 \,\theta),\\
p_2&=&[3559 - 4584\,\theta - (6604081 - 15886512\,\theta + 21013056\,\theta^2)^{1/2}]/800.
\end{eqnarray*}

\noindent(intervals 57--58)
\begin{eqnarray*}
p_1&=&(27 + 64 \,\theta)/(9 + 280 \,\theta),\quad p_2=(943 - 3168\,\theta)/[16 (33 - 8\,\theta)].
\end{eqnarray*}

\noindent(intervals 59--62)
\begin{eqnarray*}
p_1&=&(1443 - 9304 \,\theta)/(481 - 3850 \,\theta),\quad p_2=9/11.
\end{eqnarray*}

\noindent(intervals 63--103)
\begin{eqnarray*}
p_1&=&9/11,\quad p_2=9/11.
\end{eqnarray*}
\end{small}

\end{document}